\documentclass[amsfonts]{amsart} 
\usepackage[utf8]{inputenc}
\usepackage[all]{xy}
\usepackage{amsmath,amsthm, amssymb,amsfonts}
\xyoption{arc}
\usepackage{mathrsfs}
\usepackage{verbatim}
\usepackage{graphicx}
\usepackage{hyperref}
\usepackage{dsfont}
\usepackage[T1]{fontenc}
\usepackage[hmargin=1in,vmargin=1in]{geometry}

\newtheorem{thm}{Theorem}
\newtheorem{prop}[thm]{Proposition}
\newtheorem{pdef}[thm]{Proposition-Definition}
\newtheorem{lem}[thm]{Lemma}

\newtheoremstyle{bidule}
{3pt}
{3pt}
{}
{}
{\scshape}
{.}
{.5em}
{}
\newtheorem{df}[thm]{Definition}
\theoremstyle{definition}
\newtheorem{rmk}{Remark}
\newtheorem{ex}[thm]{Example}
\newtheorem*{var}{Variant}
\newtheorem*{term}{Terminology}
\newtheorem*{note}{Note}

\newtheorem*{warn}{Warning}
\newtheorem*{nota}{Notations}
\newcommand{\E}{\mathscr{E}}
\newcommand{\C}{\mathscr{C}}
\newcommand{\Ccal}{\mathcal{C}}
\renewcommand{\O}{\mathscr{O}}
\newcommand{\D}{\mathscr{D}}
\newcommand{\A}{\mathscr{A}}
\newcommand{\B}{\mathscr{B}}

\newcommand{\M}{\mathscr{M}}

\newcommand{\X}{\mathscr{X}}
\renewcommand{\S}{\mathcal{S}}
\newcommand{\T}{\mathcal{T}}
\newcommand{\W}{\mathscr{W}}
\newcommand{\F}{\mathscr{F}}
\newcommand{\Nv}{\mathscr{N}}

\renewcommand{\le}{\mathscr{L}}
\newcommand{\G}{\mathscr{G}}
\renewcommand{\L}{\mathcal{L}}
\newcommand{\Zcal}{\mathcal{Z}}
\newcommand{\Z}{\mathbb{Z}}
\newcommand{\R}{\mathbb{R}}
\newcommand{\Cx}{\mathbb{C}}

\renewcommand{\P}{\mathscr{P}}
\newcommand{\Y}{\mathscr{Y}}
\newcommand{\bR}{\mathbf{R}}
\newcommand{\Pcal}{\mathcal{P}}
\renewcommand{\to}{\longrightarrow}
\newcommand{\ol}{\overline}
\newcommand{\ul}{\underline}
\renewcommand{\o}{\textbf{O}}
\renewcommand{\bf}{\mathbf}
\newcommand{\U}{\mathbb{U}}
\newcommand{\N}{\mathcal{N}}
\renewcommand{\Vec}{\mathbf{Vect}}
\renewcommand{\it}{\textit}
\newcommand{\Fr}{\textbf{Free}}
\newcommand{\tx}{\text}
\theoremstyle{bidule}
\newtheorem{obs}{\scshape Observations}
\renewcommand{\to}{\longrightarrow}
\DeclareMathOperator\Id{Id}
\DeclareMathOperator\Hom{Hom}
\title{Segal Enriched categories I}
\author{Hugo V. Bacard}

\begin{document}
\maketitle
\begin{abstract}
We develop a theory of enriched categories over a (higher) category $\M$ equipped with a class $\W$ of morphisms called \emph{homotopy equivalences}. We call them Segal $\M_{\W}$-categories. Our motivation was to generalize the notion of ``up-to-homotopy monoids'' in a monoidal category $\M$, introduced by Leinster. The formalism adopted generalizes the classical Segal categories and extends the theory of enriched category over a bicategory. In particular we have a linear version of Segal categories which did not exist so far. 
\end{abstract}

\setcounter{tocdepth}{1}
\tableofcontents 
\section{Introduction}
Let $\M=(\ul{\tx{M}},\otimes, I)$ be a monoidal category.
An \emph{enriched} category $\C$ over $\M$, shortly called\\ `an $\M$-category', consists roughly speaking of :
\begin{itemize}
\item objects $A,\ B, \ C,\cdots$
\item \emph{hom-objects} $\C(A,B)$  $\in Ob(\ul{\tx{M}})$,
\item a unit map $I_A :I \to \C(A,A)$ for each object $A$,
\item a composition law : $c_{ABC} : \C(B,C) \otimes \C(A,B) \to \C(A,C)$, for each triple of objects $(A,B,C)$,  
\end{itemize}
satisfying the obvious axioms, associativity and  identity, suitably adapted to the situation.\ \\ 
Taking $\M$ equal to  $\bf{(Set,\times), (Ab,\otimes_{\Z}), (Top,\times), (Cat,\times)}$,... ,  an $\M$-category  is, respectively, an ordinary\footnote{By ordinary category we mean small (or locally small) category} category, a \emph{pre-additive} category, a \emph{pre-topological} category, a $2$-category, etc. The category $\M$ is called the \emph{base} as ``base of enrichment''. 

Just like for $\bf{Set}$-categories, we have a notion of $\M$-functor, $\M$-natural transformation, etc.
The reader can find an exposition of the theory of enriched categories over a monoidal category in the book of Kelly \cite{Ke}.  
For a base $\M$, we commonly denote by $\M$-Cat the category of $\M$-categories.\ \\

Bénabou defined \emph{bicategories}, and morphisms between them (see \cite{Ben2}). He pointed out that a bicategory with one object was the same thing as a monoidal category. This gave rise to a general theory of enriched categories where the base $\M$  is a  bicategory. We refer the reader to \cite{KLSS}, \cite{Str} and references therein for \emph{enrichment} over a bicategory.\

Street noticed  in \cite{Str} that for a  set $X$, an $X$-\emph{polyad} \footnote{Bénabou called \emph{polyad} the `many objects' case of \emph{monad}. `$X$-polyad' means here ``polyad associated to $X$''} of Bénabou in a bicategory $\M$ was the same thing as a category enriched over $\M$ whose set of objects is $X$. Here an $X$-polyad means a \emph{lax} morphism of bicategories from $\ol{X}$ to  $\M$, where $\ol{X}$ is the \emph{coarse}\footnote{Some authors call it the `chaotic' or `indiscrete' category associated to $X$} category associated to $X$. 
Then given a polyad $F : \ol{X} \to \M$, if we denote by $\M_{F}^{X}$ the corresponding $\M$-category, one can interpret $F$ as \emph{the nerve} of $\M_{F}^{X}$ and identify $F$ with $\M_{F}^{X}$, like for Segal categories.\ \\

Recall that a Segal category is a \emph{simplicial object} of a cartesian monoidal category $\M$, satisfying the so called \emph{Segal conditions}. The theory of Segal categories has its roots in the paper of Segal \cite{Seg1} in which he proposed a solution of the \emph{delooping problem}. The general theory starts with the works of  Dwyer-Kan-Smith \cite{DKS} and Schw{\"a}nzl-Vogt \cite{Vogt_Sch}. The major development of Segal $n$-categories was given by Hirschowitz and Simpson \cite{HS}.\\

Hirschowitz and Simpson used the same philosophy as Tamsamani \cite{Tam} and Dunn \cite{Du}, who in turn followed the ideas of Segal \cite{Seg1}. A Segal $n$-category is defined by its nerve which is an $\M$-valued functor satisfying the suitable Segal conditions. The target category $\M$ needs to have a class of maps  called \emph{weak} or \emph{homotopy} equivalences. Moreover they required  the presence of \emph{discrete objects} in $\M$ which will play the role of `set of objects'. We can interpret their approach as an enrichment over $\M$, even though it sounds better to  say ``internal weak-category-object of $\M$''. The same approach was used by Pellissier \cite{Pel}.\ \\

Independantly Rezk \cite{Rezk_mh} followed also the ideas of Segal to define \emph{Complete Segal spaces} as weakly enriched categories over $\bf{(Top,\times)}$ and $\bf{(SSet,\times)}$. We refer the reader to the paper of Bergner \cite{Bergner_inf} for an exposition of the interactions between Segal categories, Complete Segal spaces, quasicategories, $(\infty,1)$-categories, etc.  \ \\  

To avoid the use of discrete objects, Lurie \cite{Lurie_HTT} introduced a useful tool $\Delta_X$ which is a `general copy' of the usual\footnote{By usual `$\Delta$' we mean the topological one which doesn't contain the empty set} category of simplices $\Delta$. Simpson \cite{Simpson_HTHC} used this $\Delta_X$ to define Segal categories as a ``proper'' enrichment over $\M$. Here by ``proper'' we simply mean that the set $X$ which will be the set of objects is taken `outside' $\M$.\ \\ 

In the present paper we give a definition of enrichment over a bicategory in a `Segal way'. Our definition generalizes the Hirschowitz-Simpson-Tamsamani approach as well as Lurie's. As one can expect the `strict Segal case' will give the classical enrichment, which corresponds to the polyads of Bénabou. Our construction is deeply inspired from the definition of up-to-homotopy monoids given by Leinster \cite{Lei3}. The main tools in this paper are Bénabou's bicategories and the the different type of morphisms between them. \ \\

Our motivation was to `put many objects' in the definition of Leinster. The idea consists to identify monoids and one-object enriched categories. The `many-objects' case provides, among other things, a Segal version of enriched categories over noncartesian monoidal categories, e.g, $(\bf{ChVect,\otimes_k,k})$ the category of complex of vector spaces over a field $\bf{k}$.\ \\

Beyond the fact that enrichment over bicategories generalizes the classical theory of enriched categories, it gives rise to various  points of view  in many classical situations.
Walters \cite{Wa2} showed for example that a sheaf on a Grothendieck site $\Ccal$ was the same thing as \emph{Cauchy-complete} enriched category over a bicategory $\text{Rel}(\Ccal)$ build from $\Ccal$. Later Street \cite{Str} extended this result to  
describe \emph{stacks} as enriched categories with extra properties and gave an application to nonabelian cohomology.\ \\

Both Street and Walters used the notion of \emph{bimodule} (also called \emph{distributor, profunctor or module}) between enriched categories. The notion of \emph{Cauchy completeness} introduced by Lawvere \cite{Law} plays a central role in their respective work. In fact `Cauchy completeness' is a property of \emph{representability} and is used there to have the restriction of sections and to express the \emph{descent conditions}.\\
\indent This characterization of stacks as enriched categories is close to the definition of a stack as fibered category satisfying the descent conditions. One can  obviously adapt their result with the formalism we develop here. We can consider a Segal version of their results using the notion of \emph{Segal site} of Toën-Vezzosi \cite{ToVe1}.\ \\ 

By the Giraud characterization theorem  \cite{Gir} we know that a sheaf is an object of a Grothendieck topos. Then the results of Walters and Street say that a Grothendieck (higher) topos is a equivalent to a subcategory of $\M$-Cat for a suitable base $\M$. A \emph{Segal topos} of Toën-Vezzosi should be a subcategory of the category of Segal-enriched categories over a base $\M$. Street \cite{Str2} has already provided a characterization theorem of the bicategory of stacks on a site $\Ccal$, then a \emph{bitopos}. Here again one may propose a characterization theorem for Segal topoi of Toën-Vezzosi by suitably adapting the results of Street. \ \\

More generally we can extend the ideas of Jardine \cite{Jardine_simpre}, Thomason \cite{Thomason_ktheorie} followed by, Dugger \cite{Dugger} , Hirschowitz-Simpson \cite{HS}, Morel-Voevodsky \cite{Mo_Vo_A1} , Toën-Vezzosi \cite{HAG1,HAG2} and others, who develloped a homotopy theory in situations, e.g in algebraic geometry, where the notion of homotopy was not \emph{natural}. The main ingredients in these theories are essentially the use of \emph{simplicial presheaves} with their (higher) generalizations, and the \emph{functor of points} initiated by Grothendieck.

Enriched categories appear naturally there  because, for example, the category of simplicial presheaves is a \emph{simplicially enriched} category i.e an $\bf{SSet}$-category. An interesting task will be, for example to `linearize' the work of Toën-Vezzosi and develop a Morita theory in `Segal settings'. This will be discussed in a future work. \ \\

Our goal in this paper is to present the theory of Segal enriched categories and provide situations where they appear naturally. Applications are reserved for the future.\ \\

In the remainder of this introduction we're going to give a brief outline of the content of this paper.
\subsection*{\Large{Finding a `big' $\Delta$}} \ \\
\indent We start by introducing the new tool which generalizes the monoidal category $(\Delta,+,0)$. The reason of this approach is the fact that this  category $(\Delta,+,0)$ is known to contain the universal monoid which is the object $1$. More precisely, Mac Lane \cite{Mac} showed that a monoid $V$ in a monoidal category $\M$ can be obtained as the image of $1$ by a \textbf{monoidal functor} $\Nv(V): (\Delta,+,0) \to \M$. And as mentionned previously a monoid is viewed as an $\M$-category with one object, so we can consider the functor $\Nv(V)$ as the \textbf{nerve} of the \footnote{the category is unique up to isomorphism.} corresponding category whose hom-object is $V$. \ \\

From this observation it becomes natural to find a big tool which will be used to `depict' many monoids and bimodules in $\M$ to form a general $\M$-category. This led us to the following notion (see Proposition-Definition \ref{path-bicat}).\\
\\
\textbf{Proposition-Definition} [The $2$-path-category]\\
\indent Let $\C$ be a small category.
\renewcommand\labelenumi{\roman{enumi})}
\begin{enumerate}
\item There exists a strict $2$-category $\P_{\C}$ having the following properties:
\begin{itemize}
\item[$\ast$]the objects of $\P_{\C}$ are the objects of $\C$,
\item[$\ast$]for every pair $(A,B)$ of objects, a $1$-morphism from $A$ to $B$ is of the form $[n,s]$, where $s$ is a finite chain of composable morphisms, from $A$ to $B$, and $n$ is \textbf{the length of $s$}. 
\item[$\ast$]a $2$-morphism from $[n,s]$ to $[m,t]$ is given by compositions of composable morphisms or adding identities. It follows that $\P_{\C}(A,B)$ is a \textbf{posetal category}. 
\item[$\ast$] the composition in $\P_{\C}$ is given by the \textbf{concatenation of chains}.
\item[$\ast$] When $A=B$ there is a unique $1$-morphism of lenght $0$, $[0,A]$ which is identified with $A$. Moreover $[0,A]$ is the identity morphism of $A$.   
\end{itemize}
\item if $\C \cong \bf{1}$, say $ob(\C)=\{\o\}$ and $\C(\o,\o)=\{\Id_{\o}\}$, we have a \underline{monoidal} isomorphism  :
$$(\P_{\C}(\o,\o),c(\o,\o,\o),[0,\o]) \xrightarrow{\sim} (\Delta,+,0) $$
where $c(\o,\o,\o)$ is the composition functor
\item the operation $\C \mapsto \P_{\C}$ is functorial in $\C$.
\end{enumerate}
\ \\
Similar constructions have been considered by Dawson, Paré and Pronk for double categories (see \cite{DPP-path}). One can compare the \emph{Example 1.2 and Remark 1.3} of their paper with the fact that here we have: $\P_{\bf{1}}$ `is' $(\Delta,+,0)$.\\ 
\ \\
\indent As mentionned above the idea of enrichment will be to consider special types of morphisms from $\P_{\C}$ to other bicategories. We will see that when $\C$ is $\ol{X}$, $\P_{\ol{X}}$ will replace Lurie's $\Delta_{X}$ and will be used to define Segal enriched categories. This will generalize the definition of up-to-homotopy monoid in the sens of Leinster which may be called up-to-homotopy \emph{monad} in the langage of bicategories.\ \\

One of the good properties of $\P_{\C}$ is the fact that any functor of source $\C$ can be lifted to a \textbf{free} $2$-functor of source $\P_{\C}$ (see Observations \ref{obsPC1}). This process takes classical $1$-functors to  enrichment situations and gives them new interpretations. 

The fact that $\C$ is an arbitrary small category allows us to consider geometric situations when $\C$ is a Grothendieck site and in this way we can `transport' geometry in enriched category context.\\


\subsection*{\Large{The environment}} \ \\
\indent Before giving the definition of enrichment, we describe the type of category $\M$ which will contain the hom-objects $\C(A,B)$ (see Definition \ref{base_enrich}). \\

\indent We will work with a bicategory $\M$ equipped with a class $\W$ of $2$-cells satisfying the following properties.
\begin{enumerate}
\item Every invertible $2$-cell of $\M$ is in $\W$, in particular $2$-identities are in $\W$,
\item $\W$ is stable by horizontal composition,
\item $\W$ has the vertical `$3$ out of $2$' property. 
\end{enumerate}
Such a pair $(\M,\W)$ will be called \emph{base} as `base of enrichment'.
When $\M$ has one object, therefore a monoidal category, we get the same environment given by Leinster \cite{Lei3}.\\

Since we work with bicategories, $\M$ can also be :
\begin{enumerate}
\item[$\ast$] any $1$-category viewed as a bicategory with all the $2$-morphisms being identities,
\item[$\ast$] the `$2$-level part' of an $\infty$-category.
\end{enumerate}

\begin{note}
To define Segal enriched categories, $\W$ will be a class of $2$-morphisms called \textbf{homotopy $2$-equivalences}. In this case, following the terminology of Dwyer, Hirschhorn, Kan and Smith \cite{DHKS} one may call $\M$ together with $\W$ `a homotopical bicategory'.
\end{note} 
\vspace*{0.1cm}
\subsection*{\Large{Relative enrichment}}\ \\

With the previous materials we give the definition of \textbf{relative enrichment} in term of \textbf{path-objects} (\textsl{Definition} \ref{p-object}). One can compare the following definition with \textsl{Definition}  \ref{h-mon}.   \ \\
\\
\textbf{Definition.}{[Path-object]} \\
\indent Let $(\M, \W)$ be a base of enrichment.  A \emph{path-object} of $(\M,\W)$ is a couple $(\C,F)$, where $\C$ is a small category and $F=(F,\varphi)$ a \textbf{colax morphism} of Bénabou:
$$F : \P_{\C} \longrightarrow \M $$
such that for any objects $A$, $B$, $C$ of $\C$ and  any pair $(t,s)$ in $\P_{\C}(B,C) \times \P_{\C}(A,B)$, all  the $2$-cells 
$$F_{AC}(t \otimes s) \xrightarrow{\varphi(A,B,C)(t,s)} F_{BC}(t) \otimes F_{AB}(s) $$ 
$$ F_{AA}([0,A]) \xrightarrow{\varphi_{A}} I'_{FA}$$
\vspace{0.5cm}
\textbf{are in} $\W$. Such a colax morphism will be called a \textbf{$\W$-colax morphism} and $\varphi(A,B,C)$ will be called \textbf{`colaxity maps'}.\\
\begin{itemize}
\item When $\W$ is a class of homotopy $2$-equivalences, then $(\C,F)$ will be called \textbf{a Segal path-object} of $\M$ and 
$F : \P_{\C} \longrightarrow \M $ will give a relative enrichment of $\C$ over $(\M, \W)$. 
\item We will say for short that $(\C,F)$ is a \textbf{`$\C$-point'}  or a \textbf{`$\C$-module'} of $\M$. In \cite{SEC2} a duality theory of enrichment is developped and we will prefer the terminology $\C$-module.
\item When $\C= \ol{X}$ a Segal $\ol{X}$-point of  $(\M, \W)$ is called a Segal  $\M_{\W}$-category.
\end{itemize}
\ \\
\indent The reason we consider \textbf{colax} morphisms, is the fact that `colax' is the appropriate replacement of \textbf{simplicial} (see \textsl{Proposition} \ref{leiprop}) when working with general monoidal categories, in particular for noncartesian monoidal ones. In fact when $\M$ is monoidal for the cartesian product, the \emph{colaxity maps} will give \emph{all} the \emph{face maps} and we can consider a path-object of $(\M,\W)$ as a `super-simplicial object' of $(\M,\W)$ `colored by' $\C$. In the Segal case we may consider $F$ as a `$\C$-homotopy coherent nerve'
\ \\

But there are also other interpretations that arise when we consider special bases $(\M,\W)$. Some times in the Segal case we can consider $F$ as a \emph{homotopic $\M$-representation} of $\C$ (see \ref{nabcoh}). 

One of the advantages of having enrichment as morphism is the fact that classical operations such as \emph{base change} will follow immediately. In addition to that, we can use the notions of \textbf{transformations} and \textbf{modifications} to have a first categorical structure of the `moduli space' of relative enrichments of $\C$ over $\M$. This is discussed in \textsl{sections} \ref{morph} and \ref{b-change}.
\subsection*{\Large{Examples}}\ \\

In \textsl{section} \ref{examples} we show that the formalism we've adopted covers the following situations.\ \\
 
\paragraph*{Category theory}
\begin{itemize}
\item Up-to-homotopy monoid in the sens of Leinster  (\textsl{Proposition} \ref{hotmon}). 
\item Simplicial object (\textsl{Proposition} \ref{hotmon}).
\item Classical enriched categories (\textsl{Proposition} \ref{c-enr}) and the general case of enrichment over a bicategory (\textsl{Proposition} \ref{bicat-enr}).
\item Segal categories in the sens of Hirschowitz-Simpson (\textsl{Proposition} \ref{seg-ncat}).
\item Linear Segal categories are defined in Definition \ref{DG-Seg}
\end{itemize}
\paragraph*{Nonabelian cohomology}
 \begin{itemize}
 \item In \ref{g-cat} we remarked that for a group $G$ in $(\bf{Set}, \times)$, a $G$-torsors e.g $EG$, is the same thing as a ``full'' $G$-category. The cocyclicity property of torsors reflect a `degenerated' composition i.e the composition maps $c_{ABC}$ are identities.  We recover the classification role of $BG$ because to define a $G$-category we take a path-object of $BG$.
 
This remark can be extended to the general case of a group-object using the functor of points. The Segal version of this situation will lead to deformation theory and derived geometry. \ \\

\item For a nonempty set $X$, the coarse category $\ol{X}$ is the `EG' of some group G (see Remark \ref{rmk-coar}). And as we shall see we will take $\ol{X}$-point of $(\M,\W)$ to define enriched categories. When $\M$ is a $1$-category e.g $\Vec$ the category of vector space, an EG-point of $\Vec$  will give in some case a representation of G (all elements of G, which correspond to the objects of EG, are sent to the same object of $\Vec$). \ \\

 \item In Example \ref{exp-bc} we've considered the exponential exact sequence as a base change. Moreover we remarked that given a function $f: X \to (K,d)$, from some `space' $X$ to  a metric  (or normed) space $K$  e.g  $\R, \Cx$, a DVR \footnote{DVR :Discrete Valuation Ring} $L$, the pullback of the metric by $f$ gives a metric on $X$. And if $X$ has a topology then the metric space considered as an $\R$-category will have an \emph{atlas} of $\R$-categories. This is an example of iterative process of enrichment \emph{`à la'} Simpson-Tamsamani, because since Lawvere \cite{Law} it's well known that metric spaces are enriched categories over $(\ol{\R}_+,+,0, \geq)$. \ \\
 \item In \ref{p-transp} we give an example of $1$-functor considered as a (free) path-object. We want to consider a \emph{parallel transport functor} as a path-object. In the Segal case we will have \emph{homotopic holonomy}.\ \\
 \item Finally we've introduced some material for the future with the notion of \emph{quasi-presheaf} (see \ref{q-prsh}).  We define a \emph{quasi-presheaf} on $\C$ to be a Segal $\C^{op}$-point of $(\M,\W)$. This is not simply a `generalization-nonsense' of classical presheaves. We want to consider, for example, the Grothendieck anti-equivalence between affine schemes and commutative rings as a `co-enrichment' having a \emph{reconstruction property}, then a \emph{good} enrichment (see Example \ref{aff-alg}). A detailed account will appear in \cite{SEC2}.\ \\
 
Another example is to consider any cohomology theory on $\C$ as a family of `free' $\C^{op}$-points i.e relative enrichments of $\C^{op}$. We hope that using the \emph{machine} of `Segal enrichment': base changes of $\C^{op}$-points, enriched Kan extension, Segal categories etc, together with model categories, we can understand some facets of \emph{motivic cohomology}.

These considerations will require an appropriate \emph{descent theory} of relative enrichment which will be discussed in \cite{SEC2}. 
 \end{itemize}

\subsection*{\Large{Morphisms, Bimodules and Reduction}}\ \\

In section \ref{morph} we've revisited some classical notions adapted to our formalism. We've tried as much as possible to express these notions in term of morphisms of bicategories. The idea is to have everything `at once' using path-objects. \\
\begin{enumerate}
\item Given two path-objects $F:\P_{\C} \to \M$ and $G: \P_{\D} \to \M$ , we define first an $\M$-premorphism to be a couple $\Sigma=(\Sigma,\sigma)$ consisting of a functor $\Sigma : \C \to \D$ together with a transformation of (colax) morphisms of bicategories
$\sigma :  F  \longrightarrow G \circ \P_{\Sigma}$ (Definition \ref{premor}). An $\M$-morphism is a special type of an $\M$-premorphism.\  \\
\item We define bimodules (also called ``distributors'', ``profunctors'' or ``modules'') in term of path-object (Definition \ref{bimod}).\ \\
\item Finally in Proposition \ref{reduc}, we introduce a bicategory $\W^{-1}\M$ which is rougly speaking the `secondary' Gabriel-Zisman localization of a base $(\M,\W)$ with respect to $\W$. With this bicategory $\W^{-1}\M$ we can \emph{reduce} any Segal point to it's homotopic part.
\end{enumerate}
\vspace*{1cm}

\section*{Acknowledgments}
I would like to warmly thank and express my gratitude to my supervisor, Carlos Simpson, whose encouragement, guidance and support have enabled me to develop an understanding of the subject. I would like to extend my thanks to Julia Bergner and Tom Leinster who have suggested the question of finding a `Segal-like definition' of enriched categories. 

Many thanks to the staff of the Laboratoire J-A Dieudonné who provided me an excellent working environment.
\section*{Yoga of enrichment} \label{yoga}
\begin{quotation}
`` In mathematics, there are not only theorems. There are, what we call, `philosophies' or `yogas', which remain vague. Sometimes we can guess the flavor of what should be true but cannot make a precise statement. When I want to understand a problem, I first need to have a panorama of what is around it. A philosophy creates a panorama where you can put things in place and understand that if you do something here, you can make progress somewhere else. That is how things begin to fit together. '' \\ 
\begin{flushright}
Pierre Deligne, in \emph{Mathematicians}, Mariana Cook, PUP, 2009, p156.
\end{flushright}
\end{quotation}
\vspace*{1cm}
\begin{center}
\includegraphics[width=11cm,height=8cm]{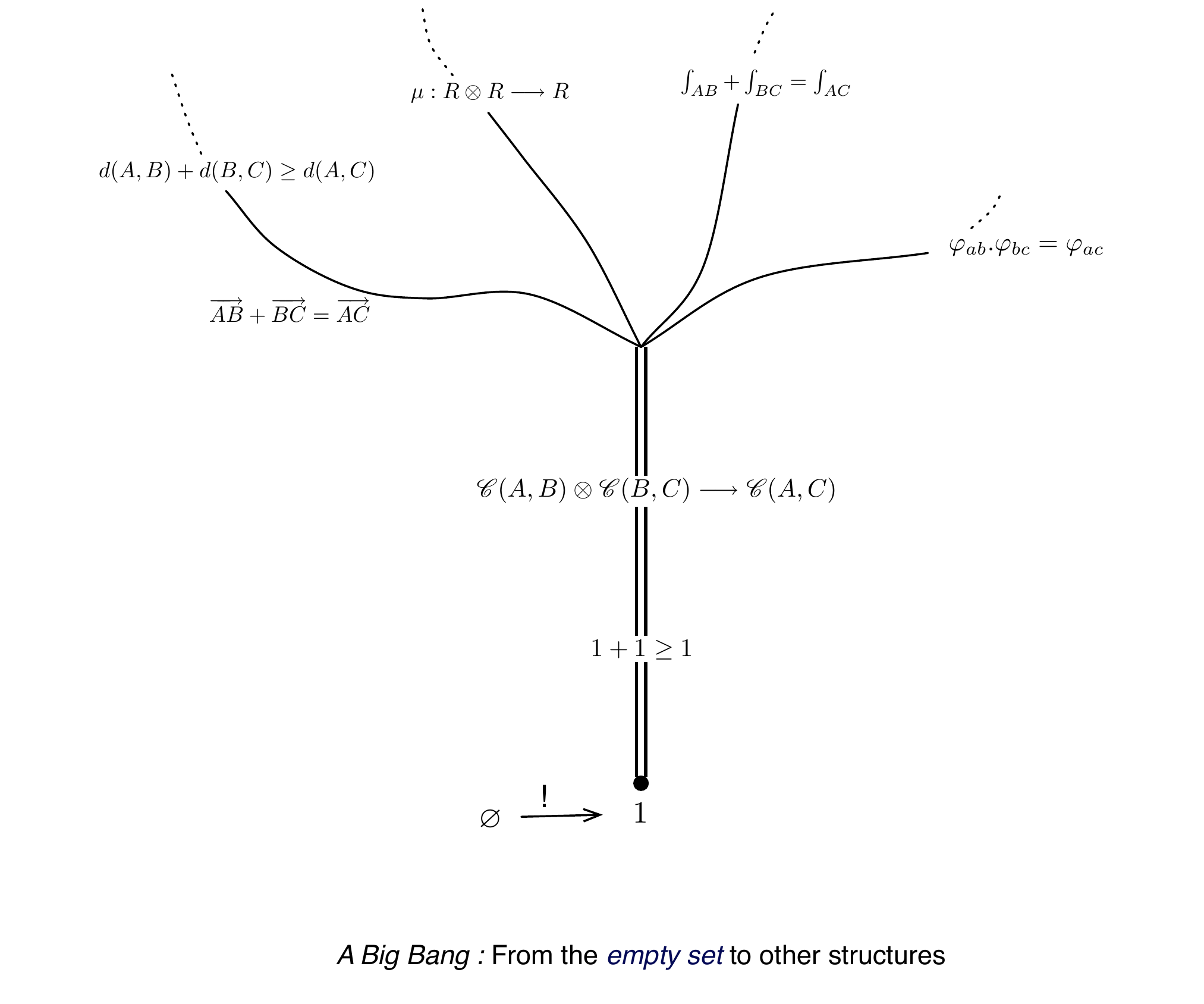}
\end{center}
\vspace*{1cm}
\ \
\indent In the following we present some facets of enrichment and give some interpretations \footnote{or philosophies}  around it. We hope that it will help a reader who is nonfamiliar to Segal categories to have a panoramic understanding of the subject.

\subsection*{The general picture} \ \\

\indent In a classical category $\C$, we have :
\begin{itemize}
\item compositions  $c_{ABC} : \C(B,C) \otimes \C(A,B) \to \C(A,C)$ : \indent thought as a \emph{partial multiplications}\\ 
\item an identity map $I_A :I \to \C(A,A)$ : \indent think $[\C(A,A), I_A]$ as a \emph{pointed space with multiplication} e.g $\pi_1(x,X)$,  $\Omega_{x}X$, etc.\ \\
\end{itemize}

In a Segal category \textbf{there is no prescription}, in general, of the previous data but we have the following diagrams. 
\renewcommand\labelenumi{\Roman{enumi})}
\begin{enumerate}
\item
\[
\xy
(-10,0)*+{\C(B,C) \otimes \C(A,B)}="X";
(50,0)*+{\C(A,C)}="Y";
(22,22)*+{\C(A,B,C)}="Z";
{\ar@{->}_{\tx{weak equiv.} }"Z"; "X"};
{\ar@{->}^{\tx{canonical}}"Z"; "Y"};
\endxy
\]

\item 
\[
\xy
(-10,0)*+{I}="X";
(50,0)*+{\C(A,A)}="Y";
(22,22)*+{\C_{[0,A]}}="Z";
{\ar@{->}_{\tx{weak equiv.}}"Z"; "X"};
{\ar@{->}^{\tx{canonical}}"Z"; "Y"};
\endxy
\]
\end{enumerate}

The two type of maps 
$$\C(A,B,C) \to \C(B,C) \otimes \C(A,B) \ \ \tx{and} \ \  \C_{[0,A]} \to I$$
are called `Segal maps' and they are required to be \emph{weak equivalences}.

The idea is that when these maps are isomorphisms (strong equivalences) then using their respective inverse we can run these diagrams from the left to the right and we will have the data of a classical category. \\

But when the Segal maps are not isomorphisms but only weak equivalences then we can think that each weak inverve of the previous maps will give a `quasi-composition' and a `quasi-identity map'. It turns out that Segal categories are more general than classical categories and appear to be a good tool for homotopy theory purposes. 
\begin{note}
In this paper the Segal maps will be the `colaxity maps'. 
\end{note}

\subsection*{Why \emph{relative} enrichment ?}\ \\

For a given small category $\C$ and a bicategory $\M$ we define a relative enrichment of $\C$ over $\M$ to be a morphism of bicategories $F :\P_{\C} \to \M$ satisfying some extra conditions which can be interpreted as `generalized Segal conditions'.\ \\

To understand the meaning of `relative' it suffices to consider the trivial case where $\M$ is a $1$-category viewed as a bicategory with identity $2$-morphisms. In this case the morphism $F$ is determined by  a $1$-functor  $F_{|\C}: \C \to \M$ (see Observations \ref{obsPC1}).\ \\

And the idea is to observe that given any functor $G: \C \to \M$ then we can form the category $G[\C]$ described as follows.

\begin{itemize}
\item $Ob(G[\C])= Ob(\C)$\\
\item For each pair of objects $(A,B)$ we take the morphism to be the image \footnote{As $G$ may not be faithful we need to take the \emph{reduced} image to have a set} of the function $$G_{AB}:\C(A,B) \to \M(GA,GB)$$
\item The composition is defined in the obvious way. 
\end{itemize} 

In this situation we will consider $G[\C]$ as a relative enrichment of $\C$ over $\M$. One can interpret $G[\C]$ as \emph{a copy of $\C$ of type $\M$}.\ \\

As usual we have the following philosophical questions.
\begin{itemize}
\item What is the `best copy' of $\C$ of type $\M$ ? 
\item Does such a `motivic copy' of $\C$ exist for a given $\M$ ?
\item Which $\M$ shall we consider to have many informations about $\C$ ?\\
\end{itemize} 
The \emph{machine} of enriched categories allows us to do base changes i.e move $\M$, and we hope that the ideas of \emph{Segal-like enrichment} can guide us, using homotopy theory, to find an aswer of those questions.\ \\

In the previous example we can see that we have the usual factorization of the functor $G$ as 
$$ \C \xrightarrow{j} G[\C] \xrightarrow{i} \M$$ where $j$ is full and $i$ is faithfull.

If the functor is an equivalence then we may say that $G[\C]$ is a `good copy' of $\C$ of type $\M$ and if $G$ reflects isomorphisms then $G[\C]$ will provide a good copy of some subcategory of the \emph{interior of $\C$}, etc.\ \\

We can also form a category whose set of objects is the image of the function $G: Ob(\C) \to Ob(\M)$ . The morphism are those `colored' or coming from $\C$. In this way we will form a category which \emph{lives} in $\M$. 

With this point of view, we can also consider that enrichment over $\M$ is a process which \emph{enlarges} $\M$. In fact it's well known that some properties of $\M$ are \emph{transfered} to $\M$-Cat and $\M$-Dist.\ \\

The recent work of Lurie \cite{Lurie_tqft} on cobordism hypothesis (framed version) says roughly speaking that a \emph{copy of  $\bf{Bord}_{n}^{\tx{\emph{fr}}}$} which respects the monoidal structure and the symmetry i.e a symmetric monoidal functor, is determined by the copy of the point which must be a \emph{fully dualizable object}. This reflect the fact that $\bf{Bord}_{n}^{\tx{fr}}$ is in some sense \emph{built} from  the point using bordisms and disjoint unions.\ \\  
 
\begin{rmk}
To have the classical theory of enriched categories we will consider the case where $\C$ is of the form $\ol{X}$ (see \ref{coarse}) for a nonempty set $X$ and $\M$ a bicategory with one object, hence a monoidal category.
\end{rmk}

\subsection*{A `Big Bang'}\ \\

\begin{quotation}
`` Will mathematics merely become more sophisticated and specialized, or
will we find ways to drastically simplify and unify the subject? Will we only
build on existing foundations, or will we also reexamine basic concepts and
seek new starting-points? Surely there is no shortage of complicated and
interesting things to do in mathematics. But it makes sense to spend at least
a little time going back and thinking about \emph{simple} things.''\\
\begin{flushright}
John C. Baez, James Dolan, \emph{From Finite Sets to Feynman Diagrams} \cite{Baez-Dolan_Fey}  p2.
\end{flushright}
\end{quotation}
\vspace*{1cm}

If we look closely the composition `$\C(A,B) \otimes \C(B,C) \to \C(A,C)$' in any category, we can see the similarity with the other classical formulas such as:\\

\begin{itemize}
\item $\overrightarrow{AB} + \overrightarrow{BC} = \overrightarrow{AC}$ : basic geometry \\
\item $d(A,B) + d(B,C) \geq d(A,C)$ : triangle inequality\\
\item $\varphi_{AB} \cdot \varphi_{BC} = \varphi_{AC}$ : cocyclicity of transition functions for a vector bundle, etc.
\end{itemize}
\vspace*{0.2cm}
In fact all of these formulas can be described in term of enriched categories and base changes (see Example \ref{exp-bc}). For example Lawvere \cite{Law} remarked that the \emph{triangle inequality} is the composition in a metric space when considered as an enriched category over $(\ol{\R}_+,+,0, \geq)$.\ \\ 

We've used the terminology `Big bang' because  many structures are encoded in this way by enriched categories. But enriched categories are defined using a big version of the category $\Delta$. And $\Delta$ is itself built from $1$, which in turn can be taking to be $\{ \varnothing \}$.\ \\ 

It appears that almost \emph{everything} comes from the \emph{empty} ... 
\section{Path-Objects in Bicategories}
\subsection{The $2$-path-category}\ \\
\indent We follow the notations of Leinster \cite{Lei3} and denote here by  $\Delta$ the ``augmented'' category of \underline{all} finite totally ordered sets, including the empty set. This $\Delta$ is different from the ``topological'' one, which does not contain the empty set and is commonly used to define simplicial objects. We will note the topological `$\Delta$' by  $\Delta^{+}$ or sometimes $\Delta[-0]$ to stress the fact that the empty set has been removed.\ \\

Recall that the objects of $\Delta$ are ordinal numbers $n=\{0,...,n-1\}$ and the arrows are nondecreasing functions $f:n\longrightarrow m$. $\Delta$ is a monoidal category for the \emph{ordinal addition}, has an initial object $0$ and a terminal object $1$. The object $1$ is a ``universal'' monoid  in the sense that any monoid in a monoidal category $\M$ is the image of $1$ by a \emph{monoidal functor} from $\Delta$ to $\M$. The reader can find this result and a complete description of $\Delta$ in \cite{Mac}.\ \\

\begin{warn}
The category $\Delta$ we consider here corresponds to the category $\Delta^{+}$ used by Deligne in \cite{Hodge_3}. And he denoted by $\Delta$ our category $\Delta[-0]$. 
\end{warn}

\begin{pdef}{\emph{[$2$-Path-category]}}\label{path-bicat}
Let $\C$ be a small category.
\renewcommand\labelenumi{\roman{enumi})}
\begin{enumerate}
\item There exists a strict $2$-category $\P_{\C}$ having the following properties:
\begin{itemize}
\item[$\ast$]the objects of $\P_{\C}$ are the objects of $\C$,
\item[$\ast$]for every pair $(A,B)$ of objects, $\P_{\C}(A,B)$ is posetal and is a category over $\Delta$ i.e we have a functor called \textbf{length} 
$$ \le_{AB} : \P_{\C}(A,B) \to \Delta$$  
\item[$\ast$]$0$ is in the image of $\le_{AB}$ if and only if $A=B$. $\le_{AA}$ becomes a \underline{monoidal} functor with the composition.\\
\end{itemize}
\item if $\C \cong \bf{1}$, say $ob(\C)=\{\o\}$ and $\C(\o,\o)=\{\Id_{\o}\}$, we have \underline{monoidal} isomorphism:
$$\P_{\C}(\o,\o) \xrightarrow{\sim} \Delta $$
\item the operation $\C \mapsto \P_{\C}$ is functorial in $\C$:
\end{enumerate}
\[
\xy
(0,8)*+{\P_{[-]}: \tx{Cat}_{\leq 1}}="X";
(30,8)*+{\tx{Bicat}}="Y";
(3,0)*+{\C \xrightarrow{F} \D}="E";
(35,0)*++{\P_{\C} \xrightarrow{\P_{F}} \P_{\D}}="W";
{\ar@{->}"X";"Y"};
{\ar@{|->}"E";"W"};
\endxy
\]
where $\tx{Cat}_{\leq 1}$ and Bicat are respectively the $1$-category of small categories and the category of bicategories.
\end{pdef}
\begin{proof}
The construction of $\P_{\C}$ is exposed in Appendix B but we give a brief idea hereafter.\\ 

\textbf{Step1}: Using the philosophy of the Bar construction with respect to the composition in $\C$, we build the following diagrams

\[
\xy
(-75,0)*++{\C(A,B)}="X";
(-40,0)*++{\coprod \C(A,A_{1})\times \C(A_{1},B)}="Y";
(16,0)*++{\coprod \C(A,A_{1}) \times \C(A_{1},A_{2})\times \C(A_{2},B)}="Z";
(47,0)*++{\cdots};
{\ar@{->}"Y";"X"};
{\ar@{.>}"Y";"Z"};
{\ar@<-1.0ex>@{.>}"X";"Y"};
{\ar@<1.0ex>@{.>}"X";"Y"};
{\ar@<-1.0ex>"Z";"Y"};
{\ar@<1.0ex>"Z";"Y"};
{\ar@<-1.8ex>@{.>}"Y";"Z"};
{\ar@<1.8ex>@{.>}"Y";"Z"};
\endxy
\]

\[
\xy
(-65,0)*++{\C(A,A)}="X";
(-30,0)*++{\coprod \C(A,A_{1})\times \C(A_{1},A)}="Y";
(26,0)*++{\coprod \C(A,A_{1}) \times \C(A_{1},A_{2})\times \C(A_{2},A)}="Z";
(-93,0)*+{\{A\} \cong 1}="O";
(57,0)*++{\cdots};
{\ar@{->}"Y";"X"};
{\ar@{.>}"Y";"Z"};
{\ar@{->}^{~~~1_{A}}"O";"X"};
{\ar@<-1.0ex>@{.>}"X";"Y"};
{\ar@<1.0ex>@{.>}"X";"Y"};
{\ar@<-1.0ex>"Z";"Y"};
{\ar@<1.0ex>"Z";"Y"};
{\ar@<-1.8ex>@{.>}"Y";"Z"};
{\ar@<1.8ex>@{.>}"Y";"Z"};
\endxy
\]
\ \\
These diagrams correspond to cosimplicial sets, that is functors from either $\Delta$ or $\Delta^{+}$ to the category of sets.\ \\

\textbf{Step 2}: We take $\P_{\C}(A,B)$ to be the category of elements of the previous functor using the Grothendieck construction. One can observe that $\P_{\C}(A,B)$ `is' the total complex in the sens of Bousfield-Kan \cite{BK-hlim}  of the corresponding cosimplicial set (see also \cite{PC-hcc} for a description).   \ \\

\textbf{Step 3}: We define the composition to be the concatenation of chains. 
\end{proof}

\begin{rmk}\ \
\begin{itemize}
\item The `path-functor' as presented above doesn't extend immediately to a $2$-functor because natural transformations in Cat are not sent to transformation in Bicat. This is due to the fact that each $\P_{\C}(A,B)$ is posetal.

\item We can fix the problem either by using co-spans in $\P_{\C}(A,B)$  or by localizing each $\P_{\C}(A,B)$ with respect to the class of maps which correspond to compositions in $\C$, then `reversing the composition'. 

\item Another solution could be to work in the area of Leinster's fc-multicategories instead of staying in Bicat, but we won't do here. In fact $\P_{\C}$ carries a good enough combinatoric for our first purpose which is to have a Segal version of enriched categories.
\end{itemize}
\end{rmk}

\begin{obs}\label{obsPC1}\ \\
\indent For a small category $\C$, the following properties follow directly from the construction of $\P_{\C}$.
\renewcommand\labelenumi{\roman{enumi})}

\begin{enumerate}
\item Since $\bf{1}$ is terminal in Cat, we have by funtoriality a homomorphism (strict $2$-funtor):
$\P_{\C} \rightarrow \P_{\bf{1}}$. We may call it the `skeleton-morphism' and we will \textbf{view $\P_{\C}$ to be over $\P_{\bf{1}}$}. 
\item We have $(\P_{\C})^{op} \cong \P_{\C^{op}}$, where $(\P_{\C})^{op}$ is the opposite $2$-category of $\P_{\C}$.
\item We have a functor : $\C \xrightarrow{i} \P_{\C}$ which is the identity on objects and sends every non-identity arrow $f$ of $\C$ to the chain $[1,f]$ of $\P_{\C}$. Each identity $\Id_A$ is sent to $[0,A]$.
\item We have also a functor :$\P_{\C} \xrightarrow{\tx{comp}} \C$ which is also the identity on objects and sends a chain $[n,s]$ to the composite of the arrows which form $s$. Each $[0,A]$ is sent obviously to $\Id_A$. This functor `kills' the $2$-cells of $\P_{\C}$. 
\item It follows that the identity functor of $\C$ factors as : $\Id_{\C}= \tx{comp} \circ i$. Then for any bicategory $\M$, every functor of $1$-categories $F:\C \to \M_{\leq 1}$ yields a \underline{strict} homomorphism of bicategories $\tx{\Fr}({F}): \P_{\C} \to \M_{\leq 1} \hookrightarrow \M$. Here $\M_{\leq 1}$ represents the underlying $1$-category of $\M$. 
\item Using the functor $i$, we see that any homomorphism $\widehat{F}: \P_{\C} \to \M$ induces a functor between $1$-categories $\widetilde{F}: \C \to \M_{\leq 1}$. \textbf{From these observations we see that we may prefer working with $\P_{\C}$ rather than $\C$ even in the classical situations}. 
\end{enumerate}
\end{obs}

\subsection{Basic propreties}\ \\
In the following we give some basic properties of the path-functor  \\

It is clear that $\P_{[-]}$ preserves equivalence and is an isomorphism reflecting in an obvious manner. It's even an `equivalence' reflecting.\\

From the definition one hase immediately that $\P_{\C \coprod \D} \cong \P_{\C} \coprod \P_{\D}$. For the product we need to be careful.\ \\ 

Let $\C$ and $\D$ be two small categories and $\C \times \D$ their cartesian product. The projections from $\C \times \D$ to each factor induce two maps 
$$\P_{\C \times \D} \to \P_{\C}, \ \ \ \ \\ \P_{\C \times \D} \to \P_{\D} $$ 
which in turn give a map
$$\P_{\C \times \D} \rightarrow (\P_{\C} \times \P_{\D})$$ 
by universal property of the cartesian product.

This map has no section, which means that it cannot be a biequivalence. We can see it from the fact that we have a canonical map $\P_{\C \times \D} \rightarrow \P_{\bf{1}}$ while such map doesn't exist with $\P_{\C} \times \P_{\D}$. This is related to the fact that $(\Delta, +,0)$, which is $\P_{\bf{1}}$, \textbf{is not symmetric monoidal}. \ \\ 

But everything is not lost since we have.
\begin{prop}
Let $\C$ and $\D$ be two small categories and $\P_{\C} \rightarrow \P_{\bf{1}}$, $\P_{\D} \rightarrow \P_{\bf{1}}$ the corresponding $2$-path-categories. Then we have an isomorphism of $2$-categories
$$ \P_{\C \times \D} \xrightarrow{\sim} (\P_{\C} \times_{\P_{\bf{1}}} \P_{\D}).$$
\end{prop}

\begin{proof}[\scshape{Sketch of proof}]
It suffices to write the definition of $\P_{\C \times \D}$. A chain $[n,s]$ in $\P_{\C \times \D}$ is by definition the same thing as a couple of chains $([n,s_{\C}],[n,s_{\D}])$. And a morphism of chains is $\P_{\C \times \D}$ is by definition a morphism of $\Delta$ which is `simultaneously' the same in both $\P_{\C}$ and  $\P_{\D}$ which means that it's a morphism of the fiber product $\P_{\C} \times_{\P_{\bf{1}}} \P_{\D}$.\ \\

Here $\P_{\C} \times_{\P_{\bf{1}}} \P_{\D}$ is given by :\ \\

\begin{itemize}
\item Objects : $Ob(\C) \times Ob(\D)$ 
\item Morphisms : Consider two couples $(A,X)$, $(B,Y)$, with $A,B$ objects of $\C$ and $X,Y$ objects of $\D$. From the length functors :  
$$ \le_{AB} : \P_{\C}(A,B) \to \Delta, \ \ \ \ \le_{XY} : \P_{\D}(X,Y) \to \Delta$$
we define $$(\P_{\C} \times_{\P_{\bf{1}}} \P_{\D})[(A,X),(B,Y)]:= \P_{\C}(A,B) \times_{\Delta} \P_{\D}(X,Y).$$ 
\ \\ 
\item The composition is given by the concatenation of chains \emph{factor-wise}. 
\end{itemize}
\end{proof}

\begin{rmk}\ \
\renewcommand\labelenumi{\roman{enumi})}
\begin{enumerate}
\item If we use cospans in each $\P_{\C}(A,B)$ and extend $\P_{[-]}$ to a $2$-functor, then one can compute the general limits and colimits with respect to $\P_{[-]}$, but we won't do it here.
\item  If we apply the $2$-functor to a monoidal category $\D$, then $\P_{\D}$ will be a monoidal $2$-category with a suitable tensor product. 
\end{enumerate}
\end{rmk}
\subsection{Base of enrichment} \ \\
\indent Let $\M$ be a bicategory and $\W$ a class of $2$-cells of $\M$.
\begin{df}\label{base_enrich}
The pair $(\M, \W)$ is said to be a \underline{base} of \underline{enrichment} if $\W$ has the following properties:
\renewcommand\labelenumi{\roman{enumi})}
\begin{enumerate}
\item Every invertible $2$-cell of $\M$ is in $\W$, in particular $2$-identities are in $\W$,
\item $\W$ has the vertical ‘$3$ out of $2$' property, that is :
\[
\xy
(0,0)*+{\xy
(0,0)*+{U}="A";
(15,0)*+{V}="B";
{\ar@/^1.9pc/^{~~~f}"A";"B"};
{\ar@/_1.9pc/_{~~~h}"A";"B"};
{\ar@{->}_{}"A";"B"};
{\ar@{=>}_{\alpha}(7.5,7);(7.5,1)};
{\ar@{=>}_{\beta}(7.5,-1);(7.5,-7)};
\endxy}="X";
(30,0)*+{\xy
(0,0)*+{U}="A";
(20,0)*+{V}="B";
{\ar@/^1.4pc/^{~~~f}"A";"B"};
{\ar@/_1.4pc/_{~~~h}"A";"B"};
{\ar@{=>}_{\beta \star \alpha}(10,4);(10,-4)};
\endxy}="Y";
{\ar@{~>}"X";"Y"};
\endxy
\]
if $2$ of  $\alpha$, $\beta$, $\beta \star \alpha$ are in $\W$ then  so is the third,
\item $\W$ is stable by \emph{horizontal} compositions, that is : 
\[
\xy
(-10,0)*+{\xy
(0,0)*+{W}="C";
(15,0)*+{V}="B";
(30,0)*+{U}="A";
{\ar@/_1.3pc/_{f}"A";"B"};
{\ar@/^1.3pc/^{f'}"A";"B"};
{\ar@/_1.3pc/_{g}"B";"C"};
{\ar@/^1.3pc/^{g'}"B";"C"};
{\ar@{=>}_{\beta}(7.5,4);(7.5,-4)};
{\ar@{=>}_{\alpha}(22.5,4);(22.5,-4)};
\endxy}="X";
(30,0)*+{\xy
(0,0)*+{W}="C";
(20,0)*+{U}="A";
{\ar@/_1.3pc/_{g \otimes f}"A";"C"};
{\ar@/^1.3pc/^{g' \otimes f'}"A";"C"};
{\ar@{=>}_{\beta \otimes \alpha}(10,4);(10,-4)};
\endxy}="Y";
{\ar@{~>}"X";"Y"}:
\endxy 
\]
if $\alpha$ and $\beta$ are both in $\W$ then so is $\beta \otimes \alpha$. 
\end{enumerate}
\end{df}
\begin{obs} \ \\
\indent This definition is simply a generalization of the environment required by Leinster to define the notion of \emph{up-to-homotopy monoid} in \cite{Lei3}, when $\M$ is a monoidal category, hence a bicategory with one object.
\end{obs}
\begin{rmk}
The reader may observe that for any bicategory $\M$,  the class $\W=2$-Iso consisting of all invertible $2$-cells of $\M$ satisfies the previous properties. In this way we can say that the pair $(\M, 2\tx{-Iso})$ is the smallest base of enrichment since by definition every base $(\M,\W)$ contains $(\M, 2\tx{-Iso})$. 

Note that if we take $\W$ to be the class $2$-$\tx{Mor}(\M)$ of all of $2$-cells we get the largest base $(\M, 2\tx{-Mor}(\M))$
\end{rmk}

\subsection{Path-object}

\begin{df}{\emph{[Path-object]}}\label{p-object}
Let $(\M, \W)$ be a base of enrichment.  A \emph{Path-object} of $(\M,\W)$ is a couple $(\C,F)$, where $\C$ is a small category and $F=(F,\varphi)$ a \underline{colax} morphism of Bénabou:
$$F : \P_{\C} \longrightarrow \M $$
such that for any objects $A$, $B$, $C$ of $\C$ and  any pair $(t,s)$ in $\P_{\C}(B,C) \times \P_{\C}(A,B)$, all  the $2$-cells 
$$F_{AC}(t \otimes s) \xrightarrow{\varphi(A,B,C)(t,s)} F_{BC}(t) \otimes F_{AB}(s) $$ 
$$ F_{AA}([0,A]) \xrightarrow{\varphi_{A}} I'_{FA}$$
\underline{are in} $\W$. Such a colax morphism will be called a \textbf{$\W$-colax morphism}.
\end{df}

\begin{term} \
\renewcommand\labelenumi{\alph{enumi})} 
\begin{enumerate}
\item If $\W$ is a class of $2$-cells called $2$-\emph{homotopy equivalences} then $(\C,F)$ will be called a \emph{Segal path-object}. The maps  $\varphi(A,B,C)$ and $\varphi_{A}$ will be called \emph{Segal maps}. If $F$ is a strict homomorphism (respectively nonstrict homomorphism) we will say that that $(\C,F)$ is a \underline{strict} Segal path-object (respectively \emph{quasi-strict} Segal path-object).\\

\item For $U$ in $Ob(\M)$, an object \textit{over} $U$ is an object $A$ of  $\C$ such that $FA=U$ (See  Figure 1 below). Here we've followed the geometric picture in enrichment over bicategories as in \cite{Str}, \cite{Wa1}, \cite{Wa2}. Sometimes it's also worthy to think it as an object \textbf{connected to $U$} because we're going to use the combinatoric of $\P_{\C}$ to `extract' from $\M$, `the skeleton' of a category.    \\

\item If $\W=2$-$Mor(\M)$, we will not mention $\W$ and call $(\C,F)$ a path-object of $\M$.\\
 
\item Since a path-object is a sort of morphism from $\P_{\C}$ to $\M$ we will call it a `$\P_{\C}$-point' or a `$\P_{\C}$-module' of $\M$. And for short we will simply say \textbf{$\C$-point} or \textbf{$\C$-module} of $\M$. We will therefore say Segal $\C$-point (or $\C$-module) for a Segal path-object $(\C,F)$.
\end{enumerate}
\end{term}

\includegraphics[width=16cm,height=10cm]{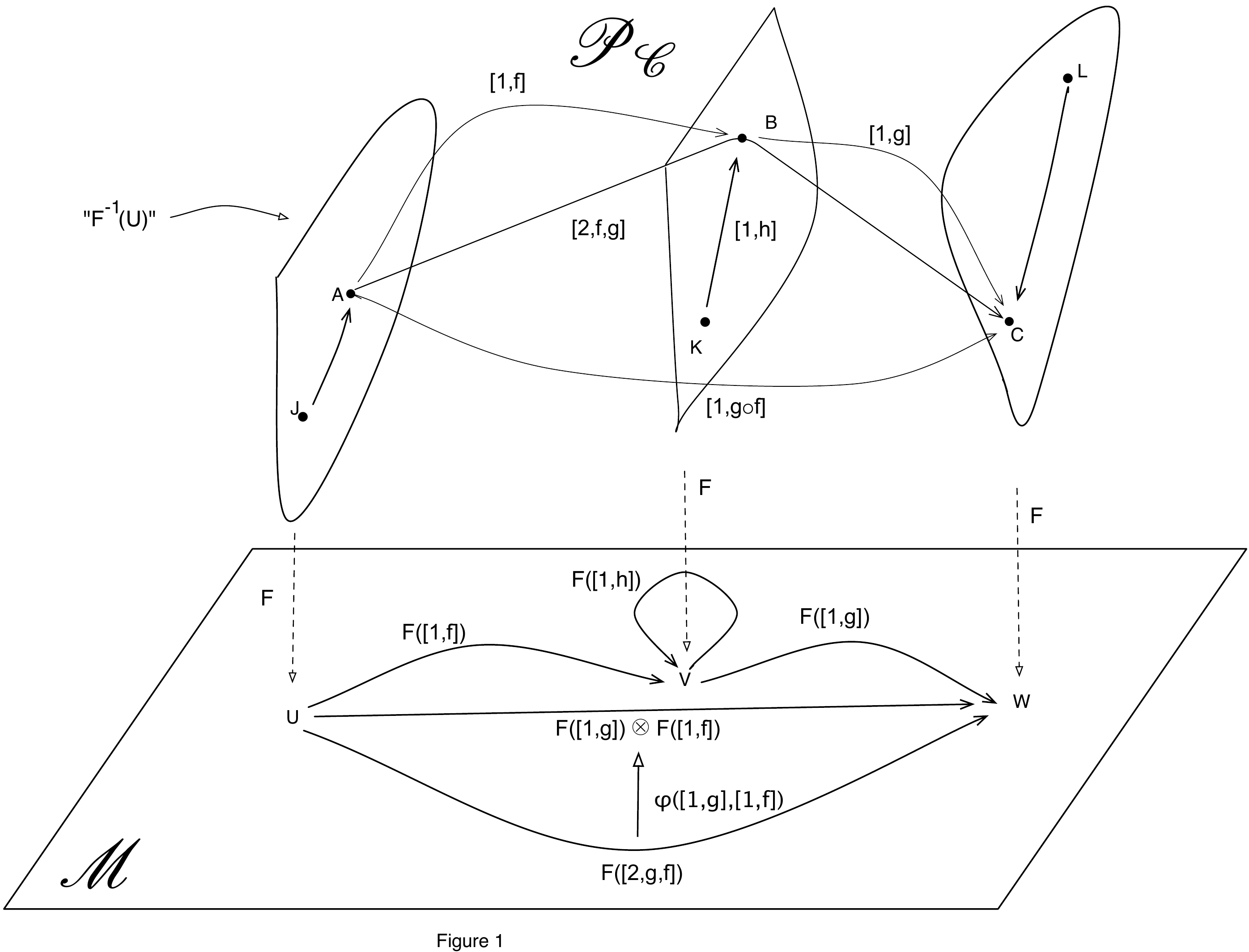}

\begin{obs} \ \\
\indent In Figure 1, we took $t=[1,B \xrightarrow{g} C]$ and $s=[1,A \xrightarrow{f} B]$. \\
We have $t \otimes s =[2,A \xrightarrow{f} B \xrightarrow{g} C]$ and a canonical $2$-cell in  $\P_{\C}(A,C)$
\[
\xy
(0,0)*+{A}="A";
(30,0)*+{C}="C";
{\ar@/_1.3pc/_{[1,A \xrightarrow{g \circ  f} C]}"A";"C"};
{\ar@/^1.3pc/^{[2,A \xrightarrow{f} B \xrightarrow{g} C]}"A";"C"};
{\ar@{=>}_{2 \xrightarrow{!} 1}(15,4);(15,-4)};
\endxy
\]
given by the composition in $\C$ and `parametrized' by the (unique) arrow  $2 \xrightarrow{!} 1$ of $\Delta$. The image by $F$ of this $2$-cell is a $2$-cell of $\M$
\[
\xy
(0,0)*+{U}="A";
(30,0)*+{W}="C";
{\ar@/_1.3pc/_{F_{AC}([1,A \xrightarrow{g \circ  f} C])}"A";"C"};
{\ar@/^1.3pc/^{F_{AC}([2,A \xrightarrow{f} B \xrightarrow{g} C])}"A";"C"};
{\ar@{=>}_{}(15,4);(15,-4)};
\endxy.
\]
Now if we combine this with the colaxity map $\varphi(A,B,C)(t,s)$ we have the following span in $\M(U,W)$ :
\[
\xy
(-10,0)*+{F_{BC}([1,B \xrightarrow{g} C]) \otimes F_{AB}([1,A \xrightarrow{f} B])}="X";
(50,0)*+{F_{AC}([1,A \xrightarrow{g \circ  f} C])}="Y";
(22,22)*+{F_{AC}([2,A \xrightarrow{f} B \xrightarrow{g} C])}="Z";
{\ar@{->}_{\varphi(A,B,C)(t,s)~~~~~}"Z"; "X"};
{\ar@{->}^{F_{AC}(`2 \xrightarrow{!} 1')}"Z"; "Y"};
\endxy
\]

If $\varphi(A,B,C)(t,s)$ is a \emph{weak equivalence} (\textit{e.g} a Segal map) therefore is \emph{weakly} invertible, any choice of a weak inverse of $\varphi(A,B,C)(t,s)$ will give a map :
$$ F_{BC}([1,B \xrightarrow{g} C]) \otimes F_{AB}([1,A \xrightarrow{f} B]) \longrightarrow F_{AC}([1,A \xrightarrow{g \circ  f} C]) $$
by running the span from the left to the right.

In that situation if we want this construction to be consistent, we have to assume that all the weak inverses of $\varphi(A,B,C)(t,s)$ must be \emph{homotopy equivalent} in some sense. In this way the `space' of the maps
$$ F_{BC}([1,B \xrightarrow{g} C]) \otimes F_{AB}([1,A \xrightarrow{f} B]) \longrightarrow F_{AC}([1,A \xrightarrow{g \circ  f} C]) $$
obtained for each weak inverse, will be  \emph{contractible} in some sense.

One of the interesting situations is when  $F_{BC}([1,B \xrightarrow{g} C])$ and $F_{AB}([1,A \xrightarrow{f} B])$ stand for hom-objects of some category-like structure.  The maps $$ F_{BC}([1,B \xrightarrow{g} C]) \otimes F_{AB}([1,A \xrightarrow{f} B]) \longrightarrow F_{AC}([1,A \xrightarrow{g \circ  f} C]) $$
will be a sort of composition up-to \emph{homotopy}, like for classical Segal categories.
\end{obs}

\begin{rmk} \ \
\renewcommand\labelenumi{\alph{enumi})} 
\begin{enumerate}
\item For every object $U$ of $\M$, denote by $F^{-1}(U)$  the set  of objects of $\C$ over $U$ \emph{via} $F$. If $F^{-1}(U)$ is nonempty consider the full subcategory $\C_U$ of $\C$, corresponding to the ``restriction'' of $\C$ to $F^{-1}(U)$. Then $F$ gives a \textbf{`foliation' of $\C$} of `leaves' $\C_U$. We get by functoriality a canonical injection $\P_{\C_U} \hookrightarrow \P_{\C}$ and the composition by $F$ gives a $\C_U$-point of $(\M_{UU}, \W_{UU})$. \ \\
\item We see that a $\C$-point of a bicategory $(\M,\W)$   is a \emph{`moduli'}\footnote{or is ``generated'' by $\C_i$} of `bimodules' between $\C_i$-points of some monoidal bases of enrichment  $(\M_i,\W_i)$. As one can see if we start with a monoidal category $(\M,\W)$, all object of $\C$ will be over the same object, say $\ast$, with $\Hom(\ast,\ast)= \M$. \ \\ 
\item We've used the terminology of foliation theory because each $\C_U$ can be an algebraic leaf of a foliated manifold $\C$. In that case $\C_U$ is determined by a collection of rings satisfying a (co)-descent condition. We will have a path-object of the bicategory \textbf{Bim} of rings, bimodules and morphism of bimodules (see \cite{Ben2}). Here again we will need a descent theory for path-objects.\ \\
\item In \textbf{Bim} we have both commutative and noncommutative rings so it appears to be a good place where both commutative and noncommutative geometry meet. Then the study of path-objects of \textbf{Bim} (and its higher versions) needs to be considered seriously. 
\end{enumerate}
\end{rmk}
\begin{obs} \ \\
\indent The collection of $\C$-points of $(\M, \W)$ forms naturally a bicategory $\M_{\W}^{+}(\C)=\tx{Bicat}[\W](\P_{\C},\M)$, of $\W$-colax morphisms, transformations and modifications. In fact, in Bicat  one has an \emph{internal colax-$\Hom$} between any two bicategories. In particular we have a `colax-Yoneda' functor (of points)\footnote{This justifies our terminology of `$\C$-points'} $\tx{Bicat}_{colax}(-,\M)$. We recall briefly this bicategorical structure on $\M_{\W}^{+}(\C)$ as follows.  
\renewcommand\labelenumi{\alph{enumi})} 
\begin{enumerate}
\item $Ob(\M_{\W}^{+}(\C))= \{F : \P_{\C} \longrightarrow \M \}$, the collection of $\W$-colax morphisms of Bénabou between $\P_{\C}$ and $\M$,
\item For every pair $(F,G)$ of $\W$-colax morphisms, a $1$-cell  $\sigma: F \to G$ is a \underline{transformation} of morphisms of bicategories 
\[
\xy
(0,0)*+{\P_{\C}}="A";
(30,0)*+{\M}="C";
{\ar@/_1.3pc/_{G}"A";"C"};
{\ar@/^1.3pc/^{F}"A";"C"};
{\ar@{=>}_{\sigma}(15,4);(15,-4)};
\endxy.
\]

\item For every pair $(\sigma_1,\sigma_2)$ of $1$-cells, a $2$-cell $\Gamma : \sigma_1 \to \sigma_2$ is a \underline{modification} of transformations :
\[
\xy
(0,15)*{}; 
(0,-15)*{}; 
(0,8)*{}="A"; 
(0,-8)*{}="B"; 
{\ar@{=>}@/_1pc/ "A"+(-4,1) ; "B"+(-3,0)}; 
{\ar@{=}@/_1pc/_{\sigma_1} "A"+(-4,1) ; "B"+(-4,1)}; 
{\ar@{=>}@/^1pc/ "A"+(4,1) ; "B"+(3,0)}; 
{\ar@{=}@/^1pc/^{\sigma_2} "A"+(4,1) ; "B"+(4,1)}; 
{\ar@3{->}^{\Gamma} (-6,0)*{} ; (6,0)*+{}}; 
(-15,0)*+{\P_{\C}}="1"; 
(15,0)*+{\M}="2"; 
{\ar@/^2.75pc/^{F} "1";"2"}; 
{\ar@/_2.75pc/_{G} "1";"2"}; 
\endxy.
\]
\end{enumerate}
The definitions of transformations and modifications are recalled in section \ref{morph}.\\
\end{obs}

\paragraph{The coarse category}\label{coarse}\ \\
\indent For any nonempty set $X$ we denote by $\ol{X}$ the \textbf{coarse category}\footnote{The symbol `$-$' means here that we put one ``link'' between any two elements of $X$} associated to $X$. This is the category having $X$ as set of objects and with exactly one morphism between two objects. Some authors called it the \emph{chaotic category} or the \emph{indiscrete category}. If $A$ and $B$ are elements of $X$ we will denote by $(A,B)$ the unique arrow  in $\ol{X}$ from $A$ to $B$. In particular $(A,A)$ is the identity morphism of $A$.
\begin{rmk}\label{rmk-coar} \ \
\renewcommand\labelenumi{\alph{enumi})}
\begin{enumerate}
\item One may observe that $\ol{X}$ looks like EG for some group G. In fact G is the free group associated to the set $X \times X$ quotiented by the relation of composition and unity. As we shall see in a moment EG is a G-category.
\item $\ol{X}$ is  a \emph{groupoid} and one can observe that this construction is \underline{functorial}. So we have a `coarse' functor:
$$ \ol{[~~~]} :\emph{\tx{Set}} \to \emph{\tx{Gpds}}.$$
\item When $X$ has only one element, say $X=\{A\}$, $\ol{X}$ consist of the object $A$ with the identity $1_{A}$, hence $\ol{X} \cong \bf{1}$. The Proposition 1 gives an monoidal isomorphism between  $\P_{\ol{X}}(A,A)$ and $(\Delta,+,0)$.
\end{enumerate}
\end{rmk}
\begin{term}
For any nonempty set $X$ we will simply say \textbf{$\ol{X}$-point} or \textbf{$\ol{X}$-module} of $(\M,\W)$ for a path-object $(\ol{X},F)$ of $(\M,\W)$.
\end{term}
We will write $\M_{\W}^{+}(X) := \tx{Bicat}[\W](\P_{\ol{X}},\M)$, for the bicategory of $\W$-colax morphisms from $\P_{\ol{X}}$  to $\M$.\\

\begin{obs}\
\begin{enumerate}
\item  For a set $X$, we've considered the coarse category $\ol{X}$ which is a groupoid, but one may consider any \emph{preorder}. Recall that  by preorder we mean a category in which there is at most one morphism between any two objects. For preorders $R$,  The $R$-points of $(\M,\W)$  are important because in some sense they `generate' the general $\C$-points  for arbitrary small categories $\C$.
This comes from the \emph{nerve} construction of a small category.\\

Recall that for a category $\C$ one defines the nerve of $\C$ to be the following functor :
$$ \Nv(\C):(\Delta^{+})^{op} \to Set $$
$$ n \mapsto \Hom([n],\C)$$
where $[n]$ is the preorder with $n$ objects. Explicitely  $[n]$ is  the category defined as follows.\\
Take $Ob([n])=\{0,1,\cdots,n \}$ the set of the first $n+1$ natural numbers and
\begin{equation*}
[n](i,j) =
  \begin{cases}
     \{(i,j) \} & \text{if $i<j$} \\
     \{\Id_i=(i,i) \}  & \text{if $i=j$ }\\
     \varnothing & \text{if $i>j$ }
  \end{cases}
\end{equation*}
The composition is the obvious one.\\

The set $\Nv(\C)_n=\Hom([n],\C)$ is the set of $n$-composable arrows of $\C$ trough $(n+1)$ objects $A_0,...,A_n$ :
$$A_0 \xrightarrow{f_1} \cdots A_{i-1} \xrightarrow{f_{i}} A_{i} \to \cdots \xrightarrow{f_n} A_n.$$
Each element of $\Nv(\C)_n$ is called a \emph{$n$-simplex} of $\C$. The $n$-simplices of $\C$ from $A_0$ to $A_n$ are exactly the $1$-cells (of length $n$) in $\P_{\C}(A_0,A_n)$.

It's important to notice that for every $n$-simplex $r$ of $\C$, $i.e$ a functor $r:[n]\to \C$, the image of $r$ 
$$A_0 \xrightarrow{f_1} \cdots A_{i-1} \xrightarrow{f_{i}} A_{i} \to \cdots \xrightarrow{f_n} A_n$$ is a category $[r]$ which is a copy of $[n]$, described as follows.\\
$Ob([r])=\{A_0,\cdots, A_n\}$ and 
\begin{equation*}
[r](A_i,A_j) =
  \begin{cases}
     \{f_j \circ \cdots \circ f_i \} & \text{if $i<j$} \\
     \{\Id_{A_i} \}  & \text{if $i=j$ }\\
     \varnothing & \text{if $i>j$ }
  \end{cases}
\end{equation*}

Now if all the arrows $f_0, \cdots ,f_n$ are invertible we can extends $[r]$ to a coarse category by adding the inverse of each $f_i$  or by formally adding the inverses of the $f_i$. This will be the case where $\C$ is a groupoid or by localizing $\C$ with respect to some class of morphisms $\S$.\\

Since the construction of the path-bicategory is functorial it follows that for every $n$-simplex $r$ of $\C$, $i.e$ a functor $r:[n]\to \C$, we have a strict homomorphism $\P_r : \P_{[n]} \to \P_{\C}$. Therefore any $\C$-point $F: \P_{\C} \to \M$ gives by pullback an $[n]$-point  

\end{enumerate}
\end{obs}
\section{Examples of path-objects}\label{examples}
\subsection{Up-to-homotopy monoids and Simplicial objects}
\begin{df}
Let $\N$, $\N'$ be two monoidal categories. A \emph{colax} monoidal functor $\N' \to \N$ consists of a functor $Y: \N' \to \N$ together with maps 
$$ \xi_{AB}: Y(A\otimes B) \to Y(A) \otimes Y(B)$$
$$\xi_{0}: Y(I) \to I $$
$(A,B \in \N')$, satisfying naturality and coherences axioms.
\end{df}
Here $\otimes$ and $I$ denote the tensor operation and unit object in both monoidal categories $\N'$ and $\N$. We refer the reader to \cite{Lei2} for the  coherences axioms of colax functors.\\

The following defintion is due to Leinster. The generalization of this definition was the motivation of this work.
\begin{df}\label{h-mon}
Let $\M$ be a monoidal category equipped with a class of \emph{homotopy equivalences} $\W$ such that the pair $(\M,\W)$ is a base of enrichment. A \emph{homotopy monoid} in $\M$ is colax monoidal functor 
$$(Y, \xi) : \Delta \to \M $$
for which the maps $\xi_0$, $\xi_{mn}$ are in $\W$ for every $m,n$ in $\Delta$.
\end{df}
\begin{df}
Let $\M$ be a category. A simplicial object of $\M$ is a functor $Y:(\Delta^{+})^{op} \to \M$.  
\end{df}
\begin{rmk}
In this definition $\M$ may be a \textbf{higher category}. For example in \cite{HS}, Simpson and Hirschowitz use simplicial objects of some  higher $($model$)$ category to define inductively Segal $n$-categories. 
\end{rmk}
A special case of simplicial object which is relevant to our path-objects comes when the ambiant category $\M$ is a category with finite products with an empty-product object  $1$, which is terminal by definition. In this case $\M$ turns to be a \textbf{cartesian} monoidal category $(\M,\times ,1$).\\

The following proposition is due to Leinster \cite{Lei2}.  
\begin{prop}\label{leiprop}
Let $(\M,\times)$ be a category with finite products. Then there is an isomorphism of categories
$$ Colax((\Delta,+,0),(\M,\times,1)) \cong [(\Delta^{+})^{op},\M ].$$ 
\end{prop}
\begin{rmk}
$Colax((\Delta,+,0),(\M,\times,1))$  represents the category of colax monoidal functors. 
\end{rmk}

In what follows we are going to rephrase the proposition and the definitions given above in term of points of $\M$. We will use the following notations.\\
$\bf{1}=\{\o,\o \xrightarrow{Id_{\o}} \o \}=$ the unit category.\\
$Iso(\M)=$ the class of all invertible morphisms of $\M$.\\
$Mor(\M)=$ the class of all morphisms of $\M$.\\
As usual, since Bénabou, we will identify $\M$ with a bicategory with one object (See Example \ref{mon} of the Appendix A).
\begin{prop}\label{hotmon}
Let $(\M,\otimes,I)$ be a monoidal category. 
\renewcommand\labelenumi{\roman{enumi})}
\begin{enumerate}
\item We have an equivalence between the following data:
\begin{itemize}
\item a $\bf{1}$-point of $(\M,Iso(\M))$ \it{i.e} an object of $\M_{Iso(\M)}^{+}(\bf{1})$,
\item a monoid of $\M$.
\end{itemize}
\item Assume that $\M$ is equipped with a class $\W$ of morphisms called homotopy equivalences, such that $(\M,\W)$ is a base of enrichment. Then we have an equivalence between the following data
\begin{itemize}
\item a $\bf{1}$-point of $(\M,\W)$ \it{i.e} an object of $\M_{\W}^{+}(\bf{1})$,
\item an up-to homotopy monoid in the sense of Leinster \cite{Lei3}.
\end{itemize} 
\item If $\M$ has finite products, and is considered to be monoidal for the cartesian product, then  we have an equivalence between 
\begin{itemize}
\item a $\bf{1}$-point of $(\M,Mor(\M))$ \it{i.e} an object of $\M_{Mor(\M)}^{+}(\bf{1})$, 
\item a simplicial  object of $\M$.
\end{itemize}
\end{enumerate}
\end{prop}
\begin{obs}\ 
\renewcommand\labelenumi{\alph{enumi})}
\begin{enumerate}
\item The assertion (i) is simply a particular case of (ii) when  $\W=Iso(\M)$. The main motivation to consider general classes of \emph{homotopy equivalences} $\W$ other than $Iso(\M)$ is to have a Segal version of enriched categories over monoidal categories. The idea is to  view a monoid of $\M$ as a category enriched over $\M$ with one object and to view an up-to-homotopy monoid of Leinster as the Segal version of it.

The `several objects' case is considered in the upcoming examples. We will call them \textbf{Segal enriched categories}.

\item As pointed out earlier, in the assertion (iii) $\M$ may be a higher category having finite  product. This suggests to  extend the definition  of $\C$-points of $(\M,\W)$ (Definition \ref{p-object}) to a general one where $\C$ and $\M$ are $\infty$-categories. A first attempt would consist to use \textbf{Postnikov systems} (see \cite{BS}) to \textbf{``go down''} to the current situation. We will come back to this later. 
\item Again in the assertion (iii), the category $\M$ may have \textbf{discrete objects} and we can  give a definition of Segal categories or \emph{(weak) internal category object} in $\M$ (see \cite{HS}) in term of $1$-point of $(\M,\W)$. \\
An immediate step is to ask what we will have with general $\ol{X}$-points. This is discussed later.
\end{enumerate}
\end{obs}
\ \\
The proof of the proposition is based on the following two facts:
\begin{enumerate}
\item[$\bullet$]the path-bicategory of $\bf{1}$ `is' $\Delta$ (see Proposition \ref{path-bicat}). 
\item[$\bullet$] bicategories with one object and morphisms between them  are identified with monoidal categories and  the suitable functors of monoidal categories, and vice versa.
\end{enumerate}
\begin{proof}
Let $F$ be $\bf{1}$-point  of $(\M,\W)$. By definition $F$ is  a $\W$-colax morphism of bicategories $F : \P_{\bf{1}} \longrightarrow \M $.\\
As $\P_{\bf{1}}$ is a one-object bicategory, $F$ is entirely determined by the following data:
\begin{enumerate}
\item[$\ast$]a functor $F_{\o\o}: \P_{\bf{1}}(\o,\o) \to \M$ which is the only component of $F$
\item[$\ast$] arrows $F_{\o\o}(t \otimes s) \xrightarrow{\varphi(\o,\o,\o)(t,s)} F_{\o\o}(t) \otimes F_{\o\o}(s) $ in $\W$, for every pair $(t,s)$ in $\P_{\bf{1}}(\o,\o)$,
\item[$\ast$] an arrow $ F_{\o\o}([0,\o]) \xrightarrow{\varphi_{\o}} I$ in $\W$,
\item[$\ast$] coherences on $\varphi(\o,\o,\o)(t,s)$ and $\varphi_{\o}$.
\end{enumerate}
But one can check that these data say exactly that $F_{\o\o}$ is a \textbf{colax monoidal functor}  from \\ 
$(\P_{\bf{1}}(\o,\o),c(\o,\o,\o),[0,\o])$ to $(\M, \otimes, I)$.\\

As remarked previously we have an isomorphism of monoidal categories $$ (\Delta,+,0) \cong (\P_{\bf{1}}(\o,\o),c(\o,\o,\o),[0,\o]).$$ We recall that this isomorphism is determined by the following identifications.\\ 
$ 0 \longleftrightarrow [0,\o].$\\
$ n \longleftrightarrow [n,\underbrace{\o \xrightarrow{Id_{\o}} \o \cdots \o \xrightarrow{Id_{\o}} \o}_{n~~identities}]=s.$\\
$ m \longleftrightarrow [m,\underbrace{\o \xrightarrow{Id_{\o}} \o \cdots \o \xrightarrow{Id_{\o}} \o}_{m~~identities}]=t.$\\
$ (n+m) \longleftrightarrow [n+m,\underbrace{\o \xrightarrow{Id_{\o}}   \cdots  \o}_{n~~identities} \underbrace{\xrightarrow{Id_{\o}} \cdots \o}_{m~~identities}]=c(\o,\o,\o)(t,s)=t \otimes s.$\\
\\
$\{$Coface maps in $\Delta$ $\}$ $\longleftrightarrow$ $\{$Replacing consecutive arrows by their composite $\}$.\\ 
$\{$Codegeneracy maps in $\Delta$ $\}$ $\longleftrightarrow$ $\{$ Adding identities $\}$  (see Appendix B).\\ 
\\

Summing up the above discussions we see that  $F$ is equivalent to a colax monoidal functor from $(\Delta,+,0)$ to $(\M,\otimes,I)$:\\
$\widetilde{F} : \Delta \to \M$,\\
$\varphi_{mn}:\widetilde{F}(m+n) \to \widetilde{F}(m) \otimes \widetilde{F}(n) \in \W$,\\
$\varphi_{0}:\widetilde{F}(0) \to I \in \W$.\\

$\bullet$If $\W$ is a class of homotopy equivalences, we recover the definition  of a homotopy monoid given by Leinster in \cite{Lei3}, which proves the assertion (ii).\\

$\bullet$ If $\W=Mor(\M)$  and $\M$ is cartesian monoidal we get an object of  $Colax((\Delta,+,0),(\M,\times,1))$ and the assertion  (iii) follows from the Proposition 7 above.\\ 

\end{proof}

\begin{rmk}
When we will view what is a morphism of $\C$-points each equivalence will be automatically an equivalence of categories. 
\end{rmk}

\subsection{Classical enriched categories}
\begin{prop}\label{c-enr}
Let $(\M,\otimes,I)$ be a monoidal category and $X$ be a nonempty set. We have an equivalence between the following data
\renewcommand\labelenumi{\roman{enumi})}
\begin{enumerate}
\item an $\ol{X}$-point of $(\M,Iso(\M))$ 
\item an enriched category over $\M$ having $X$ as set of objects.
\end{enumerate}
\end{prop}


\begin{proof}
Let $F : \P_{\ol{X}} \longrightarrow \M $  be an $\ol{X}$-point of $(\M,Iso(\M))$. By definition of a Path-object (Definition $4.$), $F$ is a $Iso(\M)$-colax morphism, which means that the  maps 
$$F_{AC}(t \otimes s) \xrightarrow{\varphi(A,B,C)(t,s)} F_{BC}(t) \otimes F_{AB}(s),~~~~~ \  F_{AA}([0,A]) \xrightarrow{\varphi_{A}} I$$
are invertible for every elements $A,B,C$ of $X$. It follows that $F$ is a (colax) \textbf{homomorphism} in the sense of Bénabou \footnote{For homomorphism being colax or lax is almost the same thing}.\\ 

As $F$ is a morphism of bicategories we have a functor $F_{AB}:\P_{\ol{X}}(A,B) \longrightarrow \M$,  for each pair of elements $(A,B)$.
\bigskip

We are going to form a category  enriched over $\M$, in the classical sense (see \cite{Ke}), denoted by $\M^{X}_{F}$.
\renewcommand\labelenumi{\roman{enumi})}
\begin{enumerate}
\item Put  $Ob(\M^{X}_{F})= X$.
\item For every pair $(A,B) $ of elements of $X$, the hom-object is $\M^{X}_{F}(A,B):= F_{AB}([1,(A,B)]) \in Ob(\M)$. 
\item For  every $A \in X$  we get the unit map $I_A : I \to \M^{X}_{F}(A,A)$ in the following manner.\\
First observe that the composition in $\P_{\ol{X}}$  induces a monoidal structure on  $\P_{\ol{X}}(A,A)$  with unit $[0,A]$. \\ 
Moreover we have a canonical $2$-cell $[0,A] \xrightarrow{!} [1,(A,A)]$ in $\P_{\ol{X}}(A,A)$, parametrized by the unique map $0 \xrightarrow{!} 1$ of $\Delta$ (see Appendix B).

The image of this $2$-cell by the  functor $F_{AA}$ is a morphism of $\M$ : $F([0,A]) \xrightarrow{F_{AA}(!)} F([1,(A,A)])$.\\ And we take the unit map $I_A$ to be the composite :
$$ I \xrightarrow{{\varphi_{A}}^{-1}} F([0,A]) \xrightarrow{F_{AA}(!)} F([1,(A,A)])= \M^{X}_{F}(A,A).$$
\item For every triple $(A,B,C)$ of elements of $X$ we construct the composition as follows. Consider the following span in $\M$ (see Observations 2) :
\[
\xy
(-10,0)*+{F_{BC}([1,(B,C)]) \otimes F_{AB}([1,(A,B)])}="X";
(50,0)*+{F_{AC}([1,(A,C)])}="Y";
(25,2)*+{c_{ABC}}="c";
(22,22)*+{F_{AC}([2,(A,B,C)])}="Z";
{\ar@{->}_{\varphi(A,B,C)(t,s)}"Z"; "X"};
{\ar@{->}^{~~~~F_{AC} \{[2,(A,B,C)] \xrightarrow{!} [1,(A,C)] \} }"Z"; "Y"};
{\ar@{.>}"X";"Y"};
\endxy
\]
with $t=[1,(B,C)]$, $s=[1,(A,B)]$ and $t\otimes s=[2,(A,B,C)]$. 
Since $\varphi(A,B,C)(t,s)$ is invertible, we can consider the composite
$$c_{ABC}= F_{AC}\{[2,(A,B,C)] \xrightarrow{!} [1,(A,C)]\} \star {\varphi(A,B,C)(t,s)}^{-1}$$ which will be our composition in $\M^{X}_{F}$.
\item Finally one can easily check that the  associativity and unity axioms required in $\M^{X}_{F}$ follow directly from the axioms required in the definition of  the morphism $F$.
\end{enumerate}
It's clear that the above data give a category enriched over $\M$ in the classical sense, as desired.

\bigskip

Conversely  let $\A$ be a category enriched over $\M$ in the classical sense. For simplicity we will assume that $Ob(\A)$ is a set \footnote{Otherwise we can choose a suitable  Grothendieck universe $\U$ so that $Ob(\A)$ will be a $\U$-set} and we will put $X=Ob(\A)$ .\\ 
In what follows we are going to construct a $Iso(\M)$-colax morphism from  $\P_{\ol{X}}$ to $\M$ denoted by $[\A]$.\\
Let's fix some notations:
\begin{enumerate}
\item[$\bullet$] If $(U_1,\cdots,U_n)$ is a $n$-tuple of objects of $\M$ we will write $U_1 \otimes \cdots \otimes U_n$ for the tensor product  of $U_1,\cdots,U_n$ with all pairs of parentheses starting in front,
\item[$\bullet$] Similary if $(f_1,\cdots,f_n)$ is a $n$-tuple of morphisms of $\M$ we will write $f_1 \otimes \cdots \otimes f_n$ for the tensor product of $f_1,\cdots,f_n$ with all pairs of parentheses starting in front,
\item[$\bullet$] For every object $U$ of $\M$ we will write $r_U$ (resp. $l_U$) for the righ identity (resp. left identity) isomorphism $ U\otimes I \xrightarrow{\sim} U$ (resp. $ I\otimes U \xrightarrow{\sim} U$). We will write $\Id_U$ for identity morphism of $U$.
\item[$\bullet$] If $A$ and $B$ are objects of $\A$, hence elements of $X$, we will write $\A(A,B)$ for the hom-object in $\A$,
\item[$\bullet$] For every triple $(A,B,C)$ of objects we design by $c_{ABC}$ the morphism of $\M$ corresponding to composition :
$$c_{ABC}: \A(B,C) \otimes \A(A,B) \to \A(A,C),$$
\item[$\bullet$] For every $A$ we denote by $I_A$ the unit map : $I_A: I \to \A(A,A)$.
\end{enumerate}
We recall that for each pair $(A,B)$ of elements of $X$, $\P_{\ol{X}}(A,B)$ is the category of elements of some functor from $\Delta^{+}$ to Set \footnote{ from $\Delta$ to Set if $A=B$}. It follows that the morphisms of $\P_{\ol{X}}(A,B)$ are parametrized by the morphisms of $\Delta$.\\

We remind that the morphisms of $\Delta$ are generated by the \emph{cofaces} $d^{i}:n+1 \to n$, and the \emph{codegeneracies} $s^{i}:n \to n+1$ (see \cite{Mac}). Therefore the morphisms of $\P_{\ol{X}}(A,B)$ are generated by the following two types of morphisms:
\[
\xy
(-20,15)*+{[n+1,A \to \cdots A_{i-1} \xrightarrow{(A_{i-1},A_{i})} A_{i}  \xrightarrow{(A_{i},A_{i+1})} A_{i+1} \cdots \to B]}="X";
(-20,0)*+{[n,A \to \cdots A_{i-1} \xrightarrow{(A_{i-1},A_{i+1})} A_{i+1}\cdots \to B]}="Y";
(-90,7)*+{(\ast)};
{\ar@{->}^{d^{i}}"X";"Y"};
\endxy
\]

and 
\bigskip
\[
\xy
(-20,15)*+{[n,A \to \cdots A_{i} \xrightarrow{(A_{i},A_{i+1})} A_{i+1}\cdots \to B]}="X";
(-20,0)*+{[n+1,A \to \cdots A_{i} \xrightarrow{(A_{i},A_{i})} A_{i}  \xrightarrow{(A_{i},A_{i+1})} A_{i+1} \cdots \to B]}="Y";
(-90,7)*+{(\ast \ast)};
{\ar@{->}^{s^{i}}"X";"Y"};
\endxy.
\]
The morphisms of type $(\ast)$ correspond to composition and those of type $(\ast \ast)$ correspond to add the identity of the $i$th object. 

Now we define  the functors $[\A]_{AB}: \P_{\ol{X}}(A,B) \longrightarrow \M$, the components of $[\A]$, in the following manner.
\begin{enumerate}
\item[$\bullet$] The image of $[n,A \to \cdots A_{i} \xrightarrow{(A_{i},A_{i+1})} A_{i+1}\cdots \to B]$ by  $[\A]_{AB}$ is the object of $\M$ : $$\A(A_{n-1},B)\otimes \cdots \otimes \A(A_{i},A_{i+1})\otimes \cdots \otimes \A(A,A_1).$$
\item[$\bullet$] When $A=B$ the image of $[0,A]$ is the object $I$, the unity of $\M$. 
\item[$\bullet$] The image of a morphism of type $(\ast)$ by $[\A]_{AB}$ is the composite :  
\[
\xy
(-20,15)*+{\A(A_{n},B)\otimes \cdots \otimes   \A(A_{i},A_{i+1}) \otimes \A(A_{i-1},A_{i})\otimes \cdots \otimes \A(A,A_1)}="X";
(-20,0)*+{\A(A_{n},B)\otimes \cdots \otimes  [\A(A_{i},A_{i+1}) \otimes \A(A_{i-1},A_{i})]\otimes \cdots \otimes \A(A,A_1)}="Y";
(-20,-15)*+{\A(A_{n},B)\otimes \cdots \otimes \A(A_{i-1},A_{i+1})\otimes \cdots \otimes \A(A,A_1)}="Z";
{\ar@{->}^{\sim}"X";"Y"};
{\ar@{->}^{\Id_{\A(A_{n},B)}\otimes \cdots \otimes c_{A_{i-1}A_{i}A_{i+1}} \otimes \cdots \otimes \Id_{\A(A,A_1)}}"Y";"Z"};
\endxy.
\]
\item[$\bullet$] Similary the image of a morphism of type $(\ast \ast)$ is the composite:
\[
\xy
(-20,15)*+{\A(A_{n-1},B)\otimes \cdots \otimes \A(A_{i},A_{i+1})\otimes \cdots \otimes \A(A,A_1)}="X";
(-20,0)*+{\A(A_{n-1},B)\otimes \cdots \otimes  [I \otimes \A(A_{i},A_{i+1})]\otimes \cdots \otimes \A(A,A_1)}="Y";
(-20,-15)*+{\A(A_{n-1},B)\otimes \cdots \otimes   [\A(A_{i},A_{i}) \otimes \A(A_{i},A_{i+1})]\otimes \cdots \otimes \A(A,A_1)}="Z";
(-20,-30)*+{\A(A_{n-1},B)\otimes \cdots \otimes  \A(A_{i},A_{i}) \otimes \A(A_{i},A_{i+1}) \otimes \cdots \otimes \A(A,A_1)}="W";
{\ar@{->}_{\sim}^{\Id_{\A(A_{n-1},B)}\otimes \cdots \otimes l_{\A(A_{i},A_{i+1})}^{-1} \otimes \cdots \otimes \Id_{\A(A,A_1)}}"X";"Y"};
{\ar@{->}^{\Id_{\A(A_{n},B)}\otimes \cdots \otimes [I_{A_i} \otimes \Id_{\A(A_{i},A_{i+1})}] \otimes \cdots \otimes \Id_{\A(A,A_1)}}"Y";"Z"};
{\ar@{->}_{\sim}"Z";"W"};
\endxy.
\]
\item[$\bullet$]Using the bifunctoriality of the tensor product $\otimes$ in $\M$, its associativity and the fact that morphisms of type $(\ast)$ and $(\ast \ast)$ generate all the morphisms of   $\P_{\ol{X}}(A,B)$, we extend the above operations to a functor $[\A]_{AB}: \P_{\ol{X}}(A,B) \longrightarrow \M$ as desired.
\end{enumerate}
To complete the proof, we have to say what are the morphisms  $[\A]_{AC}(t \otimes s) \xrightarrow{\varphi(A,B,C)(t,s)} [\A]_{BC}(t) \otimes [\A]_{AB}(s) $ for each pair $(t,s)$ in $\P_{\ol{X}}(B,C) \times \P_{\ol{X}}(A,B)$.\\
But if  $s=[n,A \to \cdots A_{i} \xrightarrow{(A_{i},A_{i+1})} A_{i+1}\cdots \to B]$,\\
$t=[m,B \to \cdots B_{j} \xrightarrow{(B_{j},B_{j+1})} B_{j+1}\cdots \to C]$ , we have:
\begin{enumerate}
\item[$\bullet$] $t\otimes s=[n+m,A \to \cdots B \to \cdots C]$,
\item[$\bullet$] $[\A]_{AC}(t \otimes s)=\A(B_{m-1},C)\otimes \cdots \otimes \A(B,B_1) \otimes \A(A_{n-1},B) \otimes \cdots \otimes \A(A,A_1)$,
\item[$\bullet$] $[\A]_{BC}(t) \otimes [\A]_{AB}(s)=[\A(B_{m-1},C)\otimes \cdots \otimes \A(B,B_1)] \otimes [\A(A_{n-1},B) \otimes \cdots \otimes \A(A,A_1)]$.
\end{enumerate}
Then the map $\varphi(A,B,C)(t,s)$ is the unique isomorphism from $[\A]_{AC}(t \otimes s)$ to $[\A]_{BC}(t) \otimes [\A]_{AB}(s)$  given by the associativity of the bifunctor $\otimes$. This maps consists to move the parentheses from  the front to the desired places.\\

Using again the fact that the associativity of $\otimes$ is a natural isomorphism, we see that $\varphi(A,B,C)(t,s)$ is functorial in $t$ and $s$. 

Finally one checks that the functors $[\A]_{AB}$ together with the maps $\varphi(A,B,C)(t,s)$ and $\varphi_A =\Id_{I}$, satisfy the coherence axioms of morphism of bicategories.  Then $[\A]$ is a colax unitary \footnote{unitary means $\varphi_A$ is the identity for every object $A$.} homomorphism from $\P_{\ol{X}}$ to $\M$ as desired.
\end{proof}

\begin{obs} \
\renewcommand\labelenumi{\alph{enumi})}
\begin{enumerate}
\item We recover with this proposition the notion of \textbf{polyad} introduced by Bénabou in \cite{Ben2}. But Bénabou used \textbf{lax morphism} and here we use \textbf{colax homomorphism}. This comes from the fact that in the path-bicategory $\P_{\C}$ of a category $\C$ we took the map $[0,A]$ to be our strict identity in $\P_{\C}(A,A)$ rather than $[1,\Id_A]$.
\item As usual for every $A \in X$ the functor $[\A]_{AA}:\P_{\ol{X}}(A,A) \to \M$ becomes a monoidal functor with the composition.
\item For every $A$, we have a canonical  functor $\ol{A} : \bf{1} \to \ol{X}$ which consists to select the object $A$ and its identity morphism. This functor induces a strict homomorphism $\P_{\ol{A}}: \P_{\bf{1}} \to \P_{\ol{X}}$. 
Therefore any $\ol{X}$-point of $(\M,Iso(\M))$ induces a $1$-point $\P_{\ol{A}} \to \M$ of $(\M,Iso(\M))$, hence a monoid of $\M$ (Proposition \ref{hotmon}). The monoid in question  is simply the hom-object $\A(A,A)$ with the composition $c_{AAA}$ and unit map $I_A$.
\end{enumerate}
\end{obs}

\subsection{Category enriched over bicategories}\ \\
\indent In the previous examples we always followed the idea of Bénabou which consists to identify a monoidal category $\M$ with a bicategory with one object denoted\footnote{Some authors denote it by $\Sigma \M$ for the ``suspension'' of $\M$} again $\M$. Bénabou pointed out in \cite{Ben2} that it is possible to define category enriched over bicategory $\M$ in such a way that one recovers the classical theory of enriched categories over a monoidal category by taking $\M$ with one object. He named them \emph{polyads}.\\

In Proposition \ref{c-enr}  we've showed that a category enriched over a monoidal category  $\M$ ($i.e$ a polyad) was the same thing as an $\ol{X}$-point of $(\M, Iso(\M))$ where $X$ is the set of objects.\\ 

In what follows we're going to see that the same situation hold when $\M$ has many objects, which means that a category enriched over a bicategory $\M$ is the same thing as an $\ol{X}$-point of $(\M, Iso(\M))$. 

We recall first the definition of a $\M$-category for a general bicategory $\M$. 
\begin{df}
Let $\M$ be a bicategory. A $\M$-category $\A$ consists of the following data :
\begin{itemize}
\item for each object $U$ of $\M$, a set $\A_U$ of objects \emph{over} $U$;
\item for objects $A,B$ over $U,V$, respectively an arrow $\A(A,B): U \to V$ in $\M$;
\item for each object $A$ over $U$, a $2$-cell $I_A: \Id_U \Longrightarrow \A(A,A)$ :
\[
\xy
(30,0)*+{U}="U";
(0,0)*+{U}="W";
{\ar@/^1.7pc/^{\A(A,A)}"U";"W"};
{\ar@/_1.7pc/_{\Id_U}"U"-(2,0);"W"+(2,0)};
{\ar@{=>}^{I_A}(15,5);(15,-5)};
\endxy
\] in $\M$;
\item for object $A,B,C$ over $U,V,W$, respectively, a $2$-cell $c_{ABC} :\A(B,C)\otimes \A(A,B) \Longrightarrow \A(A,C)$ :
\[
\xy
(50,0)*+{U}="U";
(22,22)*+{V}="V";
(-10,0)*+{W}="W";
{\ar@/_1.7pc/_{\A(A,B)}"U"; "V"};
{\ar@/_1.7pc/_{\A(B,C)}"V"; "W"};
{\ar@/^1.7pc/^{\A(A,C)}"U";"W"};
{\ar@/_2.9pc/_{\A(B,C)\otimes \A(A,B)}"U"-(2,0);"W"+(2,0)};
{\ar@{=>}^{c_{ABC}}(20,10);(20,-5)};
\endxy
\] in $\M$;
\end{itemize}
satisfying the obvious three axioms of left and right identities and associativity.
\end{df}

The reader can immediately check that if $\M$ has one object we recover the enrichment over a monoidal category. 
We've followed here  the terminology \emph{`object over'} like in \cite{Str}. This provides a geometric vision in the theory of enriched categories which is very useful in some situations. 

\begin{rmk}
We will assume that all the sets $\A_U$ are nonempty, otherwise we replace $\M$ by its ``restriction'' to the set of $U$ such that $\A_U$ is nonempty. 
\end{rmk}

The following proposition is the generalization of Proposition \ref{c-enr}.

\begin{prop}\label{bicat-enr}
Let $\M$ be a bicategory and $\W$ be the class of invertible  $2$-cells of $\M$. We have an equivalence between the following data
\renewcommand\labelenumi{\roman{enumi})}
\begin{enumerate}
\item an $\ol{X}$-point of $(\M,\W)$ 
\item an enriched category over $\M$ having $X$ as set of objects $i.e$ a polyad of Bénabou.
\end{enumerate}
\end{prop}
\begin{proof}
The proof is exactly the same as that of Proposition \ref{c-enr}.\\

Using the same construction we can see directly that any $\ol{X}$-point of $(\M,\W)$ gives rise to a category enriched over $\M$. \\

Conversely given a $\M$-category $\A$, one take $X$ to be the disjoint union of the $\A_U$ : $X=\bigsqcup \A_U$.

We define a colax homomorphism $[\A] :\P_{\ol{X}} \to \M$ by sending the elements of $\A_U$ to $U$ and defining the components 
$[\A]_ {AB}: \P_{\ol{X}}(A,B) \to \M$ in the same way as in the proof of Proposition \ref{c-enr}.\\
For $t=[n,A \to \cdots A_{i} \xrightarrow{(A_{i},A_{i+1})} A_{i+1}\cdots \to B]$ in $\P_{\ol{X}}(A,B)$, with :\\
$A$ over $U$,\\
$A_i$ over $U_i$,\\
$A_{i+1}$ over $U_{i+1}$,\\
$B$ over $V$ , then  we set :
$[\A]_ {AB}(t)=\A(A_{n-1},B)\otimes \cdots \otimes \A(A_{i},A_{i+1})\otimes \cdots \otimes \A(A,A_1)$.\\
In particular we have :\\
$[\A]_ {AC}([2,(A,B,C)])= \A(B,C) \otimes \A(A,B)$,\\
$[\A]_ {AB}([1,(A,B)])= \A(A,B)$,\\
$[\A]_ {AA}([1,(A,A)])= \A(A,A)$,\\
$[\A]_ {AA}([0,A])=\Id_U$.\\

Here again we've chosen a representation of the composition of the $1$-cells $\A(A_{n-1},B)$, $\A(A_{i},A_{i+1})$, $\A(A,A_1)$, for example  all pairs of parentheses starting in front.\\

One proceed in the same way as in the proof of Proposition \ref{c-enr} to construct the homomorphism $[\A]$. We need to use the fact that the composition in $\M$ is a functor, and the associativity of the composition is a natural isomorphism.
\end{proof} 


\subsection{Segal categories} \ \\
\indent We recall that $\Delta^{+}$ is the ``topologists's category of simplices''.  We obtain $\Delta^{+}$ by removing the oject $0$ from $\Delta$. We rename the remaining objects by $0,1,2,...$ .\\ 

For a small category $\C$ we can associate functorially a simplicial set $\Nv(\C): (\Delta^{+})^{op} \to Set$, called \emph{nerve} of $\C$. The natural maps, called \emph{Segal maps} 
$$\Nv(\C)_k \to \Nv(\C)_1 \times_{\Nv(\C)_0} \cdots \times_{\Nv(\C)_0} \Nv(\C)_1$$ 
are isomorphisms.\\

Simpson and Hirschowitz \cite{HS} generalized this process to define inductively Segal n-categories. They defined first  a `category' $n$SePC of \emph{Segal $n$-précat} as follows.
\begin{enumerate}
\item[$\bullet$] A Segal $0$-précat is a simplicial set, hence a $\bf{1}$-point of $(Set,\times,1)$ in our terminology.
\item[$\bullet$] For $n \geq 1$, a Segal $n$-précat is a functor :
$$\A:(\Delta^{+})^{op} \to (n-1)\text{SePC}$$ 
such that $\A_0=\A(0)$ is a \underline{discrete} object of $(n-1)\text{SePC}$. Since $(n-1)\text{SePC}$ has finite products, we can formulate it again in term $\bf{1}$-point of  $(n-1)\text{SePC}$.
\item[$\bullet$] A \emph{morphism} of Segal $n$-précat is a natural transformation of functors.
\end{enumerate} 
These data define the category $n\text{SePC}$. They gave a notion of \emph{equivalence} in $n\text{SePC}$ and model structure on it.\\

Finally they define a Segal $n$-category to be a Segal $n$-précat $\A:(\Delta^{+})^{op} \to (n-1)\text{SePC}$ such that :
\begin{enumerate}
\item[$\bullet$] for every $k$, $\A_k$ is a Segal $(n-1)$-category ,
\item[$\bullet$] for every $k \geq 1$, the canonical maps 
$$\A_k \to \A_1 \times_{\A_0} \cdots \times_{\A_0} \A_1$$
are equivalences of Segal $(n-1)$-précats.
\end{enumerate}  

\begin{rmk}
These definitions involved the use of \textbf{discrete} objects. A discrete object in \cite{HS}, is by definition an object in the image of some fully faithful functor from $Set$ to $(n-1)\text{SePC}$. For a Segal $n$-category $\A$ the discret object $\A_0$  plays the role of ``set of objects''. We can see the analogy with the nerve of a small category.\\
It's important to notice that in the above definitions, one needs a notion of \textbf{fiber product} to define the Segal maps
$$\A_k \to \A_1 \times_{\A_0} \cdots \times_{\A_0} \A_1.$$  In fact $(n-1)\text{SePC}$ is a cartesian monoidal category.\\

One could interpret $\A$ as a \textbf{generalized nerve} of a category enriched over $((n-1)\text{SePC}, \times, 1)$ with an `internal set' of object $\A_0$. 
\end{rmk}
If we do not have a notion of \emph{discrete object} and a \emph{fiber product} we need to change the construction a little bit to define  generalized Segal categories. For this purpose, one needs a category $\M$ together with a class of \emph{homotopy equivalences} such that $(\M,\W)$ form a base of enrichment . We take the set of objects `outside' $\M$, to avoid the use of discrete objects, by introducing the set $X$.\\

The following definition is on the ``level $2$'' when $\M$ is bicategory. We will extend it later to the case where $\M$ is an $\infty$-category .  
\begin{df}
For $(\M,\W)$ a base of enrichment with  $\W$ a class of  homotopy equivalences. For any set $X$, an $\ol{X}$-point of $(\M,\W)$  will be called a \emph{Segal} $\M_{\W}$-category.
\end{df}

\begin{prop}\label{seg-ncat}
Let $\M= (n-1)\text{SePC}$ and $\W$ be the equivalences of Simpson and Hirschowitz in  \cite{HS}. Let $X$ be a nonempty set. Then we have an equivalence between the following data 
\begin{enumerate}
\item[$\bullet$] a Segal $n$-category $\A$ in the sense of Simpson-Hirschowitz, with $\A_0=X$
\item[$\bullet$] an $\ol{X}$-point $F$ of $(\M,\W)$, satisfying the induction hypothesis:\\ $F[p,(x_0,\cdots,x_p)]$ is a Segal $(n-1)$-category.  
\end{enumerate}
\end{prop}

\begin{proof}
Obvious.
\end{proof}

\subsection{Linear Segal categories}\ \\
\indent We fix $\M=(\bf{ChMod_R,\otimes_R,R})$ the monoidal category of (co)-chain complexes of $\bR$-modules for a commutative ring $\bR$.\ \\

\paragraph{Choice of the class of maps $\W$} 
\renewcommand\labelenumi{\alph{enumi})}
\begin{enumerate}
\item When working with a general commutative ring $\bR$ then we will take $\W$ to be the class of \textbf{chain homotopy equivalences}. \\
\item But if $\bR$ is a field we can take $\W$ to be the class of \textbf{quasi-isomorphisms}.
\end{enumerate}

\begin{rmk}
Leinster \cite{Lei2} pointed out that for a general commutative ring $\bR$, the quasi-isomorphisms may not be stable by tensor product because of the Künneth formula. 
\end{rmk}

\begin{df}\label{DG-Seg}
Let $X$ be a nonempty and  $\M=(\bf{ChMod_R,\otimes_R,R})$ together with $\W$ the suitable class of weak equivalences. 
A Segal DG-category is an $\ol{X}$-point of $(\M,\W)$, that is a $\W$-colax morphism :

$$F: \P_{\ol{X}} \to \M$$
\end{df}

\begin{rmk}\ \
\begin{itemize}
\item As one can see a strict Segal point of $(\M,W)$ is a classical DG-category. 
\item As usual we can use the iterative process \emph{à la} Simpson-Tamsamani by defining enrichment over $\M$-Cat with the suitable weak equivalences. In this way we can define also higher linear Segal categories. 
\end{itemize}
\end{rmk}

\subsection{Nonabelian cohomology}\label{nabcoh}
\subsubsection{$G$-categories}\label{g-cat} \ \\
\indent Bénabou \cite{Ben2} pointed out that we can use polyads to `pick up' a coherent family of isomorphisms satisfying \textbf{cocyclicity}. But polyads are enriched categories and correspond to strict Segal $\ol{X}$-points in our langage, it appear that \textbf{the cocyclicity conditions of torsors reflect a composition operation}. For example EG for a group G in $\bf{(Set,\times)}$ is a G-category in an obvious manner. The reader can find in \cite{Joy-Tier} an account on torsors.\ \\

We denote by BG the usual category having one object say, $\star$, and $\Hom(\star ,\star) = G$.  
\begin{df}
Let $X$ be a nonempty set. An $\ol{X}$-point $F : \P_{\ol{X}} \longrightarrow BG $  is called a G-category.
\end{df}
If we denote by $\M^{X}_{F}$ the corresponding category then we have for every pair $(a,b)$ of elements of $X$, an element $\M^{X}_{F}(a,b)$ of $G$. The composition is the identity and gives a cocyclicity condition \\
$\M^{X}_{F}(a,b) \centerdot \M^{X}_{F}(b,c)=\M^{X}_{F}(a,c)$ and $\M^{X}_{F}(a,a)=e$, where $e$ is the unit in G.\ \\

\begin{obs}\ \
\begin{itemize}
\item We've considered a group in the category of sets but we can generalize it to any group object using the functor of points. This will be an iterative process of enrichment, that is enrichment over the categories of G-Cat when G is group in $\bf{(Set,\times)}$. 
\item It follows immediately that any group homomorphism from G to H will take a G-category to an H-category. 
\item The geometric picture behind a G-category is the notion of G-bundle. Roughly speaking we want to consider each element of  $X$ as an opent set of some space and to consider $\M^{X}_{F}(a,b)$ as a transition function. We can then consider $F$ a \textbf{generic trivialisation}. When all the $\M^{X}_{F}(a,b)$ are equal to $e$, then our vector bundle (or local system) is trivial. 
\item From this observation we can define the characteristic classes of Segal path-object in general by suitably adapting the classical definitions. We will discuss in \cite{SEC2} when we will introduce the descent theory of path-objects. 
\end{itemize}
\end{obs}

\begin{term}
Let $\A$ be a small category. Following the terminology of Simpson, we will call \textbf{interior of $\A$} and denote by Int($\A$) the biggest groupoid  contained in $\A$. For a base $(\M,\W)$, we take the interior Int[$(\M,\W)$] to be the sub-bicategory whose underlying $1$-category is the interior of $\M_{\leq 1}$. 
\end{term}
 
\begin{df}
Let $\C$ be a small category. A $\C$-generic cohomological class in coefficient in $\M$ is a Segal path-object of \emph{Int}$[(\M,\W)]$.
\end{df} 

\begin{rmk}
When $\C= \ol{X}$ an $\ol{X}$-generic cohomological class is precisely a Segal $\M_{\W}$-category having $X$ as set of object and such that each $\Hom(A,B)$ is invertible. We can require also that each $\Hom(A,A)$ is \emph{contractible}.
\end{rmk}

\begin{ex}\label{exp-bc}
Let $Y$ be a complex manifold and $f : Y \to \Cx$ a function. 
\vspace*{0.2cm}
\\
\textbf{Case 1}. We consider $\Cx$ with the monoid structure $(\Cx,+,0)$ and view it as discrete monoidal category.\ \\

Let $Y_{\delta f}$ be the following  category enriched over $(\Cx,+,0)$.
\begin{itemize}
\item $Ob(Y_{\delta f})=$ set of \emph{points} of $Y$.
\item For every pair of points $(a,b)$ the hom-object is $Y_{\delta f}(a,b)= f(b)-f(a)$.
\item The composition is given by the cocyclicity: $Y_{\delta f}(a,b) + Y_{\delta f}(b,c)= Y_{\delta f}(a,c)$  
\item There is only the identity morphism between $a$ and $a$,  $Y_{\delta f}(a,a)=0$
\end{itemize}
We see through this example that the coboundary operation gives rise to a $(\Cx,+,0)$-category. But this category forget many informations of both $Y$ and $f$. This category reduces to a category having  as objects the fibers of $f$. \\
\vspace*{0.2cm}
\\
\textbf{Case 2}. With the previous category we form a new category by base change from $(\Cx,+,0)$ to $(\Cx^{\star},\times,1)$ using the exponential function. We denote this category enriched over $(\Cx^{\star},\times,1)$ by $e^{Y_{\delta f}}$.
\begin{itemize}
\item $Ob(e^{Y_{\delta f}})=$ set of \emph{points} of $Y$.
\item For every pair of points $(a,b)$ the hom-object is $e^{Y_{\delta f}}(a,b)= e^{f(b)-f(a)}$.
\item The composition is given by the `multiplicative' cocyclicity: $e^{Y_{\delta f}}(a,b) \times e^{Y_{\delta f}}(b,c)= e^{Y_{\delta f}}(a,c)$  
\item Here again there is only the identity morphism between $a$ and $a$,  $e^{Y_{\delta f}}(a,a)=1$
\end{itemize}

As we can see these data correspond to a line bundle over a `generic space' whose set of open covering has the same cardinality as the set of points of $Y$. Moreover any two members $a,b$ of this covering intersect necessarily in oder to have a $e^{Y_{\delta f}}(a,b)$. We can think for example that this is a covering of an irreducible component of some space with the Zariski topology.\ \\ 

Here again the categorical structure reduces to a category whose set of objects is the set of fibers of $f$. And each fiber gives a trivial category which correspond to a trivial line bundle over some space. In some cases the fiber $f^{-1}(0)$ is of interest e.g  Riemann \emph{zêta function}, divisor associated to a function, etc.\\  
\vspace*{0.2cm}
\\
\textbf{Case 3}. We consider again the category $Y_{\delta f}$ and form a new category $|Y_{\delta f}|$ having the same set of objects. 

\begin{itemize}
\item For every pair of points $(a,b)$ the hom-object is $|Y_{\delta f}|(a,b)= |f(b)-f(a)|$.
\item The composition is given by the \emph{triangle inequality}: $|Y_{\delta f}|(a,b) + |Y_{\delta f}|(b,c) \geq |Y_{\delta f}|(a,c)$  
\end{itemize}
We get then a metric space which the same thing as a category enriched over $(\ol{\R}_+,+,0, \geq)$.

This last case can be generalized as follows. \\

Given two spaces $X$ and $Y$ and a morphism $f :X \to Y$ we can transport any metric $d_Y$ on $Y$ to $X$ by pull back. We define $f^*d_Y$ by the obvious formula  $f^*d_Y(a,b)= d_Y(fa,fb)$.\\ 
\end{ex}

\begin{rmk}
If the space $X$ comes with a topology then in each construction we will have, an `atlas' of subcategories enriched over respectively $(\Cx,+,0)$, $(\Cx^{\star},\times,1)$ and $(\ol{\R}_+,+,0, \geq)$. Here we have an example of \emph{descent of relative enrichment}. \ \\

These basic observations will rise many questions  for general bases $\M$ and will help us to understand the structure of $\M$-Cat.  With the powerful language of higher categories we get new points of view on classical situations and new notions are created. 
\end{rmk}

\subsubsection{Parallel transport}\label{p-transp}\ \\
\indent In the following we give an example of $1$-functor which is viewed as an enrichment. We refer the reader to Schreiber-Waldorf \cite{Sch-Wal-transp}  and references therein for an account on \emph{parallel transport} with a guidance toward higher categories. \\

Let $M$ be a smooth manifold and $\E \to M$ a vector bundle equipped with a connection $\nabla$. The connection induces a  functor 
$$ \text{Tra} \nabla: \Pcal_{1}(M) \to \Vec $$
called `parallel transport functor'.\ \\

Here $\Pcal_{1}(M)$ is the \emph{Path-groupoid}  of $M$ (morphisms are thin-homotopy classes of smooth paths in $M$) and $\Vec$ is the category of vector spaces.

The functor sends each point $x$ of $M$ to its fiber $\E_x$, and each path $f:x \to y$, to  the \emph{parallel transport} \\  $\text{Tra} \nabla(f):\E_x \to \E_y$ induces by the connection along the path. \\

The relation with enriched categories comes when we view each point $x$ of $\Pcal_1(M)$ to be \textbf{over its fiber $\E_x$}.\\

In fact if we consider $\bf{Vect}$ as a bicategory, and even a strict $2$-category, with all the $2$-cells being identities (or \emph{degenerated}) we can ``lift'' the functor $$ \text{Tra} \nabla: \Pcal_{1}(M) \to \Vec $$  to  a \underline{strict} homomorphism  from the $2$-path-bicategory of $\Pcal_{1}(M)$ to $\Vec$ (see Observations \ref{obsPC1}). In our terminology this will be   \underline{strict} `free'  $\Pcal_1(M)$-point of $\Vec$ but we may prefer the terminoly $\Pcal_1(M)$-module in this situation. \ \\

The corresponding $\Pcal_1(M)$-module will be denoted $\E^{-1}$ and is  described as follows.\\
\begin{enumerate}
\item[$\bullet$] For every $x$ in $\Pcal_1(M)$,  $\E^{-1}(x)= \E_x$.
\item[$\bullet$] For every pair $(x,y)$ , the component $\E^{-1}_{xy}: \P_{\Pcal_1(M)}(x,y) \to \Vec$ is given by :\\
if $s=[n,x \to \cdots x_{i} \xrightarrow{f_i} x_{i+1}\cdots \to y]$ with each $f_i : x_i \to x_{i+1}$ a morphism in $\Pcal_1(M)$  then we set 
$$\E^{-1}(s):= \text{Tra} \nabla(f_{n-1}) \circ \cdots \circ \text{Tra} \nabla(f_i) \circ \cdots \circ \text{Tra} \nabla(f_0) .$$ We see that $\E^{-1}(s)$ is a linear map from $\E_x$ to $\E_y$. 
\item[$\bullet$] For $x=y$ , we have $\E^{-1}([0,x])= \Id_{\E_x}$.
\item[$\bullet$] for every $s,s'$ in $\P_{\Pcal_1(M)}(x,y)$, and any morphism $u: s \to s'$ then we define $\E^{-1}_{xy}(u)= \Id_{\E^{-1}(s)}$. This definition is well defined because we know that morphisms in $\P_{\Pcal_1(M)}(x,y)$ are generated by the morphisms of type $(\ast)$ and $(\ast \ast)$ as we saw in the proof of Proposition \ref{c-enr}. And one easily see that  the image of a morphism of type $(\ast)$ or $(\ast \ast)$ is the identity, therefore $\E^{-1}(s)= \E^{-1}(s')$.
\item[$\bullet$] Finally for every triple $(x,y,z)$ and every $(t,s)$ in $\P_{\Pcal_1(M)}(y,z) \times \P_{\Pcal_1(M)}(x,y)$ it's easy to see that $$\E^{-1}(t \otimes s)= \E^{-1}(t) \circ \E^{-1}(s).$$
\end{enumerate}
These data satisfy the coherences axioms and $\E^{-1}$ is a \emph{strict} $\Pcal_1(M)$-module (or $\Pcal_1(M)$-point) of $(\Vec, \Id_{\Vec})$. 
\begin{obs}\
\renewcommand\labelenumi{\alph{enumi})}
\begin{enumerate}
\item Since $\Pcal_1(M)$ is a groupoid, every morphism $f:x \to y$ is invertible therefore the induced map \\ $\text{Tra} \nabla(f):\E_x \to \E_y$   is invertible in $\Vec$. Taking $x=y$ we see that $\E^{-1}_{xx}$ is a \textbf{representation} of the (smooth) fundamental group $\pi_1(M,x)$.  Therefore studying $\C$-point  with $\C$ a groupoid becomes important to understand the homotopy of \emph{generalized spaces} $M$.\ \\
\item It's well known that if we consider \underline{flat} connection $\nabla$, then the functor  $\text{Tra} \nabla$ factor through $\Pcal_1(M)$, the fundamental groupoid of $X$. And we can still work in enriched category context.\ \\
\item The idea of thinking a vector bundle on $M$ as an enriched category extend our intuition which consists to `view' a category as a topological space (the classifying space). We can consider a vector bundle with a connection as \textbf{ a linear copy} \footnote{ this terminology matches with the expression `linear representation'} of our space $M$. A point $x$  is identify with the corresponding fiber $\E_x$ and every path from a point $x$ to a point $y$ gives a linear map by parallel transport.\ \\

\item Grothendieck defined the fundamental group in algebraic geometry as the group automorphism  of a fiber functor (see \cite{SGA1}).
This suggests to identify a point $x$ of a generalized space $M$ with  it's \emph{``motivic'' fiber functor} $\tx{Mot}(\omega_x)$ (to be defined). In our terminology we will view $x$ as being over (or taking as `copy') $\tx{Mot}(\omega_x)$. We will then have an enrichment over  the ``category of fiber functors''. Enrichment in this situation can be thought as \emph{giving a copy of $\C$} `of type $\M$'. \ \\
\item We see through out this example how enriched category theory appears in geometry and homotopy context. We saw that if we take $\bf{Vect}$ as our base of enrichment we have a ``linearization'' of the  $1$-homotopy type of $M$.
Now if want more informations on the higher homotopy, we need to replace $\Pi_1(M)$  by $\Pi_{\infty}(M)$ and $\bf{Vect}$ by another base which contains more informations, then doing \textbf{a base changes and base extensions}. 

One can take for example $\bf{SVect,nVect, ChVect, Perf }$, which are ,respectively, the category of simplicial vector spaces, $n$-vector spaces, complex of vector spaces,  perfect complexes. In these categories there is a notion of \emph{weak equivalence} , and we can consider Segal  $\C$-points (or $\C$-module). It appears that having a theory of Segal enriched categories becomes important.
\item A further step will be to consider the notion of \emph{gluing} Segal $\C_i$-points of $\M$ where $\C_i$ is a \emph{covering} of $\C$. This will be part of \cite{SEC2}. 
\end{enumerate}
\end{obs}

\subsection{Quasi-presheaf}\label{q-prsh}\ \\
\begin{df}
Let $\C$ be a small category and $(\M,\W)$ a base of enrichment with $\W$ a class of homotopy $2$-equivalences. A Segal $\M_{\W}$-presheaf in values in $\M$ is a Segal ${\C}^{op}$-point of $(\M,\W)$, that is a $\W$-colax morphism 
$$ \F : \P_{\C^{op}} \to \M.$$  
\end{df}

\begin{ex}\label{aff-alg}
Let's consider the Grothendieck anti-equivalence given by the `global section functor' : 
$$ \tx{Aff}^{~op} \xrightarrow{\Gamma(Z,\O_Z)} \tx{AlgCom}$$ 

We want to consider this functor as a quasi-presheaf which is a real presheaf taking its values in \textbf{Bim}. Recall that \textbf{Bim} is the bicategory described as follows. 
\begin{itemize}
\item Objects are rings : R, S,...
\item a $1$-morphism from R to S is a bimodule $_{S}M_{R}$, 
\item a $2$-morphism from $_{S}M_{R}$ to $_{S}N_{R}$ is a morphism of bimodule,
\item The composition is given by the obvious tensor product.
\end{itemize}
The reader can find a detailed description of \textbf{Bim} in the paper of Bénabou \cite{Ben2}. \ \\

Then the presheaf consists roughly speaking to send
\begin{itemize}
\item each $(Spec(R), \O_{Spec(R)})$ to $R$ 
\item each morphism of schemes $f :Spec(R) \to Spec(S)$ to the $(S,R)$-bimodule $\varphi^{\star}: S \nrightarrow R$, where $\varphi$ is the corresponding ring homomorphism given by the anti-equivalence. 
\end{itemize}
\end{ex}
\begin{note}
In a more compact way we obtain the presheaf using the `embedding' described in \cite{Ben2} from AlgCom to \textbf{Bim}.
\end{note}
\section{Morphisms of path-objects}\label{morph}
\subsection{Transformation}\ \\
\indent We recall briefly the notion of transformation between colax morphisms. 
\begin{df}{\emph{[Transformation]}} \ \\
\indent Let $\B$ and $\M$ be two bicategories and $F=(F,\varphi)$, $G=(G,\psi)$ be \textbf{two colax morphisms} from $\B$ to $\M$. A transformation $\sigma: F \to G$ 
\[
\xy
(0,0)*+{\B}="A";
(30,0)*+{\M}="C";
{\ar@/_1.3pc/_{G}"A";"C"};
{\ar@/^1.3pc/^{F}"A";"C"};
{\ar@{=>}_{\sigma}(15,4);(15,-4)};
\endxy.
\]
is given by the following data and axioms. \\

\emph{Data :}
\begin{itemize}
\item $1$-cells $\sigma_A: FA \to GA$ in $\M$ 
\item Natural transformations 
\[
\xy
(-15,25)*+{\B(A,B)}="A";
(-15,0)*+{ \M(FA,FB)}="B";
(45,25)*+{\M(GA,GB)}="C";
(45,0)*{\M(FA,GB)}="D";
{\ar@{->}^{G_{AB} }"A"; "C"};
{\ar@{->}_{F_{AB} }"A"; "B"};
{\ar@{->}^{-\otimes \sigma_A}"C"; "D"};
{\ar@{->}_{\sigma_B \otimes - }"B"+(10,-1); "D"+(-10,-1)};
{\ar@/_1.8pc/"A"+(2,-2); "D"};
{\ar@/^1.8pc/"A"+(8,-0.7); "D"};
{\ar@{=>}_{~~\sigma_{AB}}(10,10);(21,14)};
\endxy
\]
thus $2$-cells of $\M$,  $\sigma_t : \sigma_B \otimes Ft \to Gt \otimes \sigma_A$, for each $t$ in  $\B(A,B)$.
\end{itemize}
\emph{Axioms :} \\

The following commute :
\[
\xy
(-50,20)*+{\sigma_C \otimes F(t\otimes s)}="A";
(30,20)*+{G(t \otimes s) \otimes \sigma_A}="B";
(-50,0)*+{\sigma_C \otimes (Ft \otimes Fs)}="C";
(30,0)*+{(Gt \otimes Gs) \otimes \sigma_A}="D";
(-50,-15)*+{(\sigma_C \otimes Ft) \otimes Fs}="M";
(-23,-30)*+{(Gt \otimes \sigma_B)\otimes Fs}="N";
(11,-30)*+{Gt \otimes (\sigma_B \otimes Fs)}="O";
(30,-15)*+{Gt \otimes (Gs \otimes \sigma_A)}="P";
{\ar@{->}_{\Id \otimes \sigma_s}"O";"P"};
{\ar@{->}_{\sigma_t \otimes \Id}"M";"N"};
{\ar@{->}^{\sigma_{t \otimes s}}"A";"B"};
{\ar@{->}_{\Id \otimes \varphi}"A";"C"};
{\ar@{->}^{\psi \otimes \Id}"B";"D"};
{\ar@{->}_{a^{-1}}"C";"M"};
{\ar@{->}_{a^{-1}}"P";"D"};
{\ar@{->}^{a}"N";"O"};
{\ar@{.>}^{}"C";"D"};
\endxy
\]

\[
\xy
(10,20)*+{\sigma_A \otimes FI_A }="A";
(35,0)*+{\sigma_A}="M";
(60,20)*+{GI_{A} \otimes \sigma_A}="B";
(10,0)*+{\sigma_{A} \otimes I_{FA}}="C";
(60,0)*+{I_{GA} \otimes \sigma_A}="D";
{\ar@{->}_{r}^{\sim}"C";"M"};
{\ar@{->}_{l^{-1}}^{\sim}"M";"D"};
{\ar@{->}_{\Id \otimes \varphi_A}"A";"C"};
{\ar@{->}^{\psi_A \otimes \Id}"B";"D"};
{\ar@{->}^{\sigma_{I_A}}"A";"B"};
\endxy
\]
\end{df}

\begin{rmk}
When all the $1$-cells  $\sigma_A: FA \to GA$ are identities we will not represent them in the diagrams. 
\end{rmk}
\subsection{Morphism of path-objects}\ \\

\indent In this section we're going to define what is a morphism between $\C$-point and $\D$-point of $(\M,\W)$, for $\C$ and $\D$ two small categories, we call them \textbf{pré-morphisms}. We will  see in a moment  that the morphisms of  points of $(\M, \W)$ which are relevant to enrichment behave exactly as morphisms of vector bundle over $\M$ , which means \emph{fiber wise} compatible. This is not surprising because it only makes sense to speak about `morphism' between  enriched categories  having the same `type of enrichment'. When $\M$ has one object then this condition will be fulfilled but the morphisms we consider are more general than a classical morphisms between enriched categories.\ \\

Recall that for any category $\C$, by construction of $\P_{\C}$ we have $Ob(\C)=Ob(\P_{\C})$. Moreover any functor \\ $\Sigma: \C \to \D$ extends to a strict  homomorphism $\P_{\Sigma} :\P_{\C} \to \P_{\D}$. 

\begin{df}\label{premor}
Let $F:\P_{\C} \to \M$ and $G: \P_{\D} \to \M$ be respectively  two path-objects of $(\M,\W)$. An \textbf{$\M$-premorphism} from $F$ to $G$, is a pair $\Sigma=(\Sigma,\sigma)$ consisting of a functor $\Sigma : \C \to \D$ together with a transformation of morhphism of bicategories
$\sigma :  F  \longrightarrow G \circ \P_{\Sigma}$
\[
\xy
(-18,0)*+{\P_{\C}}="X";
(30,0)*+{\P_{\D}}="Y";
(9,-18)*+{\M}="E";
{\ar@{->}^{\P_{\Sigma}}"X";"Y"};
{\ar@{->}_{F}"X";"E"};
{\ar@{->}^{G}"Y";"E"};
{\ar@{=>}^{\sigma}(1,-7);(15,-7)};
\endxy
\]

An $\M$-premorphism is called an $\M$-morphism if all the $1$-cells $\sigma_A$ are identities. In particular if $A$ is over $U \in Ob(\M)$ then so is $\Sigma A$ (see \emph{Figure 4} below).  

\end{df}
\includegraphics[width=15cm,height=10cm]{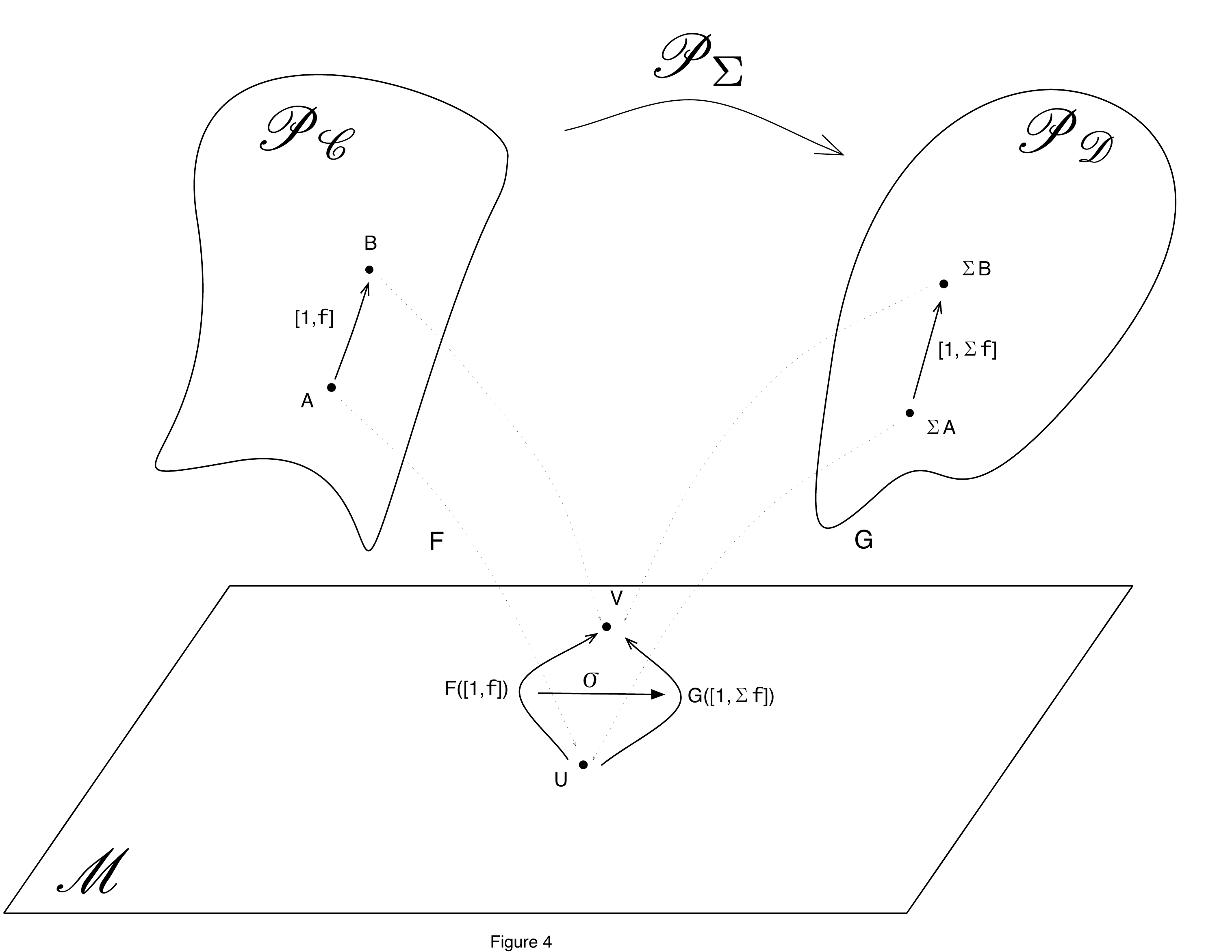}

\begin{obs}
For nonempty set $X$ and $Y$, if $F$  and $G$ are respectively  strict $\ol{X}$-point and $\ol{Y}$-point of $\M$, it's easy to check that an $\M$-morphism is exactly an $\M$-functor from $\M^{X}_{F}$ to $\M^{Y}_{G}$ in the classical sense.  
\end{obs}

\subsection{Bimodules}\label{bimod}
\begin{warn}
We remind the reader that the composition in an $\M$-category is presented here in this order :
 $$\C(B,C) \otimes \C(A,B) \to \C(A,C).$$ 
Then each $\C(A,B)$ is a $(\C(B,B), \C(A,A))$-bimodule with $\C(B,B)$ \textbf{acting on the left} and $\C(A,A)$ \textbf{on the right}.

But as one can see if the composition was presented as : $\C(A,B) \otimes \C(B,C) \to \C(A,C)$, then the action of $\C(B,B)$ would have been on the right. 
\end{warn}
\ \\
We saw that a monoid $T$ (or monad) in $(\M,\W)$ is given by a strict $\bf{1}$-point that is a homomorphism : 
$$T: \P_{\bf{1}} \to \M.$$

Let $\bf{2}$ be the posetal category described as follows. \\
$Ob(\bf{2})=\{0,1\}$ and
\begin{equation*}
\bf{2}(i,j) =
  \begin{cases}
     \{(i,j) \} & \text{if $i<j$} \\
     \{\Id_i=(i,i) \}  & \text{if $i=j$ }\\
     \varnothing & \text{if $i>j$ }
  \end{cases}
\end{equation*}
The composition is the obvious one.\\

We have two functors : $\bf{1} \xrightarrow{i_0} \bf{2}$ and $\bf{1} \xrightarrow{i_1} \bf{2}$. These functors induce by functoriality two functors $\P_{i_0}$ and $\P_{i_0}$ from $\P_{\bf{1}}$ to  $\P_{\bf{2}}$.

\begin{df}
Let $T_0$, $T_1$ be two Segal $\bf{1}$-points of $(\M,\W)$. A bimodule from $T_0$ to  $T_1$ is a Segal path-object 
$$ \Psi:\P_{\bf{2}} \to \M $$
such that  $\Psi \circ \P_{i_0}= T_0$ and $\Psi \circ \P_{i_1}= T_1$.
\end{df}

This definition has a natural generalization for every $\C$-point and $\D$-point of $(\M,\W)$. \ \\

\paragraph*{\textbf{The general case}} \ \\
\indent All allong this work we've always identified monoids with enriched categories with one object. Now for bimodules in $\M$ , e.g $\C(A,B)$, we want to identify them with \textbf{oriented enriched categories} having two objects. Here by `oriented' we mean that there may not be a hom-object between some pair of objects.\ \\

We saw previously, that in some cases, given a $\C$-point $F: \P_{\C} \to \M$ we want to identify $F$ with a generalized $\M$-category $\M^{\C}_{F}$. In the following we're going to express the classical notion of bimodule (also called distributor, profunctor or module) using path-objects. We will express everything in term of morphisms of path-object but one should keep in mind that these definitions generalize the classical ones.\ \\

Our idea to define  a bimodule in general  between a $\C$-point  and a $\D$-point  is to consider an $\E$-point,  where $\E$ contains both $\C$ and $\D$ together with an `order' in $\E$ between the objects of $\C$ and $\D$. This lead us to introduce the following. 

\begin{df}
Let $\C$ and $\D$ be two small categories. A \textbf{bridge} from $\C$ to $\D$ (resp. $\D$ to $\C$) is a category $\E$ equipped with two embedding \footnote{By `embedding' we mean injective on object and fully faithfull} functors 
$$\E_{|{\C}}: \C \to \E, \ \ \E_{|{\D}}: \D \to \E$$
such that for every $A$ in $Ob(\C)$ and $B$ in $Ob(\D)$ we have
$\E(B,A)= \varnothing$ (resp. $\E(A,B)= \varnothing$). \footnote{We've identified $A$ with $\E_{|{\C}}(A)$ and $B$ with $\E_{|{\D}}(B)$}\ \\

A morphism of bridges is a functor $\beta : \E \to \G$ such that : $\beta \circ \E_{|{\C}} = \G_{|{\C}}$ and  
$\beta \circ \E_{|{\D}} = \G_{|{\D}}$.\\

A bridge $\E$ is said to be \textbf{rigid} if $Ob(\E) \cong Ob(\C) \coprod Ob(\D)$. 
\end{df}
\begin{ex}\ \
\renewcommand\labelenumi{\alph{enumi})}
\begin{enumerate}
\item The first example is the previous category $\bf{2}$ which is a bridge from $\bf{1}$ to $\bf{1}$.\ \\

\item In the following we're going to construct the `thin' bridge between any small categories. 

Let's denote by $\C \prec \D$ the small category described as follows. \\

We take $Ob(\C \prec \D) = Ob(\C) \coprod Ob(\D)$. \\

\begin{equation*}
[\C \prec \D](A,B) =
  \begin{cases}
     \C(A,B) & \text{if $(A,B) \in Ob(\C) \times Ob(\C)$} \\
     \D(A,B) & \text{if $(A,B) \in Ob(\D) \times Ob(\D)$} \\
     \{(A,B)\} \cong 1  & \text{if $(A,B) \in Ob(\C) \times Ob(\D)$ }\\
     \varnothing & \text{if $(A,B) \in Ob(\D) \times Ob(\C)$ }
  \end{cases}
\end{equation*}
\ \\
The composition is given by the following rules.
\begin{itemize}
\item For $A$ in $Ob(\C)$ and $B$ in $Ob(\D)$ ,  if $O$ is an object of $\C$ then the composition $c_{OAB}$ is the constant (unique) function which sends every pair $[f,(A,B)]$ to $(O,B)$.
\item Similary if $P$ is an object of $\D$, then the composition $c_{ABP}$ is the constant function which sends every pair $[(A,B),g]$ to $(A,P)$.
\item The restriction of the composition to $\C$ ( resp. to $\D$)  is the original one. 
\end{itemize}
\ \\
One easily check that $\C \prec \D$ is a category and we have two canonical embeddings : $i_{\C}: \C \to (\C \prec \D)$ and  $i_{\D}: \D \to (\C \prec \D)$.\ \\   

\end{enumerate}
\end{ex}
\begin{rmk}\ \
It's easy to see that $(\C \prec \D)$ is the terminal rigid bridge. In some cases depending of the base $\M$ we will only consider this terminal rigid bridge.
\end{rmk}
\begin{nota}
We will denote by $\P_{\C \hookrightarrow \E}$ and $\P_{\D \hookrightarrow \E}$ the induced embeddings on the $2$-path-categories. 
\end{nota}
Bridges and classical bimodules (or distributors) are connected by the following proposition. The reader can find an account on distributors in \cite{Ben-dis},\cite{borceux_1}, \cite{Law}, \cite{Str}. 
\begin{prop}
We have an equivalence between the following data.
\begin{itemize}
\item A distributor $\X : \D \to \widehat{\C}$
\item A rigid bridge $\E$ from $\C$ to $\D$ 
\end{itemize}
This equivalence is an equivalence of categories.
\end{prop}\label{bridge}
\begin{proof}[\scshape{Sketch of proof}]
\item[\textbf{Step 1}.]Given a bridge $\E$ from $\C$ to $\D$ one define the associated distributor $\X(\E): \D \to \widehat{\C}$ by the functor of points  $\X(\E)(D):=\Hom(\E_{|{\C}}(-),D)$ and similary on morphisms $\X(\E)(f):=\Hom(\E_{|{\C}}(-),f)$. \ \\
\\
\item[\textbf{Step 2}.]Conversely given a distributor $\X : \D \to \widehat{\C}$, we define the associated bridge $\E(\X)$ as follows.\ \\  

Set  $Ob(\E(\X)) = Ob(\C) \coprod Ob(\D)$. \ \\

The restriction of $\E$ to $Ob(\C)$ (resp. $Ob(\D)$) is equal to $\C$ (resp. $\D$) and 
for $A$ in $Ob(\C)$ and $D$ in $Ob(\D)$ we take $\E(A,D):= \X(D)(A)$. \ \\

We define the composition in the following manner.
\begin{itemize}

\item For a triple of object $(A,A',D)$ with $A,A'$ in $Ob(\C)$ and $D$ in $Ob(\D)$, the composition function is given by $$c_{AA'D} : \X(D)(A') \times \C(A,A') \to \X(D)(A)$$ 

which sends each element $(a,f)$  of $\X(D)(A') \times \C(A,A')$ to $\X(D)(f)a$.\ \\
\item Similary given $D$, $D'$ two objects of $\D$ and $A$ an object of $\C$ then
$$c_{ADD'} : \D(D,D') \times \X(D)(A) \to \X(D')(A)$$ 
sends an element $(g,b)$ of $\D(D,D') \times \X(D)(A)$ to $\X(g)_A(b)$, where $\X(g)_A$ is the component at $A$ of the natural transformation $\X(g) : \X(D) \to \X(D')$.  
\end{itemize}
\end{proof}
\begin{df}
Let $F:\P_{\C} \to \M$ and $G: \P_{\D} \to \M$ be respectively  two Segal $\C$-point and $\D$-point of $(\M,\W)$ and $\E$ a rigid bridge from $\C$ to $\D$, 
\begin{itemize}
\item An $\E$-$(G,F)$-bimodule  $\Psi: G  \nrightarrow  F$ is a Segal $\E$-point of $(\M,\W)$
$$\Psi : \P_{\E} \to \M $$
satisfying the `boundary conditions': $\Psi \circ \P_{\C \hookrightarrow\E} =F$   and $\Psi \circ \P_{\D \hookrightarrow \E} = G$ \ \\
\item Given $\Psi_1$, $\Psi_2$ two $\E$-$(G,F)$-bimodules, a morphism of bimodules  from $\Psi_1$ to $\Psi_2$ is an $\M$-morphism $(\Id_{\E}, \Theta)$ which induces the identity on both $F$ and $G$.\ \\
\item More generally, let $\E_1$ and $\E_2$  be two rigid bridges  from $\C$ to $\D$ and $\Psi_1$ (resp. $\Psi_2$) be an  $\E_1$-$(G,F)$-bimodule  (resp. $\E_2$-$(G,F)$-bimodule). 

A morphism of $(G,F)$-bimodules from $\Psi_1$ to $\Psi_2$ is an $\M$-morphism 
$$\Sigma =(\Sigma,\sigma) : \Psi_1  \to \Psi_2$$ 
such that the induced morphism from $\Psi_1$ to $\Sigma^{\star}\Psi_2$ is a morphism of $\E_1$-$(G,F)$-bimodules. Here $\Sigma^{\star}\Psi_2$ is the obvious pullback of $\Psi_2$ along $\Sigma$.
\end{itemize}
\end{df}

\begin{obs}\ \
\renewcommand\labelenumi{\alph{enumi})}
\begin{enumerate}
 \item To understand what's really happening in this definition it suffices to write it when $\C =\ol{X}$, $\D =\ol{Y}$, $\E=(\ol{X} \prec \ol{Y})$ and $F$, $G$ and $\Psi$ are respectively  strict $\ol{X}$-point, $\ol{Y}$-point and  $(\ol{X} \prec \ol{Y})$-point of $\M$. \ \\

Let $\Psi= _{G}\Psi_{F}$ be an $(\ol{X} \prec \ol{Y})$-strict point of $\M$. Given a pair $(P,Q)$ of objects of $\ol{X}$ and an object $R$ of $\ol{Y}$, we have by definition of $\Psi$ the following span of the same type as the ones which give the composition in both $\M^{X}_{F}$ and  $\M^{Y}_{G}$
\[
\xy
(-10,0)*+{\Psi_{QR}([1,(Q,R)]) \otimes \Psi_{PQ}([1,(P,Q)])}="X";
(50,0)*+{\Psi_{PR}([1,(P,R)])}="Y";
(25,2)*+{c_{PQR}}="c";
(22,22)*+{\Psi_{PR}([2,(P,Q,R)])}="Z";
{\ar@{->}_{\varphi(P,Q,R)}"Z"; "X"};
{\ar@{->}^{~~~~\Psi_{PR} \{[2,(P,Q,R)] \xrightarrow{!} [1,(P,R)] \} }"Z"; "Y"};
{\ar@{.>}"X";"Y"};
\endxy
\]
\ \\
And the condition $\Psi \circ \P_{i_{\C}}=F$ says that  $\Psi_{PQ}([1,(P,Q)])= F_{PQ}([1,(P,Q)])=\M^{X}_{F}(P,Q)$ and we have a map $$c_{PQR} : \Psi_{QR}([1,(Q,R)]) \otimes \M^{X}_{F}(P,Q) \to \Psi_{PR}([1,(P,R)]).$$ 
\ \\
Similary if we take one object $Q$ in $\ol{X}$ and two objects $R,S$ in $\ol{Y}$, we will have a map 
$$c_{QRS} : \Psi_{QR}([1,(Q,R)]) \otimes \M^{Y}_{G}(R,S) \to \Psi_{QS}([1,(Q,S)]).$$ 

It's clear that these data together with unity and the associativity coherences contained in the definition of $\Psi$ give a bimodule (also called distributor, profunctor or module) from $\M^{X}_{F}$ to $\M^{Y}_{G}$ in the classical sense.\ \\
\item We have the classical fact any $\M$-morphism $\Sigma=(\Sigma,\sigma)$  from $F$ to $G$ induces two bimodules : 
$\Sigma^{\star}: F \nrightarrow G$ and  $\Sigma_{\star} : G  \nrightarrow  F$, see  \cite{Ben-dis}, \cite{borceux_1} \cite{Law} \cite{Str}, for a description.
\end{enumerate}
\end{obs}

\begin{rmk}\ \
\renewcommand\labelenumi{\alph{enumi})}
\begin{enumerate}
\item We can define the classical operations such as composite or `tensor product' of a $\E$-$(G,F)$-bimodule by another  $\E'$-$(F,D)$-bimodule but the existence of such $(G,D)$-bimodule will involve some cocompleteness conditions on the hom-categories in $\M$. The idea consists to consider the `composite bridge' which is given by the composite of the corresponding distributors and define a path-object satisfying the `boundary conditions'. \\
\item With this composite we can define a category of `enriched distributors' in a suitable manner. We will come back to this when we will give a model structure in \cite{SEC2}.
\end{enumerate}
\textbf{For the moment we will assume that $\M$ is `big and good' enough to have all these operations}. We will denote by $\M$-Dist the bicategory described as follows.
\begin{itemize}
\item Objects are Segal path-objects $(\C,F)$
\item Morphisms are Bimodules
\item $2$-morphisms are morphism of bimodules.  
\end{itemize}
\end{rmk}

\paragraph*{\textbf{Presheaves on path-object}} \ \\
\indent For a given Segal path-object $F :\P_{\C} \to \M$, denote by $\M^{\C}_{F}$ the corresponding generalized Segal $\M$-category. In the following we give the definition of the analogue of a presheaf on $\M^{\C}_{F}$, that is functor from $(\M^{\C}_{F})^{op}$ to $\M$. When $\M$ is $(\bf{ChVect,\otimes_k,k})$, then $\M^{\C}_{F}$ will be a generalized DG-category, and a functor $\A :(\M^{\C}_{F})^{op} \to \M$ is sometimes called `$\A$-DG-module' or simply $\A$-module. So in general we may call such a functor an $\M$-module, like in \cite{Str}.

\begin{nota}\ \
\renewcommand\labelenumi{\alph{enumi})}
\begin{enumerate}
\item We will denote by $\eta$ the `generic object' of $\M$, which consist to select an object $U$ of $\M$ with it's identity arrow $\Id_U$.  We have $$\eta_U: \P_{\bf{1}} \xrightarrow{[U, \Id_U]} \M$$ 
 which express $\Id_U$ as the trivial monoid. We will identify $\eta_U$ with $U$. When $\M$ has one object, hence a monoidal category there is only one generic object. $\eta_U$ is sometimes denoted simply $U$ or $\hat{U}$.   
\item For an object $A$ of $\C$, we have the canonical distributor $h_A: \bf{1} \to \widehat{\C}$ which consists to select the functor of points $\C(-,A)$. We will denote by $\E(\C)_A$ the associated bridge from $\C$ to $\bf{1}$ given by Proposition \ref{bridge}.

Taking  $\bf{1}=\{\star, \Id_{\star}\}$, $\E(\C)_A$ is described as follows.\ \\
\begin{itemize}
\item[-] $Ob(\E(\C)_A)= Ob(\C) \coprod \{\star\}$ \ \\
\item[-] For every $B$ in $Ob(\C)$ we define $\E(\C)_A(B,\star):=\C(B,A)$ and , $\E(\C)_A(\star, B):= \varnothing$ \\ 
\item[-] We take $\E(\C)_A(\star,\star)= \{\Id_{\star}\} $\\
\end{itemize}  
The composition is the obvious one and we check easily that $\E(\C)_A$ is a rigid bridge from $\C$ to $\bf{1}$.\\
\item In general a distributor $\bf{1} \to \widehat{\C}$ consists precisely to select an object of $\widehat{\C}$, say $\Zcal$, and we will denote by $\E_{\Zcal}$ the corresponding bridge from $\C$ to $\bf{1}$. \\
\end{enumerate} 
\end{nota} 

\begin{df}
Let $F :\P_{\C} \to \M$ be a Segal path-object. We denote by $\Pcal F$ the category described as follows. 
\renewcommand\labelenumi{\alph{enumi})}
\begin{enumerate}
\item Objects are $(\eta_U,F)$-bimodules i.e Segal path-object $\Psi : \P_{\E(\eta_U)} \to \M $ with  $\E(\eta_U)$ a rigid bridge from $\C$ to $\bf{1}$, such that the `boundary conditions' are satisfied: \\
$\Psi \circ \P_{\C \hookrightarrow \E(\eta_U) } =F$   and $\Psi \circ \P_{\bf{1} \hookrightarrow \E(\eta_U)} = \eta_U$.\\
\item For $\Psi_U$, $\Psi_V$ respectively in $\M$-Dist$(\eta_U,F)$, $\M$-Dist$(\eta_V,F)$, a morphism  $\alpha :\Psi_U \to \Psi_V$ when it exists, is the object of $\M$-Dist$(\eta_U,\eta_V)=\M(U,V)$ who represents the functor 
$$ \M\tx{-Dist}[\Psi_V \otimes - , \Psi_U] : \M\tx{-Dist}(\eta_U,\eta_V) \to \M\tx{-Dist}(\eta_U,F) $$
\end{enumerate}
\end{df}

\begin{rmk}\ \
\begin{enumerate}
\item We have a relative $F$-Yoneda functor $\Y_F : \C \to \Pcal F$ who sends an object $A$ of $\C$ to a  path-object $\Y_{F,A}: \P_{\E(\C)_A} \to \M$ $\in \M\tx{-Dist}(\eta_{FA},F)$, described as follows.

Recall that here $\E(\C)_A$ is the rigid bridge from $\C$ to $\bf{1}$ obtained by the distributor $h_A: \bf{1} \to \widehat{\C}$. 
To define $\Y_{F,A}$ we need to specify the image of a chain $[1,P \xrightarrow{\gamma} \star]$ which generated the other chains of $\E(\C)_A$ ending at $\star$. 

But for this it suffices to specify only for $P=A$, because the morphisms in $\E(\C)_A$ between $P$ and  $\star$ are generated by $\C(P,A)$ and the morphism between $A$ and $\star$. But in some sense we can think $\star$ `as' a copy of $A$, which means that $A$ has a \emph{`multiplicity'}. 

So to define the path-object $\Y_{F,A}$ we need to remove the discrepancy between the actions of $A$ and $\star$. We do it by sending every chain $[1,A \xrightarrow{\gamma} \star]$ to the identity arrow $\Id_{FA}$. More generally for a chain $[n,s]$ ending at $\star$, we take the image of $[n,s]$ to be the image of $[n',s']$, where $s'$ is the `longest' chain ending at $A$ contained in $s$.\\
\item  When $\C= \ol{X}$ and $F$ is a strict path-object, then $\Y_{F,A}$ is just the classical Yoneda functor, see for example \cite{Str}.  
\end{enumerate}
\end{rmk}
\subsection{Base Change and Reduction}\label{b-change}

\begin{df}
Given two bases of enrichment $(\M_1,\W_1)$, $(\M_2,\W_2)$, a morphism of bases is a homomorphism of bicategories $\L : \M_1 \to  \M_2$ such that $\L(\W_1) \subseteq  \W_2$.\\
Then if  $(\C,F)$ is a point of $(\M_1,\W_1)$ it follows immediately that $(\C,\L \circ F)$  is a point of $(\M_2,\W_2)$.
This operation is called \textbf{base change along $\L$}.
\end{df}

\begin{prop}\label{reduc}
Let $(\M,\W)$ be a base of enrichment. There exists a bicategory  $\W^{-1}\M$ together with a homomorphism $\L_{\W} : \M \to \W^{-1}\M$ such that:
\renewcommand\labelenumi{\alph{enumi})}
\begin{enumerate}
\item $\L_{\W}$ makes $\W$ invertible,
\item any homomorphism $ \Phi : \M \to \B$  which makes $\W$ invertible factor as  $\Phi = \ol{\Phi} \circ \L_{W}$ with $$\ol{\Phi}: \W^{-1}\M \to \B$$ a homomorphism.
\item $\W^{-1}\M$ is unique up to a biequivalence. 
\end{enumerate}
\end{prop}

\begin{proof}
See Appendix \ref{loc-bicat}.
\end{proof}

\begin{df}
Let $(\M,\W)$ be a base of enrichment and $\L_{\W} : \M \to \W^{-1}\M$ a localization. For any Segal point $(\C,F)$ of  $(\M,\W)$ the couple 
$(\C, \L_{\W}\circ F)$ is called  \emph{a reduction} of $(\C,F)$. It's a strict Segal point of $\W^{-1}\M$.
\end{df}

\appendix
\section{Review of the notion of bicategory}
\subsection{Definitions}
\begin{df}
A small bicategory $\C$ is determined by the following data:
\begin{itemize}
\item a nonempty set of objects $\underline{\C} = Ob(\C)$
\item a category $\C(A,B)$ of arrows for each pair  $(A,B)$ of objects of $\C$
\item a composition functor $c(A,B,C) :\C(B,C) \times \C(A,B) \longrightarrow \C(A,C)$ for each triple $(A,B,C)$ of objects of $\C$ 
\item an identity arrow $I_{A}: 1 \longrightarrow  \C(A,A)$  for any object $A$ of $\C$ 
\item for each quadruple $(A,B,C,D)$ of objects of $\C$ a natural isomorphism $a(A,B,C,D)$, called \textbf{associativity isomorphism}, between the two composite functors bounding the diagram :

\[
\xy
(-15,25)*+{\C(C,D) \times \C(B,C) \times \C(A,B)}="A";
(-15,0)*+{ \C(B,D) \times \C(A,B)}="B";
(45,25)*+{\C(C,D) \times \C(A,C) }="C";
(45,0)*{\C(A,D)}="D";
{\ar@{->}^{~~~~\Id \times c(A,B,C) }"A"; "C"};
{\ar@{->}_{c(B,C,D) \times \Id }"A"; "B"};
{\ar@{->}^{c(A,C,D)}"C"; "D"};
{\ar@{->}_{~~~~c(A,B,D)}"B"; "D"};
{\ar@/_1.5pc/"A"; "D"};
{\ar@/^1.5pc/"A"; "D"};
{\ar@{=>}_{~~a(A,B,C,D)}(10,10);(21,14)};
\endxy
\]

Explicitely :\
\[
a(A,B,C,D):c(A,B,D) \circ (c(B,C,D) \times \Id)  \to c(A,C,D) \circ (\Id \times c(A,B,C) ) 
\]

\

Then if $(h,g,f)$ is an object of $\C(C,D) \times \C(B,C) \times \C(A,B)$, the isomorphism, component of $a(A,B,C,D)$ at $(h,g,f)$ will be abbreviated into $a(h,g,f)$ or even $a$ :
\[
a=a(h,g,f)=a(A,B,C,D)(h,g,f):
\xy
(0,0)*+{(h \otimes g)  \otimes f}="A";(30,0)*+{h \otimes (g \otimes f)}="B";
{\ar@{->}^{\sim}"A" ;"B"};
\endxy
\]

\item for each pair $(A,B)$ of objects of $\C$, two natural isomorphisms $l(A,B)$ and $r(A,B)$ called \textbf{left} and \textbf{right} identities, between the functors bounding the diagrams:

\[
\xy
(-40,20)*+{ 1 \times \C(A,B) }="A";
(-5,0)*+{\C(A,B)}="B";
(-40,0)*+{\C(B,B) \times \C(A,B) }="C";
{\ar@{->}^{\sim}"A"; "B"};
{\ar@{->}_{ I_{B} \times \Id}"A"; "C"};
{\ar@{->}_{~~~~~~c(A,B,B)}"C"; "B"};
{\ar@{=>}_{l(A,B)}(-28,4);(-23,9)};
(30,20)*+{\C(A,B) \times 1}="X";
(65,0)*+{\C(A,B)}="Y";
(30,0)*+{  \C(A,B) \times \C(A,A) }="Z";
{\ar@{->}^{\sim}"X"; "Y"};
{\ar@{->}_{ \Id \times I_{A} }"X"; "Z"};
{\ar@{->}_{~~~~~~c(A,A,B)}"Z"; "Y"};
{\ar@{=>}_{r(A,B)}(42,4);(47,9)};
\endxy
\]

If $f$ is an object of $\C(A,B)$, the isomorphism, component of $l(A,B)$ at $f$ 
\[ l(A,B)(f):
\xy
(0,0)*+{I_{B} \otimes f}="A";
(20,0)*+{f}="B";
{\ar@{->}^{\sim}"A" ;"B"};
\endxy
\]
is abbreviated into $l(f)$ or even $l$, and similary we write 
\[
r= r(f)=r(A,B)(f):
\xy
(0,0)*+{f \otimes I_{A}}="A";
(20,0)*+{f}="B";
{\ar@{->}^{\sim}"A" ;"B"};
\endxy
\]

\end{itemize}
The natural isomrphisms $a(A,B,C,D)$, $l(A,B)$ and $r(A,B)$ are furthermore required to satisfy the following axioms : 

\begin{itemize}
\item[(A. C.):] Associativity coherence :\
If $(k,h,g,f)$ is an object of  $~\C(D,E) \times \C(C,D) \times \C(B,C) \times \C(A,B)$ the following 
diagram commutes :
\[
\xy
(-20,20)*+{((k\otimes h) \otimes g) \otimes f}="A";
(30,20)*+{(k \otimes (h \otimes g)) \otimes f}="B";
(-20,0)*+{(k \otimes h) \otimes (g \otimes f)}="C";
(30,0)*+{k \otimes ((h \otimes g) \otimes f)}="D";
(5,-15)*+{k \otimes (h \otimes (g\otimes f))}="E";
{\ar@{->}^{a(k,h,g) \otimes \Id}"A";"B"};
{\ar@{->}_{a(k \otimes h,g,f)}"A";"C"};
{\ar@{->}^{a(k,h \otimes g,f)}"B";"D"};
{\ar@{->}_{a(k,h,g \otimes f) }"C";"E"};
{\ar@{->}^{\Id \otimes a(h,g,f)}"D";"E"};
\endxy
\]

\item[(I. C.):] Identity coherence :\
If $(g,f)$ is an object of $\C(B,C) \times \C(A,B)$, the following diagram commutes :
\[
\xy
(-20,0)*+{(g \otimes I_{B}) \otimes f}="C";
(30,0)*+{ g \otimes (I_{B} \otimes f)}="D";
(5,-15)*+{g\otimes f}="E";
{\ar@{->}^{a(g,I_{B},f)}"C";"D"};
{\ar@{->}_{r(g) \otimes \Id }"C";"E"};
{\ar@{->}^{\Id \otimes l(f)}"D";"E"};
\endxy
\]

\end{itemize}

\end{df}
\begin{var}
When all the natural isomorphisms $a,l,r$ are \emph{identities} then $\C$ is said to be a \emph{strict} $2$-\emph{category}
\end{var}
Classically objects of $\C$ are called $0$-\emph{cells}, those of each $\C(A,B)$ are called $1$-$cells$ or $1$-\emph{morphisms} and arrows between $1$-\emph{morphisms} are called $2$-\emph{cells} or $2$-\emph{morphisms}. 
\begin{itemize}
\item In each $\C(A,B)$ :
\begin{itemize}
\item[$\ast$] every $1$-\emph{cell} $f$ has an identity $2$-\emph{cell} :
\[
\xy
(0,0)*+{A}="A"; (20,0)*+{B}="B";
{\ar@/^1pc/^{f}"A";"B"};
{\ar@/_1pc/_{f}"A";"B"};
{\ar@{=>}^{1_{f}}(10,3);(10,-3)};
\endxy
\]
\item[$\ast$]we have a \emph{vertical} composition of $2$-\emph{cells}: `$- \star -$'
\[
\xy
(0,0)*+{\xy
(0,0)*+{A}="A";
(15,0)*+{B}="B";
{\ar@/^1.9pc/^{~~~f}"A";"B"};
{\ar@/_1.9pc/_{~~~h}"A";"B"};
{\ar@{->}^{}"A";"B"};
{\ar@{=>}_{\alpha}(7.5,7);(7.5,1)};
{\ar@{=>}_{\beta}(7.5,-1);(7.5,-7)};
\endxy}="X";
(30,0)*+{\xy
(0,0)*+{A}="A"; (20,0)*+{B}="B";
{\ar@/^1.4pc/^{~~~f}"A";"B"};
{\ar@/_1.4pc/_{~~~h}"A";"B"};
{\ar@{=>}_{\beta \star \alpha}(10,4);(10,-4)};
\endxy}="Y";
{\ar@{~>}"X";"Y"};
\endxy
\]
\end{itemize}
\item In the \emph{composition functor} we have:
\begin{itemize}
\item[$\ast$]  a \emph{classical} composition of $1$-\emph{cells}: `$- \otimes -$'
\[
\xy
(-5,0)*+{\xy
(0,0)*+{C}="C";
(15,0)*+{B}="B";
(30,0)*+{A}="A";
{\ar@{->}^{f}"A";"B"};
{\ar@{->}^{g}"B";"C"};
\endxy}="X";
(30,0)*+{\xy
(0,0)*+{C}="C";
(15,0)*+{A}="A";
{\ar@{->}^{g \otimes f}"A";"C"};
\endxy}="Y";
{\ar@{~>}"X";"Y"}:
\endxy
\]
\item[$\ast$] a \emph{horizontal} composition of $2$-\emph{cells}: `$- \otimes -$'
\[
\xy
(-10,0)*+{\xy
(0,0)*+{C}="C";
(15,0)*+{B}="B";
(30,0)*+{A}="A";
{\ar@/_1.3pc/_{f}"A";"B"};
{\ar@/^1.3pc/^{f'}"A";"B"};
{\ar@/_1.3pc/_{g}"B";"C"};
{\ar@/^1.3pc/^{g'}"B";"C"};
{\ar@{=>}_{\beta}(7.5,4);(7.5,-4)};
{\ar@{=>}_{\alpha}(22.5,4);(22.5,-4)};
\endxy}="X";
(30,0)*+{\xy
(0,0)*+{C}="C";
(20,0)*+{A}="A";
{\ar@/_1.3pc/_{g \otimes f}"A";"C"};
{\ar@/^1.3pc/^{g' \otimes f'}"A";"C"};
{\ar@{=>}_{\beta \otimes \alpha}(10,4);(10,-4)};
\endxy}="Y";
{\ar@{~>}"X";"Y"}:
\endxy 
\]
\[
(\beta \otimes \alpha)(g \otimes f) =\beta(g) \otimes \alpha (f)= g' \otimes f'
\]
\end{itemize}
\end{itemize}

\begin{ex}{[Bénabou]}\label{mon} 
Let $(\M, \otimes,I,\alpha,\lambda,\rho)$ be a \emph{monoidal} category. We define a bicategory $\widetilde{\M}$ by:
\begin{itemize}
\item[-] $Ob(\widetilde{\M})= \{ \bigstar \}$
\item[-] $\widetilde{\M}(\bigstar,\bigstar)= \M$
\item[-] $c(\bigstar,\bigstar,\bigstar)= \otimes$
\item[-] $I_{\bigstar}=I $
\item[-] $a(\bigstar,\bigstar,\bigstar,\bigstar)= \alpha$
\item[-] $l(\bigstar,\bigstar)= \lambda$
\item[-] $r(\bigstar,\bigstar)= \rho$
\end{itemize}
We easily check that the isomorphisms $a,l,r$ satisfy the (A.C.) and (I.C.) axioms since $\alpha,\lambda,\rho$ satisfy the \emph{associativity} and \emph{identities} axioms of a monoidal category. Conversely every bicategory with one object ``is'' a monoidal category.
\end{ex}
More generally we have:
\begin{prop}
Let $\C$ be a bicategory and $A$ an object of $\C$, then $ \otimes = c(A,A,A)$, $I=I_{A}$, $ \alpha=a(A,A,A,A)$, $\lambda=l(A,A)$, $\rho=r(A,A)$ determine a \emph{monoidal structure} on the category $\C(A,A)$.
\end{prop}

\subsection{Morphisms of bicategories}
\begin{df}{\emph{[Lax morphism]}}
Let $\B=(\underline{\B},c,,I,a,l,r)$ and $\C=(\underline{\C},c',I',a',l',r')$ be two small bicategories. A lax morphism $F=(F,\varphi)$ from $\B$ to $\C$ is determined by the following:
\begin{itemize}
\item A map $F : \underline{\B} \longrightarrow \underline{\C}$, $A \leadsto FA$ 
\item A family functors
$$F_{AB}= F(A,B): \B(A,B) \to \C(FA,FB),$$
$$ f \leadsto Ff , \alpha \leadsto F\alpha$$
\item For each object $A$ of $\B$ an arrow of $\C(FA,FA)$ (i.e a $2$-\emph{cell} of $\C$) :
$$ \varphi_{A} : I'_{FA} \to F(I_{A})$$
\item A family of natural transformations :
$$\varphi(A,B,C) : c'(FA,FB,FC) \circ(F_{BC} \times F_{AB}) \to F_{AC} \circ c(A,B,C)$$
\[
\xy
(-15,25)*+{\B(B,C) \times \B(A,B)}="A";
(-15,0)*+{ \C(FB,FC) \times \C(FA,FB)}="B";
(45,25)*+{\B(A,C)}="C";
(45,0)*{\C(FA,FC)}="D";
{\ar@{->}^{~~~~ c(A,B,C) }"A"; "C"};
{\ar@{->}_{ F_{BC} \times F_{AB} }"A"; "B"};
{\ar@{->}^{F_{AC}}"C"; "D"};
{\ar@{->}_{~~~~c'(FA,FB,FC)}"B"; "D"};
{\ar@/_1.5pc/"A"; "D"};
{\ar@/^1.5pc/"A"; "D"};
{\ar@{=>}_{~~\varphi(A,B,C)}(10,10);(21,14)};
\endxy
\]
If $(g,f)$ is an object of $\B(B,C) \times \B(A,B)$, the $(g,f)$-\emph{component} of $\varphi(A,B,C)$ 

$$Fg \otimes Ff \xrightarrow{\varphi(A,B,C)(g,f)} F(g \otimes f)$$
shall be usually abbreviated to $\varphi_{gf}$ or even $\varphi$.

\emph{These data are required to satisfy the following coherence axioms:}
\item[(M.1):] If $(h,g,f)$ is an object of $\B(C,D) \times \B(B,C) \times \B(A,B)$ the following diagram, where indices $A,B,C,D$ have been omitted, is commutative:
\[
\xy
(-55,20)*+{(Fh\otimes Fg) \otimes Ff }="A";
(-10,20)*+{(F(h \otimes g))\otimes Ff}="M";
(30,20)*+{F((h \otimes g) \otimes f)}="B";
(-55,0)*+{Fh\otimes (Fg \otimes Ff)}="C";
(30,0)*+{F(h\otimes (g \otimes f))}="D";
(-10,0)*+{Fh\otimes (F(g \otimes f))}="N";
{\ar@{->}^{\varphi_{hg} \otimes \Id}"A";"M"};
{\ar@{->}^{\varphi_{(h\otimes g)f}}"M";"B"};
{\ar@{->}_{a'(Fh,Fg,Ff)}"A";"C"};
{\ar@{->}^{Fa(h,g,f)}"B";"D"};
{\ar@{->}^{\Id \otimes \varphi_{gf}}"C";"N"};
{\ar@{->}^{\varphi_{h(g\otimes f)}}"N";"D"};
\endxy
\]
\item[(M.2):] If $f$ is an object of $\B(A,B)$ the following diagrams commute:
\[
\xy
(10,20)*+{Ff \otimes I'_{FA}}="A";
(35,20)*+{Ff\otimes FI_{A}}="M";
(60,20)*+{F(f \otimes I_{A})}="B";
(10,0)*+{Ff}="C";
(60,0)*+{Ff}="D";
{\ar@{->}_{~\Id \otimes \varphi_{A}  }"A";"M"};
{\ar@{->}_{~~\varphi_{fI_{A}}}"M";"B"};
{\ar@{->}_{r'(Ff)}"A";"C"};
{\ar@{->}^{Fr(f)}"B";"D"};
{\ar@{=}^{}"C";"D"};
\endxy
\xy
(10,20)*+{I'_{FB} \otimes Ff  }="W";
(35,20)*+{FI_{B}\otimes Ff }="T";
(60,20)*+{F(I_{B} \otimes f )}="X";
(10,0)*+{Ff}="Y";
(60,0)*+{Ff}="Z";
{\ar@{->}_{\varphi_{B} \otimes \Id   }"W";"T"};
{\ar@{->}_{~~\varphi_{I_{B}f}}"T";"X"};
{\ar@{->}^{l'(Ff)}"W";"Y"};
{\ar@{->}^{Fl(f)}"X";"Z"};
{\ar@{=}^{}"Y";"Z"};

\endxy
\]

\end{itemize}

\end{df}

\begin{var}\ \\
\begin{enumerate}
\item[$\bullet$] We will say that $F=(F,\varphi)$ is a \textbf{colax morphism} if $\varphi(A,B,C)$ and $\varphi_{A}$ are in the \textbf{opposite sense} i.e 
$$Fg \otimes Ff \xleftarrow{\varphi(A,B,C)(g,f)} F(g \otimes f)$$ $$I'_{FA} \xleftarrow{\varphi_{A}} F(I_{A})$$
and all of the horizontal arrows in the diagrams of \emph{(M.1)} and \emph{(M.2)} are in the opposite sense.
\item[$\bullet$] If $\varphi(A,B,C)$ and $\varphi_{A}$ are \textbf{natural isomorphisms}, so that 
$Fg \otimes Ff \xrightarrow{\sim} F(g \otimes f)$ and $I'_{FA} \xrightarrow{\sim} F(I_{A})$ then $F=(F,\varphi)$ is called a \textbf{homomorphism}.
\item[$\bullet$] If $\varphi(A,B,C)$ and $\varphi_{A}$ are \textbf{identities}, so that $Fg \otimes Ff=F(g \otimes f)$ and $I'_{FA}= F(I_{A})$ then 
$F=(F,\varphi)$ is called a \textbf{strict homomorphism}.
\end{enumerate}

\end{var}
\section{The $2$-Path-category of a small category}

Let $\C$ be a small category. For any pair $(A,B)$ of objects such that $\C(A,B)$ is nomempty, we build \textbf{from the composition operation and its properties} a simplicial diagram as follows :\\
\begin{enumerate}
\item[$\star$] If $A \neq B$:
\[
\xy
(-75,0)*++{\C(A,B)}="X";
(-40,0)*++{\coprod \C(A,A_{1})\times \C(A_{1},B)}="Y";
(16,0)*++{\coprod \C(A,A_{1}) \times \C(A_{1},A_{2})\times \C(A_{2},B)}="Z";
(47,0)*++{\cdots};
{\ar@{->}"Y";"X"};
{\ar@{.>}"Y";"Z"};
{\ar@<-1.0ex>@{.>}"X";"Y"};
{\ar@<1.0ex>@{.>}"X";"Y"};
{\ar@<-1.0ex>"Z";"Y"};
{\ar@<1.0ex>"Z";"Y"};
{\ar@<-1.8ex>@{.>}"Y";"Z"};
{\ar@<1.8ex>@{.>}"Y";"Z"};
\endxy
\]

\item[$\star$] If $A =B$:

\[
\xy
(-65,0)*++{\C(A,A)}="X";
(-30,0)*++{\coprod \C(A,A_{1})\times \C(A_{1},A)}="Y";
(26,0)*++{\coprod \C(A,A_{1}) \times \C(A_{1},A_{2})\times \C(A_{2},A)}="Z";
(-93,0)*+{\{A\} \cong 1}="O";
(57,0)*++{\cdots};
{\ar@{->}"Y";"X"};
{\ar@{.>}"Y";"Z"};
{\ar@{->}^{~~~1_{A}}"O";"X"};
{\ar@<-1.0ex>@{.>}"X";"Y"};
{\ar@<1.0ex>@{.>}"X";"Y"};
{\ar@<-1.0ex>"Z";"Y"};
{\ar@<1.0ex>"Z";"Y"};
{\ar@<-1.8ex>@{.>}"Y";"Z"};
{\ar@<1.8ex>@{.>}"Y";"Z"};
\endxy
\]

\end{enumerate}

\ \\

Here the dotted arrows correspond to add an identity map of an object and  the normal arrows correspond to replace composable pair of arrows by their composite.\\ 

In each case the diagram ``represents''  a functor which is a \textbf{cosimplicial set}: 
\begin{itemize}
\item[$\star$] If $A \neq B$:  $\P_{AB} : \Delta^{+} \to \tx{Set}$,
\item[$\star$] If $A=B$: $\P_{AA} : \Delta \to \tx{Set}$.
\end{itemize}
\vspace*{1cm}
\begin{obs}\ \\
\renewcommand\labelenumi{\alph{enumi})}
\begin{enumerate}
\item Here $\P_{AB}(n)$ is the set of $n$-simplices of the \emph{nerve of $\C$}, with extremal vertices $A$ and $B$ :
$$\P_{AB}(n)= \coprod_{(A=A_{0},...,A_{n}=B)} \C(A_{0},A_{1}) \times \cdots \times \C(A_{n-1},A_{n})$$ in particular we have : 
$\P_{AB}(1)= \C(A,B)$.

\item If $A=B$ and for $n=0$, $\P_{AA}(0)$ has a unique element which is identified with the object $A$.  
\item We will represent an element  $s$ of $\P_{AB}(n)$ as a $n$-tuple 
$$s=(A_{0}  \to A_{1}, \cdots ,A_{i} \to A_{i+1}, \cdots, A_{n-1} \to A_{n})$$
or as an \emph{oriented graph} 
$$s =A_{0}  \to A_{1} \to \cdots \to A_{i} \to A_{i+1} \to \cdots \to A_{n-1} \to A_{n}.$$
\item For a map $u :n \to m$  of $\Delta$,  $\P_{AB}(u) : \P_{AB}(n) \to \P_{AB}(m)$ is a function which sends a $n$-simplex to a $m$-simplex. \\ 
Such function corresponds to :
 \begin{itemize}
 \item (one or many) insertions of identities if $n < m$ 
 \item (one or many) compositions at some vertices if $n > m$. 
 \end{itemize}
\end{enumerate}
\end{obs}

\begin{term}
An element of $\P_{AB}(n)$  will be called a \emph{path or chain of length $n$} from $A$  to $B$. When $A=B$ we will call \emph{loops of lentgh $n$} the elements of $\P_{AA}(n)$. In particular there is a unique path of length $0$, which is identified with the object $A$.
\end{term}

We can rewrite the simplicial digrams above as :
\[
\xy
(-20,0)*++{\P_{AB}(1)}="X";
(5,0)*++{\P_{AB}(2)}="Y";
(30,0)*++{\P_{AB}(3)}="Z";
(39,0)*++{\cdots};
{\ar@{->}"Y";"X"};
{\ar@{.>}"Y";"Z"};
{\ar@<-1.0ex>@{.>}"X";"Y"};
{\ar@<1.0ex>@{.>}"X";"Y"};
{\ar@<-1.0ex>"Z";"Y"};
{\ar@<1.0ex>"Z";"Y"};
{\ar@<-1.8ex>@{.>}"Y";"Z"};
{\ar@<1.8ex>@{.>}"Y";"Z"};
\endxy
\]

\[
\xy
(-40,0)*+{\P_{AA}(0)}="O";
(-20,0)*++{\P_{AA}(1)}="X";
(5,0)*++{\P_{AA}(2)}="Y";
(30,0)*++{\P_{AA}(3)}="Z";
(39,0)*++{\cdots};
{\ar@{->}"Y";"X"};
{\ar@{.>}"Y";"Z"};
{\ar@<-1.0ex>@{.>}"X";"Y"};
{\ar@<1.0ex>@{.>}"X";"Y"};
{\ar@<-1.0ex>"Z";"Y"};
{\ar@<1.0ex>"Z";"Y"};
{\ar@<-1.8ex>@{.>}"Y";"Z"};
{\ar@<1.8ex>@{.>}"Y";"Z"};
{\ar@{.>}^{1_{A}}"O";"X"};
\endxy
\]

\begin{df}{\emph{[Concatenation of paths]}}\ \\
\indent Given $s$ in $\P_{AB}(n)$ and $t$ in $\P_{BC}(m)$ 
      $$s =A  \to A_{1} \to \cdots \to A_{i} \to A_{i+1} \to \cdots \to A_{n-1} \to B$$ 
      $$t =B \to B_{1} \to \cdots \to B_{j} \to B_{j+1} \to \cdots \to B_{m-1} \to C$$ 
we define the \textbf{concatenation} of $t$ and $s$ to be the element of $\P_{AC}(n+m)$ :
$$s \ast t := \underbrace{A  \to A_{1} \to \cdots \to A_{n-1} \to }_{s}B\underbrace{ \to B_{1} \to  \cdots \to B_{m-1} \to C}_{t}.$$ 
\end{df}

\begin{obs}\ \\
\indent From the definition it follows immediately that for any $n$ and for any $s \in \P_{AB}(n)$ we have :
\begin{enumerate}
\item[$\bullet$] $s \ast B = s$,
\item[$\bullet$] $A \ast s= s$ .
\end{enumerate}
\end{obs}
\vspace*{1cm}
\paragraph*{\textbf{The Grothendieck construction}}\ \\
\indent In the following we're going to apply the Grothendieck construction to the functors $\P_{AB}$, $\P_{AA}$. \\

For any pair of objects $(A,B)$ we denote by  $\P_{\C}(A,B)$ the \emph{category of elements} or the \emph{Grothendieck integral} of the functor $\P_{AB}$ described as follows.\\
\begin{itemize}
\item The objects of $\P_{\C}(A,B)$ are pairs $[n,s]$, where $n$ is an object of $\Delta$ and $s \in \P_{AB}(n)$.
\item A morphism $[n,s] \xrightarrow{u} [m,t]$ in $ \P_{\C}(A,B)$ is a map $ u: n\longrightarrow m $ of $\Delta$ such that image of $u$ by $\P_{AB}$  sends $s$ to $t$ :
$$\P_{AB}(u) : \P_{AB}(n) \longrightarrow \P_{AB}(m) $$ 
and $$\P_{AB}(u)s=t.$$
\end{itemize}
\begin{obs}\ \\
\indent We have a forgetful functor $\le_{AB}$ which makes each $\P_{\C}(A,B)$ a category over $\Delta$ (or $\Delta^{+}$) :
$$\le_{AB}: \P_{\C}(A,B) \to \Delta$$
with $\le_{AB}([n,s])= n $ and  $\le_{AB}([n,s] \xrightarrow{u} [m,t]) = u$.\\
The functor $\le_{AB}$ will be called \textbf{length}.
\end{obs}

\begin{rmk}
\begin{itemize}
\item[$\lozenge$] The concatenation of paths is a functor. For each  triple $(A,B,C)$ of objects of $\C$ we denote by $c(A,B,C)$ that functor: 
\[
\xy
(-10,15)*+{c(A,B,C): \P_{\C}(B,C) \times \P_{\C}(A,B)}="X";
(40,15)*+{\P_{\C}(A,C)}="Y";
(-1,0)*+{\left[\xy
(-10,0)*+{\left(\vcenter{ \xy
(0,13)*+{[n',s']}="A";
(0,0)*+{[m',t']}="B";
{\ar@{=>}^{u'}"A";"B"};
\endxy } \right),}="M";
(8,0)*+{\left(\vcenter{ \xy
(0,13)*+{[n,s]}="A";
(0,0)*+{[m,t]}="B";
{\ar@{=>}^{u}"A";"B"};
\endxy } \right)}="N";\endxy \right]}="E";
(46,0)*+{\left(\vcenter{ \xy
(0,13)*+{[n+n',s \ast s']}="A";
(0,0)*+{[m+m',t \ast t']}="B";
{\ar@{=>}^{u+u'}"A";"B"};
\endxy } \right)}="W";
{\ar@{->}"X";"Y"};
{\ar@{|->}"E";"W"};
\endxy
\]
\item[$\lozenge$] It's easy to check that the concatenation is strictly assosiactif.
\end{itemize}
\end{rmk}

\begin{nota}
We will use the following notations:

$s' \otimes s := c(A,B,C)(s',s) =  s \ast s' $,\\

$t' \otimes t := c(A,B,C)(t',t) = t \ast t' $ and \\

$u' \otimes u := c(A,B,C)(u',u) =  u+u' $\\

\end{nota}

Now we've set up all the tools needed  for the definition of the $2$-path-category.

\begin{df}{\emph{[$2$-path-category]}}
Let $\C$ be  a small category. The \emph{$2$-path-category} $\P_{\C}$ of $\C$ is the bicategory given by the following data:

\begin{itemize}
\item the objects of $\P_{\C}$ are the objects of $\C$
\item for each pair  $(A,B)$ of objects of $\P_{\C}$, the category of arrows of  $\P_{\C}$ is the category $\P_{\C}(A,B)$ described above
\item for each triple $(A,B,C)$ the composition functor is given by \emph{the concatenation functor}  described in the remark above:
$$c(A,B,C) :\P_{\C}(B,C) \times \P_{\C}(A,B) \longrightarrow \P_{\C}(A,C)$$
\item for any object $A$ of $\C$ we have a \underline{\emph{strict} identity} arrow $I_{A}: 1 \longrightarrow  \P_{C}(A,A)$ which is $[0,A]$ 
\item for each quadruple $(A,B,C,D)$ of objects of $\C$ the associativity natural isomorphism $a(A,B,C,D)$ is \underline{the identity}
\item the left and right identities natural isomorphisms are \underline{the identity} for each pair $(A,B)$ of objects of $\C$
\end{itemize}
These data satisfy clearly the \emph{Associativity and Identity Coherence axioms} (A. C.)  and (I. C.) so that $\P_{\C}$ is even a \underline{strict $2$-category}.

\end{df}

\begin{obs}\ \\
\indent Let $\C$ and $\D$ be two small categories  and $F$ a functor $F :\C  \to \D$. By definition $F$ commutes with the compositions of $\C$ and $D$, sends composable arrows of $\C$ to composable arrows of $\D$ and sends identities to identities. We can then easily see that  $F$ induces a strict homomorphism $\P_{F} :\P_{\C} \to \P_{\D}$. That is we have a functor:
\[
\xy
(0,8)*+{\P_{[-]}: \tx{Cat}_{\leq 1}}="X";
(30,8)*+{\tx{Bicat}}="Y";
(3,0)*+{\C \xrightarrow{F} \D}="E";
(35,0)*++{\P_{\C} \xrightarrow{\P_{F}} \P_{\D}}="W";
{\ar@{->}"X";"Y"};
{\ar@{|->}"E";"W"};
\endxy
\]

where $\tx{Cat}_{\leq 1}$ and Bicat are respectively the $1$-category of small categories and the category of bicategories.
\end{obs}

\begin{rmk}
\begin{itemize}
\item[$\diamondsuit$] There exists another `$2$-path-category' associated to a small category $\C$ which is \underline{not a strict} $2$-category but a bicategory.
This new bicategory $\overline{\P}_{\C}$ is obtained by removing the object $[0,A]$ from each $\P_{\C}(A,A)$.The identity morphism in $\overline{\P}_{\C}(A,A)$ is now the object $\Id_{A}=[1,A \xrightarrow{1_{A}} A]$ and is  \underline{not} a strict identity.

One has the \underline{right  identity} property using the the equalities  $  f \circ 1_{A} =f $ for any arrow $f$ in $\C(A,B)$. In fact we have a \underline{canonical invertible} $2$-cell  for every $1$-morphism $t =[n, A \xrightarrow{f_{1}} A_{1} \to  \cdots \to A_{n-1} \to B]$ in $\overline{\P}_{\C}(A,B)$:

\[
\xy
(20,0)*+{\left(\vcenter{ \xy
(0,13)*+{c(t, \Id_{A})}="A";
(0,0)*+{t}="B";
{\ar@{=>}^{u}"A";"B"};
\endxy } \right)}="W";
\endxy= 
\xy
(20,0)*+{\left(\vcenter{ \xy
(0,13)*+{[n+1,A\xrightarrow{1_{A}}A \xrightarrow{f_{1}} A_{1} \to \cdots \to  \cdots \to A_{n-1} \to B]}="A";
(0,0)*+{[n, A \xrightarrow{f_{1}} A_{1} \to  \cdots \to A_{n-1} \to B]}="B";
{\ar@{=>}^{u}"A";"B"};
\endxy } \right)}="W";
\endxy  
\]

given by the map $u: n+1 \to n$ of $\Delta_{+}$:
\[
\xy
(20,0)*+{\left(\vcenter{ \xy
(0,13)*+{n+1}="A";
(0,0)*+{n}="B";
{\ar@{->}^{u}"A";"B"};
\endxy } \right)}="W";
\endxy= 
\xy
(20,0)*+{\left(\vcenter{ \xy
(0,13)*+{\{0<1<\cdots <n-1<n\}}="A";
(0,0)*+{\{0<1<\cdots <n-1\}}="B";
{\ar@{->}^{u}"A";"B"};
\endxy } \right)}="W";
\endxy 
\]
with

\[
u(0)=u(1)=0,~~u(j)=j-1 ~~ for~~ 2 \leq j \leq n
\]

The inverse of these $2$-cells are given by the map $\overline{u} :n \to n+1$ of $\Delta$ :
\[
\xy
(20,0)*+{\left(\vcenter{ \xy
(0,13)*+{n+1}="A";
(0,0)*+{n}="B";
{\ar@{->}^{\overline{u}}"B";"A"};
\endxy } \right)}="W";
\endxy= 
\xy
(20,0)*+{\left(\vcenter{ \xy
(0,13)*+{\{0<1<\cdots <n-1<n\}}="A";
(0,0)*+{\{0<1<\cdots <n-1\}}="B";
{\ar@{->}^{\overline{u}}"B";"A"};
\endxy } \right)}="W";
\endxy 
\]
with $\overline{u}(j)=j+1$ for all $0 \leq j \leq n-1$.

In the same manner we get the \underline{left identity} property using  the maps $v :n+1 \to n$ and
$\overline{v}:n \to n+1 $of $\Delta$ defined respectively by : 
\[
v(j)=j ~~ for~~ 0 \leq j \leq n-2, ~~ v(n-1)=v(n)=n-1
\]
\[
\overline{v}(j)=j ~~ for~~ 0 \leq j \leq n-1
\]

\end{itemize}
\end{rmk}
\section{Localization and cartesian products}

\begin{nota} In this section we will use the following notations.\\ 
Cat $=$ the category of small categories.\\
$\Hom(\C,\E)=$ category of functors from $\C$ to $\E$.\\
$\L_{\S}: \C \to \C[\S^{-1}]=$ a Gabriel-Zisman localization of $\C$ with respect to a class of maps $\S$.\\
${\L_{\S}}^{*}(\E) = \Hom(\C[\S^{-1}],\E) \xrightarrow{-\circ \L_{\S}} \Hom(\C,\E)$.\\ 
$\Hom_{\S}(\C,\E)=$the full subcategory of $\Hom(\C,\E)$ whose objects are functors which make $\S$ invertible \footnote{We say that $F:\C \to \E$ makes $\S$ invertible if for all $s \in \S$, $F(s)$ is invertible in $\E$.} in $\D$.\\
\end{nota}

\begin{note}
It is well known that every functor making $\S$ invertible, factors in a unique way through $\L_{\S}$, hence ${\L_{\S}}^*$ induces an isomporphism of categories:
$${\L_{\S}}^{*}(\E): \Hom(\C[\S^{-1}],\E) \xrightarrow{\sim} \Hom_{\S}(\C,\E).$$
\end{note}
\textbf{Our goal is to prove the following lemma.}
\begin{lem}\label{loc-prod}
Let $\C$ and $\D$ be two small categories, $\S$ and $\T$ be respectively two class of morphisms of $\C$ and $\D$. Choose  localizations $\L_{\S}: \C \to \C[\S^{-1}]$ and $\L_{\T}: \D \to \D[\T^{-1}]$.

Assume that :
\renewcommand\labelenumi{\roman{enumi})}
\begin{enumerate}
\item[$\bullet$] $\S$ contains all identities of $\C$ 
\item[$\bullet$] $\T$ contains all  identities of $\D$
\end{enumerate}
Then the canonical functor  $$\C \times \D \xrightarrow{\L_{\S} \times \L_{\T}} \C[\S^{-1}] \times \D[\T^{-1}]$$ is a localization of $\C \times \D$ with respect to $\S \times \T$.
\end{lem}

\begin{obs}\label{obs-loc-prod} \ \\
From the lemma we have the following consequences.
\begin{enumerate}
\item Any  object $F$ of $\Hom_{\S \times \T}(\C \times \D,\E)$ factors uniquely as $F = \ol{F} \circ (\L_{\S} \times      \L_{\T})$ where  $\ol{F}$ is an object of $\Hom(\C[\S^{-1}] \times \D[\T^{-1}],\E)$.
\item  We have an isomorphism :
$${\L_{\S \times \T}}^{*}(\E): \Hom(\C[\S^{-1}] \D[\T^{-1}] \times ,\E) \xrightarrow{\sim} \Hom_{\S \times \T}(\C \times \D,\E).$$
\item For every pair  $(\E_1,\E_2)$ of categories and any functors $F$ in $\Hom_{\S}(\C,\E_1)$,  $G$ in $\Hom_{\T}(\D,\E_2)$, if we write :\\
$F= \ol{F} \circ \L_{\S}$,\\
$G= \ol{G} \circ \L_{\T}$,\\
then the functor   $F \times G$ is in $\Hom_{\S \times \T}(\C \times \D,\E_1 \times \E_2)$ and factors (uniquely) as:
$$ F \times G = (\ol{F} \times \ol{G}) \circ (\L_{\S} \times \L_{\T}).$$
We therefore have `` $\ol{F \times G} = \ol{F} \times \ol{G}$ ''. 
\end{enumerate}
\end{obs}
\ \\
For the proof of the lemma we will use the following:
\begin{itemize}
\item $\Hom : \tx{Cat}^{op} \times \tx{Cat} \to \tx{Cat}$ is a bifunctor,
\item Cat is \textbf{symmetric closed} for the cartesian product, 
\item the universal properties of the Gabriel-Zisman localization.
\end{itemize}

\paragraph{\textbf{`Cat is symmetric closed'}} The fact that Cat is symmetric closed means that for every category $\B$ the endofunctor  $-\times \B : \tx{Cat} \to \tx{Cat}$ (and also `$\B\times-$')  has a right adjoint :
$$\Hom(\B,-): \tx{Cat} \to \tx{Cat}.$$ 
\ \\
The adjunction says that the following functor is an isomorphism:
\[
\xy
(-20,10)*+{\alpha: \Hom(\A \times \B,\E)}= "X";
(30,10)*+{\Hom(\A,  \Hom(\B,\E))}="Y";
(-20,0)*+{F:\A \times \B \to \E }= "Z";
(30,0)*+{\alpha(F): \A \to \Hom(\B, \E).}="W";
{\ar@{->}"X";"Y"};
{\ar@{|->}"Z";"W"};
\endxy 
\]

We give an explicit description of $\alpha(F)$ hereafter.\ \\

\paragraph{\textbf{On objects}} For $A$ in $Ob(\A)$, $[\alpha(F)A]$ is the functor $ F(A,-): \B \to  \E $, given by the formula:
\begin{itemize}
\item $[\alpha(F)A]B :=F(A,B)$
\item $[\alpha(F)A](B \xrightarrow{f} B'):= F(Id_A,f)$, 
for every $B$ in $Ob(\B)$ and every arrow $B \xrightarrow{f} B'$ of $\B$.
\end{itemize}
\ \\
To see that  $[\alpha(F)A]$ is indeed a functor it suffices to check that 
$[\alpha(F)A](g\circ f)=[\alpha(F)A](g)\circ [\alpha(F)A](f)$ . \\

But this follows immediately from the functoriality of $F$ :\\ 
\begin{equation*}
\begin{split}
[\alpha(F)A](g\circ f) & :=F(Id_A,g\circ f) \\
& = F(Id_A \circ Id_A,g\circ f) \\
& = F(Id_A,g) \circ F(Id_A,f)\\
& =[\alpha(F)A](g)\circ [\alpha(F)A](f).
\end{split}
\end{equation*}

\paragraph*{\textbf{On morphisms}} For a morphism $A \xrightarrow{h} A'$ of $\A$, $[\alpha(F)h]$ is the natural transformation from $[\alpha(F)A] $ to $[\alpha(F)A']$ whose components are : 
$$[\alpha(F)h]_B := F(A \xrightarrow{h} A', Id_B): F(A,B) \to F(A',B)$$
\ \\

Here gain we need to check that $$ [\alpha(F)h]_{B'} \circ [\alpha(F)A](f) =[\alpha(F)A'](f) \circ [\alpha(F)h]_B$$ for 
$B \xrightarrow{f} B'$.\\

And this is true because we have : 
$$(h,I_{B'})\circ (I_A,f)=(h,f)=(I_{A'},f) \circ(h,I_B))$$
and by applying $F$ we get :
$$F(h,I_{B'})\circ F(I_A,f)  = F(h,f)= F(I_{A'},f) \circ F(h,I_B)),$$
that is $$ [\alpha(F)h]_{B'} \circ [\alpha(F)A](f) =[\alpha(F)A'](f) \circ [\alpha(F)h]_{B} .$$ 

\begin{rmk}
Given a functor $G$ in $\Hom(\A,  \Hom(\B,\E))$, we define $ \alpha^{-1}(G): \A \times \B \to\E $ by:\\ 
$G(A,B):=[GA]B$ on objects\\
$G(f,g):=[Gf]g$  for every  morphism $(f,g)$ of $\A \times \B$.\\
One can immediately check that this defines indeed a functor from $\A \times \B$ to $\E$. It's obvious that $\alpha^{-1}$ is a functor and for every every $F \in \Hom(\A \times \B,\E)$ we have an equality :
$$F= \alpha^{-1}(\alpha(F)).$$
\end{rmk}
\begin{note}
We will always write $\alpha$ and $\alpha^{-1}$ for the functor in the adjunction. 
\end{note}
\begin{proof}[\scshape{Proof of the lemma}] \ \\
\indent Let  $F : \C \times \D \to \E$ be an object of $\Hom_{\S \times \T }(\C \times \D,\E)$.\ \\
We want to show that $F$ factors in a unique way as : $F= \ol{F} \circ ( \L_{\S} \times \L_{\T})$, with $\ol{F}$ in $\Hom(\C[\S^{-1}] \times \D[\T^{-1}],\E)$.\ \\

\paragraph*{\textbf{Step 1}}  Consider $\alpha(F): \C \to \Hom(\D, \E)$ the functor given by the above adjunction. \\

Given $s: A \to A'$ a morphism of $\S$ and $U$ an object of $\D$, we have $(s,\Id_U) \in \S \times \T$ and by assumption $F(s,\Id_U)$ is invertible in $\E$.\ \\

By definition $\alpha(F)s$ is a natural transformation whose component at $U$ is exactly $F(s,\Id_U)$ :
$$[\alpha(F)s]_U := F(s, Id_U): F(A,U) \to F(A',U).$$ 
Then $\alpha(F)s$ is a natural isomorphism which means that $\alpha(F)$ makes $\S$ invertible, hence factors uniquely as 
$\alpha(F)= \ol{\alpha(F)} \circ  \L_{\S}$ with $\ol{\alpha(F)}: \C[\S^{-1}] \to  \Hom(\D,\E)$.\ \\

Now if we apply the inverse `$\alpha^{-1}$' to  both $\ol{\alpha(F)}$ and $\alpha(F)$ it's easy to see that we have the following equality :
$$F= \ol{F_0} \circ ( \L_{\S} \times \Id_{D})$$ 
where $\ol{F_0}=\alpha^{-1}(\ol{\alpha(F)})$ is a functor from $\C[\S^{-1}] \times \D$ to $\E$.\ \\

\paragraph*{\textbf{Step 2}}
It suffices to apply the \textbf{Step 1} to  $\ol{F_0} :\C[\S^{-1}] \times \D \to \E$  with $\S_0 \subseteq \C[\S^{-1}]$,    $\S_0:=\Id(\C[\S^{-1}])$, and interchanging the role of $\T$ and $\S_0$ using the symmetry of the cartesian product in Cat.\ \\
We then have a factorization :
$$\ol{F_0}= \ol{F_1} \circ ( \Id_{\C[\S^{-1}]} \times \L_{\T})$$
with $\ol{F_1} : \C[\S^{-1}] \times \D[\T^{-1}] \to \E $.\ \\

Combining this with the previous equality we have :
\begin{equation*}
\begin{split}
F & = \ol{F_0} \circ ( \L_{\S} \times \Id_{D}) \\
& =[\ol{F_1} \circ ( \Id_{\C[\S^{-1}])} \times \L_{\T}) ] \circ [ \L_{\S} \times \Id_{D}] \\
& = \ol{F_1} \circ [( \Id_{\C[\S^{-1}])} \circ \L_{\S})  \times ( \L_{\T} \circ \Id_{D})] \\
& = \ol{F_1} \circ (\L_{\S} \times \L_{\T})
\end{split}
\end{equation*}
Then $\ol{F}=\ol{F_1}$. 
\end{proof}

\begin{obs}\ \\
\renewcommand\labelenumi{\alph{enumi})}
\begin{enumerate}
\item A consequence of this lemma is that the canonical functor :
$$(\L_{\S} \times \L_{\T})^*(\E) : \Hom(\C[\S^{-1}] \times \D[\T^{-1}],\E) \xrightarrow{-\circ (\L_{\S} \times \L_{\T})} \Hom_{\S \times \T }(\C \times \D,\E) $$ is an isomorphism. 
\item  For the property `$\ol{F} \times \ol{G}= \ol{F \times G}$' it suffices to  write:
$$\ol{F \times G} \circ (\L_{\S} \times \L_{\T})= F \times G =( \ol{F} \circ \L_{\S}) \times (\ol{G} \circ \L_{\T})= (\ol{F} \times \ol{G}) \circ (\L_{\S} \times \L_{\T})$$ the unicity of the factorization forces the equality $\ol{F} \times \ol{G}= \ol{F \times G}$.
\end{enumerate}
\end{obs}

\section{Secondary Localization of a bicategory} \label{loc-bicat}
In this section we're going to define the Gabriel-Zisman locatization of $\M$ with respect to $\W$ when $(\M,\W)$ is a base of enrichment.

We will use the following notations.\\
$\M_{UV}=\M(U,V)$.\\
$\W_{UV}=\W \bigcap \M(U,V)$.\\
$ c_{\M}(U,V,W): \M_{VW} \times \M_{UV} \to \M_{UW}$= the composition functor in $\M$.\\
$\Hom(\C,\E)=$ category of functors from $\C$ to $\E$.\\
$\L_{\S}: \C \to \C[\S^{-1}]=$ a Gabriel-Zisman localization of $\C$ with respect to a class of maps $\S$.\\
${\L_{\S}}^{*}(\E) = \Hom(\C[\S^{-1}],\E) \xrightarrow{-\circ \L_{\S}} \Hom(\C,\E)$.\\ 
$\Hom_{\S}(\C,\E)=$the full subcategory of $\Hom(\C,\E)$ whose objects are functors which make $\S$ invertible \footnote{We say that $F:\C \to \D$ makes $\S$ invertible if for all $s \in \S$, $F(s)$ is invertible in $\E$.} in $\E$.\\

\begin{rmk}
From the assumptions made on $\W$, we clearly see that each $\W_{UV}$ is a subcategory of $\M_{UV}$ having the same objects. Moreover the functor $c_{\M}(U,V,W)$ sends $\W_{VW} \times \W_{UV}$ to $\W_{UW}$, so we can view $\W$ as a sub-bicategory $\M$, having the same objects and $1$-cells.
\end{rmk}
\begin{df}
Let $\B$ be a bicategory and $ \Phi : \M \to \B$ a homomorphism in the sense of Bénabou \cite{Ben2}. We will say that $\Phi$  makes $\W$ invertible if for every pair $(U,V)$ in $ob(\M)$, the functor $$ \Phi_{UV} : \M(U,V) \to \B(\Phi U,\Phi V)$$ makes $\W_{UV}$ invertible.
\end{df}

Our purpose is to construct a bicategory $\W^{-1}\M$ with a homomorphism $\L_{\W} : \M \to \W^{-1}\M$ which is ``universal" among those making $\W$ invertible. The universality here means that for any homomorphism $ \Phi : \M \to \B$ making $\W$ invertible we have a factorization, unique up-to a  transformation\footnote{the transformation is unique up to a \textbf{unique} modification},   $\Phi = \ol{\Phi} \circ \L_{\W}$, where\\
$\ol{\Phi}: \W^{-1}\M \to \B$ is a homomorphism.\

Like in the classical case the target bicategory $\W^{-1}\M$ should  (essentially) have the same object as $\M$,  so that   $\L_{\W}$ will be the identity on objects. Moreover if such localization homomorphism $\L_{\W}$vexists, we should have  factorizations of its components: 
$$ \L_{\W,UV}:\M_{UV} \xrightarrow{\L_{\W_{UV}}} \M_{UV}[\W_{UV}^{-1}] \xrightarrow{\ol{\L_{\W,UV}}} \W^{-1}\M(\L_{\W}U ,\L_{\W}V).$$\
This suggests to take $\M_{UV}[\W_{UV}^{-1}]$ as category of morphisms in $\W^{-1}\M$ for each $(U,V)$.
\begin{prop}
Let $(\M,\W)$ be a base of enrichment. There exists a bicategory  $\W^{-1}\M$ together with a homomorphism $\L_{\W} : \M \to \W^{-1}\M$ such that :
\renewcommand\labelenumi{\alph{enumi})}
\begin{enumerate}
\item $\L_{\W}$ makes $\W$ invertible,
\item any homomorphism $ \Phi : \M \to \B$  which makes $\W$ invertible factor as  $\Phi = \ol{\Phi} \circ \L_{W}$ with $$\ol{\Phi}: \W^{-1}\M \to \B$$ a homomorphism.
\item $\W^{-1}\M$ is unique up to a biequivalence \footnote{The biequivalence is itself unique up to a unique strong transformation.}. 
\end{enumerate}
\end{prop}
\begin{proof}[\scshape{Proof of the proposition}]\ \\
\indent Choose a localization $\L_{\W_{UV}}: \M_{UV} \to \M_{UV}[\W_{UV}^{-1}]$ for each pair $(U,V)$ of objects of $\M$.\ \\
Set $Ob(\W^{-1}\M) = Ob(\M)$,\\
$\W^{-1}\M(U,V)= \M_{UV}[\W_{UV}^{-1}]$.\ \\

\paragraph*{\textbf{For the composition}} By applying lemma \ref{loc-prod} for each triple $(U,V,W)$, we have a localization  
$$ \M_{VW} \times \M_{UV} \xrightarrow{\L_{\W_{VW}} \times \L_{\W_{UV}}} \M_{VW}[\W_{VW}^{-1}] \times \M_{UV}[\W_{UV}^{-1}]$$ of $\M_{VW} \times \M_{UV}$ with respect to $\W_{VW} \times \W_{UV}$.\ \\

Since $ c_{\M}(U,V,W): \M_{VW} \times \M_{UV} \to \M_{UW}$ sends $\W_{VW} \times \W_{UV}$ to $\W_{UW}$, it  follows that the composite $$\L_{\W_{UW}} \circ c_{\M}(U,V,W) :\M_{VW} \times \M_{UV} \to \M_{UW}[\W_{UW}^{-1}]$$ makes $\W_{VW} \times \W_{UV}$ invertible, hence factors as :
\[
\xy
(-20,0)*+{\M_{VW} \times \M_{UV}}="X";
(25,0)*+{\M_{UW}[\W_{UW}^{-1}]}="Y";
(-20,-20)*+{\M_{VW}[\W_{VW}^{-1}] \times \M_{UV}[\W_{UV}^{-1}]}="E";
{\ar@{->}^{\L_{\W_{UW}} \circ c_{\M}(U,V,W)}"X";"Y"};
{\ar@{->}_{\L_{\W_{VW}} \times \L_{\W_{UV}}}"X";"E"};
{\ar@{-->}_{~~~c_{\W^{-1}\M}(U,V,W)}"E";"Y"};
\endxy
\]
which gives the composition functor.\ \\ 

If we follow the notations of the factorization as in lemma \ref{loc-prod} we will write:
$$ \L_{\W_{UW}} \circ c_{\M}(U,V,W)=\ol{\L_{\W_{UW}} \circ c_{\M}(U,V,W)} \circ (\L_{\W_{VW}} \times \L_{\W_{UV}})$$ which means that $c_{\W^{-1}\M}(U,V,W) := \ol{\L_{\W_{UW}} \circ c_{\M}(U,V,W)}$.\ \\

\paragraph*{\textbf{For the associativity}} We build the following commutative diagram using the universal property of the Gabriel-Zisman localization and lemma \ref{loc-prod}. 

\[
\xy
(-60,20)*+{\M_{WZ} \times \M_{VW} \times \M_{UV}}="A";
(20,30)*+{\M_{WZ} \times \M_{UW}}="B";
(-20,-10)*+{\M_{VZ} \times \M_{UV}}="C";
(60,0)*+{\M_{UZ}}="D";
(-10,5)="m";
(10,15)="n";
{\ar@{=>}^{a(U,V,W,Z)~~}"m";"n"};
(-10,-45)="r";
(10,-35)="q";
{\ar@{=>}^{\ol{a(U,V,W,Z)}~~}"r";"q"};
{\ar@{->}^{\Id_{\M_{WZ}} \times c_{\M}(U,V,W)}"A";"B"};
{\ar@{->}_{c_{\M}(V,W,Z) \times \Id_{\M_{UV}}}"A";"C"};
{\ar@{->}^{c_{\M}(U,W,Z)}"B";"D"};
{\ar@{->}_{c_{\M}(U,V,Z)}"C";"D"};
(-60,-30)*+{\M_{WZ}[\W_{WZ}^{-1}] \times \M_{VW}[\W_{VW}^{-1}] \times \M_{UV}[\W_{UV}^{-1}]}="X";
(20,-20)*+{\M_{WZ}[\W_{WZ}^{-1}] \times \M_{UW}[\W_{UW}^{-1}]}="Y";
(-20,-60)*+{\M_{VZ}[\W_{VZ}^{-1}] \times \M_{UV}[\W_{UV}^{-1}]}="Z";
(60,-50)*+{\M_{UZ}[\W_{UZ}^{-1}]}="W";
{\ar@{-->}^{\gamma_1~~~~~~~~}"X";"Y"};
{\ar@{:>}_{\L_{\W_{WZ}} \times \L_{\W_{VW}} \times \L_{\W_{UV}}}"A";"X"};
{\ar@{->}^{c_{\W^{-1}\M}(U,W,Z)}"Y";"W"};
{\ar@{:>}^{}"C";"Z"};
{\ar@{:>}^{}"B";"Y"};
{\ar@{:>}^{\L_{\W_{UZ}}}"D";"W"};
{\ar@{-->}_{\gamma_2}"X";"Z"};
{\ar@{->}_{c_{\W^{-1}\M}(U,V,Z)}"Z";"W"};
{\ar@/_2.3pc/^{\eta_2}"A"; "D"};
{\ar@/^2.3pc/^{\eta_1}"A"; "D"};
{\ar@/_2.3pc/^{\sigma_2}"X"; "W"};
{\ar@/^2.3pc/^{\sigma_1}"X"; "W"};
\endxy
\]

\ \\

We use hereafter the same notations as in lemma \ref{loc-prod}. Then for every functor $F$ which factor through a localization $\L$, we will denote by $\ol{F}$ the unique functor such that $F= \ol{F} \circ \L$.

\begin{itemize}
\item The double dotted vertical maps are localizations given by lemma \ref{loc-prod}.
\item We've denoted for short $\eta_1=c_{\M}(U,W,Z)\circ [\Id_{\M_{WZ}} \times c_{\M}(U,V,W)].$\ \\

\item Similary $\eta_2=[ c_{\M}(V,W,Z)\times \Id_{\M_{UV}}] \circ c_{\M}(U,V,Z)$\ \\

\item $\gamma_1$ is by definition $\ol{\L_{\W_{WZ}} \times (\L_{\W_{UW}} \circ [\Id_{\M_{WZ}} \times c_{\M}(U,V,W)]}$ and is given  universal property with respect to  $\L_{\W_{WZ}} \times \L_{\W_{VW}} \times \L_{\W_{UV}}$.\ \\
 
\item $\gamma_2$ is  $\ol{(\L_{\W_{VZ}} \times \L_{\W_{UV}}) \circ [c_{\M}(V,W,Z) \times \Id_{\M_{UV}}]}$ \ \\

\item $\sigma_1= \ol{\L_{\W_{UZ}} \circ \eta_1}$ \ \\

\item $\sigma_2= \ol{\L_{\W_{UZ}} \circ \eta_2}$ \ \\
\item $\ol{a(U,V,W,Z)}$ is the inverse image of $a(U,V,W,Z) \otimes \Id_{\L_{\W_{UZ}}}$, which is an invertible the $2$-cell in Cat, by the isomorphism of categories $[\L_{\W_{WZ}} \times \L_{\W_{VW}} \times \L_{\W_{UV}}]^{\ast}(\M_{UZ}[\W_{UZ}^{-1}])$.\ \\

Recall that $[\L_{\W_{WZ}} \times \L_{\W_{VW}} \times \L_{\W_{UV}}]^{\ast}(\M_{UZ}[\W_{UZ}^{-1}])$ is an isomorphism from the hom-category  $$\Hom( \M_{WZ}[\W_{WZ}^{-1}] \times \M_{VW}[\W_{VW}^{-1}] \times \M_{UV}[\W_{UV}^{-1}],\M_{UZ}[\W_{UZ}^{-1}])$$ to the hom-category $$\Hom_{\W_{WZ} \times \W_{VW} \times \W_{UV}}(\M_{WZ} \times \M_{VW} \times \M_{UV},\M_{UZ}[\W_{UZ}^{-1}]).$$

It's clear that $\ol{a(U,V,W,Z)}$ is an invertible $2$-cell in Cat (a natural isomorphism) from $\sigma_2$ to $\sigma_1$
\end{itemize}
\ \\

\textbf{We need to show that the following hold}.
\begin{itemize}
\item $\gamma_1= Id_{\M_{WZ}[\W_{WZ}^{-1}]} \times c_{\W^{-1}\M}(U,V,W)$ \ \\

\item $\gamma_2= c_{\W^{-1}\M}(V,W,Z) \times Id_{\M_{UV}[\W_{UV}^{-1}]}$ \ \\

\item $\sigma_1= c_{\W^{-1}\M}(U,W,Z) \circ \gamma_1$ \ \\

\item $\sigma_2= c_{\W^{-1}\M}(U,V,Z) \circ \gamma_2$
\end{itemize}
\ \\
We proof the equality for $\gamma_1$ and $\sigma_1$, the argument is the same for the remaining cases. \ \\

For $\gamma_1$ we use the the property `$\ol{F \times G}= \ol{F} \times \ol{G}$' (see  Observations \ref{obs-loc-prod}). We have 
\begin{equation*}
\begin{split}
\gamma_1 & = \ol{\L_{\W_{WZ}} \times (\L_{\W_{UW}} \circ c_{\M}(U,V,W))}\\
         & =\ol{\L_{\W_{WZ}}} \times  \ol{\L_{\W_{UW}} \circ c_{\M}(U,V,W)} \\
\end{split}
\end{equation*}

But since $\L_{\W_{WZ}}= Id_{\M_{WZ}[\W_{WZ}^{-1}]} \circ  \L_{\W_{WZ}}$ then $\ol{\L_{\W_{WZ}}}=Id_{\M_{WZ[\W_{WZ}^{-1}]}}$.\ \\

Combining with the fact that $c_{\W^{-1}\M}(U,V,W) := \ol{\L_{\W_{UW}} \circ c_{\M}(U,V,W)}$, we deduce that 
$$\gamma_1= Id_{\M_{WZ}[\W_{WZ}^{-1}]} \times c_{\W^{-1}\M}(U,V,W)$$ as desired.\ \\

For $\sigma_1$ we're going to use the commutativity of the vertical faces in the `cubical' diagram and the fact that Cat is as strict $2$-category.\ \\

We write 
\begin{equation*}
\begin{split}
[c_{\W^{-1}\M}(U,W,Z) \circ \gamma_1] \circ (\L_{\W_{WZ}} \times \L_{\W_{VW}} \times \L_{\W_{UV}})& = c_{\W^{-1}\M}(U,W,Z)\circ [\gamma_1 \circ (\L_{\W_{WZ}} \times \L_{\W_{VW}} \times \L_{\W_{UV}})]\\
&=c_{\W^{-1}\M}(U,W,Z) \circ [ (\L_{\W_{WZ}} \times \L_{\W_{UV}}) \circ (\Id_{\M_{WZ}} \times c_{\M}(U,V,W))] \\
&=[c_{\W^{-1}\M}(U,W,Z) \circ  (\L_{\W_{WZ}} \times \L_{\W_{UV}})] \circ [\Id_{\M_{WZ}} \times c_{\M}(U,V,W)]\\
&=[\L_{\W_{UZ}} \circ c_{\M}(U,W,Z)] \circ [\Id_{\M_{WZ}} \times c_{\M}(U,V,W)]\\
&=\L_{\W_{UZ}} \circ [c_{\M}(U,W,Z) \circ (\Id_{\M_{WZ}} \times c_{\M}(U,V,W))]\\
&=\L_{\W_{UZ}} \circ \eta_1 \\
&=\sigma_1 \circ (\L_{\W_{WZ}} \times \L_{\W_{VW}} \times \L_{\W_{UV}}). 
\end{split}
\end{equation*}

The unicity of the factorization implies :
$\sigma_1= c_{\W^{-1}\M}(U,W,Z) \circ \gamma_1$.\ \\

\paragraph*{\textbf{For the axioms in $\W^{-1}\M$}} We give hereafter the argument for the associativity coherence. The argument is the same for the identity axioms.\ \\ 

The idea is to say that these axioms are satisfied in $\M$ and we need to check that they're transfered through the localization and this is true. The reason is that the property `$\ol{F \times G}= \ol{F} \times \ol{G}$' of functors hold also for natural transformations and commute with the composition. \ \\

For every objects T, U, V, W of $\M$, the pentagon of associativity from $\M_{WZ} \times \M_{VW} \times \M_{UV} \times \M_{TU}$  to $\M_{TZ}$ gives by composition with $\L_{\W_{TZ}}$ a commutative pentagon from 
$\M_{WZ} \times \M_{VW} \times \M_{UV} \times \M_{TU}$ to $\M_{TZ}[\W_{TZ}^{-1}]$. \ \\

For each vertex other than $\M_{TZ}[\W_{TZ}^{-1}]$, each `path' from the vertex to  $\M_{WZ}[\W_{WZ}^{-1}]$ factors through the suitable localization functor. \ \\

These factorizations fit together because we have
\begin{itemize}
\item a unicity of the factorization of the path from $\M_{WZ} \times \M_{VW} \times \M_{UV} \times \M_{TU}$ with respect to the localization $\L_{\W_{WZ}} \times \L_{\W_{VW}} \times \L_{\W_{UV}} \times \L_{\W_{TU}}$.
\item for every triple of objects, we have a cubical commutative diagram. 
\end{itemize}  
\ \\
We finally have a pentagon of associativity from $\M_{WZ}[\W_{WZ}^{-1}] \times \M_{VW}[\W_{VW}^{-1}] \times \M_{UV}[\W_{UV}^{-1}] \times \M_{TY}[\W_{TU}^{-1}]$ to $\M_{TZ}[\W_{TZ}^{-1}]$ as desired. \ \\

Finally one easily check that these data define a bicategory $\W^{-1}\M$ with a canonical homomorphism  $\L_{\W}: \M \to \W^{-1}\M$, and that  $\L_{\W}$ satisfies the universal property. 
\end{proof}

\nocite{*}

\bibliographystyle{plain}
\bibliography{biblio_these} 

\end{document}